\documentclass[12pt,reqno]{amsart}

\usepackage{latexsym}
\usepackage{amsmath}
\usepackage{amscd}
\usepackage{amssymb}
\usepackage{dsfont}

\usepackage{mathrsfs}
\usepackage{comment}
\usepackage{color}
\usepackage{enumerate}
\usepackage{soul}
\usepackage{framed}

\usepackage{graphicx}
\usepackage{subfigure}
\usepackage{url}

\newtheorem{theorem}{Theorem}[section]
\newtheorem{corollary}{Corollary}[theorem]

\newtheorem{proposition}[theorem]{Proposition}
\newtheorem{remark}[theorem]{Remark}

\numberwithin{equation}{section}

\begin{document}
\title{Noise and dissipation on coadjoint orbits}
\author[A. Arnaudon]{Alexis Arnaudon}
\author[A. L. De Castro]{Alex L. De Castro}
\author[D. D. Holm]{Darryl D. Holm}

\address{AA, DH: Department of Mathematics, Imperial College, London SW7 2AZ, UK}
\address{AC: Department of Mathematics, Imperial College, London SW7 2AZ, UK and Departamento de Matem\'atica PUC-Rio, Rio de Janeiro 22451-900.}

\maketitle

\begin{abstract}
	We derive and study stochastic dissipative dynamics on coadjoint orbits by incorporating noise and dissipation into mechanical systems arising from the theory of reduction by symmetry, including a semidirect-product extension.
	Random attractors are found for this general class of systems when the Lie algebra is semi-simple, provided the top Lyapunov exponent is positive.  
	We study two canonical examples, the free rigid body and the heavy top, whose stochastic integrable reductions are found and numerical simulations of their random attractors are shown. 
\end{abstract}

\setcounter{tocdepth}{1}
\tableofcontents

\section{Introduction}

Geometric mechanics, introduced in Poincar\'e \cite{poincare1901forme}, is a powerful formalism for understanding dynamical systems whose Lagrangian and Hamiltonian are invariant under the transformations of the configuration manifold $M$ by a Lie group $G$. Examples of its applications range from the simple finite dimensional dynamics of the freely rotating rigid body, to the infinite dimensional dynamics of the ideal fluid equations. For a historical review and basic references, see, e.g., \cite{holm1998euler}. See \cite{marsden1999intro,holm2008geometric,holm2009geometry} for textbook introductions to geometric mechanics and background references.
One of the main approaches of geometric mechanics is the method of reduction of the motion equations of a mechanical system by a Lie group symmetry $G$ in either its Lagrangian formulation on the tangent space $TM$, or its Hamiltonian formulation on the cotangent space $T^*M$. {This method yields symmetry-reduced Lagrangian and Hamiltonian formulations of the Euler-Poincar\'e equations governing the dynamics of the \emph{momentum map} $J: T^*M\to\mathfrak{g}^*$, where $\mathfrak{g}^*$ is the dual Lie algebra of the Lie symmetry group $G$.} 

In general terms, Lie group reduction by symmetry simplifies the motion equations of a mechanical system with symmetry by transforming them into new dynamical variables in $\mathfrak{g}^*$ which are invariant under the same Lie group symmetries as the Lagrangian and Hamiltonian for the dynamics of the mechanical system.  More specifically, on the Lagrangian side, the new invariant variables under the Lie symmetries are obtained from Noether's theorem, via the tangent lift of the infinitesimal action of the Lie symmetry group on the configuration manifold. The unreduced Euler--Lagrange equations are replaced by equivalent Euler-Poincar\'e equations expressed in the new invariant variables in $\mathfrak{g}^*$, plus an auxiliary reconstruction equation, which restores the information in the tangent space of the configuration space lost in transforming to group invariant dynamical variables. On the Hamiltonian side, after a Legendre transformation, equivalent new invariant variables in $\mathfrak{g}^*$ are defined by the momentum map $J:T^*M\to \mathfrak{g}^*$ from the phase space $T^*M$ of the original system on the configuration manifold $M$ to the dual $\mathfrak{g}^*$ of the Lie symmetry algebra $\mathfrak{g}\simeq T_eG$, via the cotangent lift of the infinitesimal action of the Lie symmetry group on the configuration manifold. The cotangent lift momentum map is an equivariant Poisson map which reformulates the canonical Hamiltonian flow equations in phase space as noncanonical Lie-Poisson equations governing flow of the momentum map on an orbit of the coadjoint action of the Lie symmetry group on the dual of its Lie algebra $\mathfrak{g}^*$, plus an auxiliary reconstruction equation for lifting the Lie group reduced coadjoint motion back to phase space $T^*M$. 

Thus, Lie symmetry reduction yields coadjoint motion of the corresponding momentum map. The dimension of the dynamical system reduces because its solutions are restricted to remain on certain subspaces of the original phase space, called coadjoint orbits. These are orbits of the action of the group $G$ on $\mathfrak{g}^*$, the dual space of its Lie algebra $\mathfrak{g}\simeq T_eG$. Coadjoint orbits lie on level sets of the distinguished smooth functions $C\in \mathcal{F}:\mathfrak{g}^*\to\mathbb{R}$ of the symmetry-reduced dual Lie algebra variables $\mu\in\mathfrak{g}^*$ called Casimir functions. Thus, the Casimir functions are conserved quantities. Indeed, Casimir functions have null Lie-Poisson brackets $\{C,F\}(\mu)=0$ with any other functions $F\in \mathcal{F}(\mathfrak{g}^*)$, including the reduced Hamiltonian $h(\mu)$. Furthermore, level sets of the Casimirs, on which the coadjoint orbits lie, are symplectic manifolds which provide the framework on which geometric mechanics is constructed. These symplectic manifolds have many applications in physics, as well as in symplectic geometry, whenever Lie symmetries are present. In particular, coadjoint motion of the momentum map $J(t)={\rm Ad}^*_{g(t)}J(0)$ for a solution curve $g(t)\in C(G)$ takes place on the intersections of level sets of the Casimirs with level sets of the Hamiltonian. 

Given this framework for Lie group reduction by symmetry in deterministic geometric mechanics, we seek strategies for adding stochasticity and dissipation in classical mechanical systems with symmetry which preserve the coadjoint motion structure of the unperturbed deterministic dynamics. Specifically, we seek stochastic coadjoint motion equations whose solutions dissipate energy, while they also lie on the coadjoint orbits of the unperturbed equation. 
Consequently, our first goal in this paper will be to replace the deterministic equations for coadjoint motion by stochastic processes whose solutions lie on coadjoint orbits. 
However, simply inserting additive noise into the deterministic equations will not, in general, produce coadjoint motion on level sets of the Casimirs of a Lie-Poisson bracket. 
Instead, our approach in developing a systematic derivation of {stochastic deformations that preserve coadjoint orbits} will be to constrain the variations in Hamilton's principle to preserve the transport relations for infinitesimal transformations defined by the action of a stochastic vector field on the configuration manifold. 

Having used the constrained Hamilton's principle to derive the stochastic coadjoint motion equation, the study of the associated Fokker-Planck equation and its invariant measure will follow naturally, and be well defined, at least provided one restricts to finite dimensional mechanical systems. 
The resulting Fokker-Planck equation defines a probability density for coadjoint motion on Casimir surfaces, since it takes the form of a Lie-Poisson equation for the transport part, and a double Lie-Poisson structure for the diffusion part, both of which generate motion along coadjoint orbits.  
As we will discover, this form of the Fokker-Planck equation in the absence of any additional energy dissipation will imply that the invariant measure (asymptotically in time) simply tends to a constant on Casimir surfaces. 

Next, we shall include an additional energy dissipation mechanism, called double bracket dissipation, or selective decay, which preserves the coadjoint orbits while it decays the energy toward its minimum value, usually associated with an equilibrium state of the deterministic system.
We refer to \cite{bloch1996euler,gaybalmaz2013selective, gaybalmaz2014geometric} and references therein for complete studies of double bracket dissipations. 
In a second step, we will include this double bracket dissipative term in our stochastic coadjoint motion equations and again study the associated Fokker-Planck equation and its invariant measure, which will no longer be a constant. Instead, the invariant measure will be an exponential function of the energy (Gibbs measure).  

The procedure we shall follow will produce stochastic dissipative dynamical systems on coadjoint orbits. 
The study of multiplicative noise and nonlinear dissipation in these systems is greatly facilitated by the geometric structure of the equations for coadjoint motion. 
Indeed, a large part of standard dynamical system theory will still apply in our setting. In particular, the proof of existence of random attractors will follow a standard approach. 
We will mainly focus on this particular feature of random attractors of our systems, as it is an important diagnostic and has recently been an active field of research. 
The main idea behind the random attractor is the decomposition of the invariant measure of the Fokker-Planck equation into random measures, called Sinai-Ruelle-Bowen, or SRB measures, whose expectation recovers the invariant measure of the Fokker-Planck equation. See, e.g., \cite{young2002srb} for a short insightful review. 
Random attractors can also help in understanding the notion of reliability in complex dynamical systems, see for example \cite{lu2013reliability}. These ideas have recently been developed and applied actively  in the context of climate studies. For example, see \cite{chekroun2011stochastic} for discussions and illustrations of how the notion of random SRB measures and random attractors enable the investigation of detailed geometric structures of the random attractors associated with nonlinear, stochastically perturbed systems. In particular, high-resolution numerical studies of two idealised models of interest for climate dynamics (the Jin97 ENSO model and the Lorenz63 model) are reported in \cite{chekroun2011stochastic}. The present work follows a similar line of investigation for a class of nonlinear, stochastically perturbed systems which exhibit coadjoint motion.
The proof of existence of non-singular SRB measures requires some work, but it can be accomplished for our general class of mechanical systems written on semi-simple Lie algebras.
Although geometric mechanics can also describe infinite dimensional systems such as fluid mechanics, \cite{holm1998euler}, we will only focus here on finite dimensional systems, and in particular on systems described by semi-simple Lie algebras. 
The natural non-degenerate and bi-invariant pairing admitted by semi-simple Lie algebras will facilitate the computations involved in proving our results, although some of the results may still apply more generally.

In the Euler-Poincar\'e theory, introducing a parameter into the Lagrangian or Hamiltonian which breaks the symmetry results in a semidirect product of groups acting on coset spaces representing the order parameters, or advected quantities, which are not invariant under the original symmetry group \cite{holm1998euler}. 
We will apply the theory of stochastic deformations that preserve coadjoint orbits for a particular class of semidirect product systems whose advected quantities live in the underlying vector space of the Lie algebra $\mathfrak g$.  
With this particular structure, which can be viewed as a generalisation of the heavy top, we will be able to prove the existence of SRB measures. 
Although much of the present theory may also apply for more general systems than we treat here, as a first investigation we will show that these particular mechanical systems in geometric mechanics exhibit interesting random attractors when both noise and a certain type of double bracket dissipation are included. 

As illustrations, we will discuss in detail two canonical elementary examples in the science of stochastic dissipative geometric mechanics. 
These two examples are the rigid body and the heavy top, which are also the well known canonical examples for understanding symmetry reduction for deterministic geometric mechanics, \cite{marsden1999intro,holm2008geometric,holm2009geometry}. As mentioned earlier, the extensions here to include stochasticity and dissipation which preserve coadjoint orbits may be regarded as natural counterparts for geometric mechanics of the standard nonlinear dissipative systems, such as the stochastic Lorenz systems, treated, e.g., in \cite{chekroun2011stochastic, kondrashov2015data}.

\subsection*{Main contributions of this work}
Section \ref{noise} uses the Clebsch approach of \cite{holm2015variational} to introduce noise into the Euler-Poincar\'e equation for the momentum map, including its extension for semidirect product Lie symmetry groups. By construction, the momentum map for the stochastic vector field is the same as that for the deterministic vector field, so the stochastic and deterministic Euler-Poincar\'e equations for the momentum map may be compared directly. 
Section \ref{dissipation} introduces the selective decay mechanism for dissipation and studies the existence of random attractors. 
The first example of the Euler-Poincar\'e equation is treated in Section \ref{RB} with the free rigid body. 
Section \ref{HT} treats the Heavy Top as an example of the semidirect product extension.
Finally, Section \ref{others} briefly sketches the treatments of two other examples, the $SO(4)$ free rigid body and the spring pendulum.

\section{Structure preserving stochastic mechanics}\label{noise}

Stochastic Hamilton equations were introduced along parallel lines with the deterministic canonical theory in \cite{bismut1982mecanique}. 
These results were later extended to include reduction by symmetry in \cite{lazaro2008stochastic}.   
Reduction by symmetry of expected-value stochastic variational principles for Euler-Poincar\'e equations was developed in \cite{arnaudon2014stochastic,chen2015constrained}.
Stochastic variational principles were also used in constructing stochastic variational integrators in \cite{bourabee2009stochastic}.

The present work is based on recent work of \cite{holm2015variational}, which used variational principles to introduce noise in fluid dynamics. This variational approach was developed further for fluids with advected quantities in \cite{gaybalmaz2016geometric}.
The inclusion of noise in fluid equations has a long history in the scientific literature. For reviews and recent advances in stochastic turbulence models, see \cite{kraichnan1994anomalous}, \cite{gawedzki1996university}; and in the analysis of stochastic Navier-Stokes equations, see \cite{mikulevicius2001equations}. 
These studies of the stochastic Navier-Stokes equation are fundamental in the analysis of fluid turbulence.  
Expected-value stochastic variational principles leading to the derivation of the Navier-Stokes motion equation for incompressible viscous fluids have been investigated in \cite{arnaudon2012lagrangian}. For further references, we refer to \cite{holm2015variational}.

The present section incorporates stochasticity into finite dimensional mechanical systems admitting Lie group symmetry reduction, by using the standard Clebsch variational method for deriving cotangent lifted momentum maps. 
We review the standard approach to Lie group reduction by symmetry for finite dimensional systems in \ref{EP-section} and incorporate noise into this approach in section \ref{EP-noise-section}.
Next, we describe the semidirect extension in \ref{SM-section} and study the associated Fokker-Planck equations and invariant measures in \ref{FP-section}. The primary examples from classical mechanics with symmetry will be the free rigid body and the heavy top under constant gravity. 

\subsection{Euler-Poincar\'e reduction}\label{EP-section}
Classical mechanical systems with symmetry can often be understood geometrically in the context of Lagrangian or Hamiltonian reduction, by lifting the motion $q(t)$ on the configuration manifold $Q$ to a Lie symmetry group via the action of the symmetry group $G$ on the configuration manifold, by setting $q(t)=g(t)q(0)$, where the multiplication is to be understood as the action of $G$ to $Q$. 
This procedure lifts the solution of the motion equation from a curve $q(t)\in Q$ to a  curve $g(t)\in G$, see \cite{marsden1999intro, holm2008geometric}. 
The simplest case is when $Q= G$. This case, called Euler-Poincar\'e reduction, will be described in the present section. 

In the Lagrangian framework, reduction by symmetry may be implemented in Hamilton's principle via restricted variations in the reduced variational principle arising from variations on the corresponding Lie group. 
In the standard approach, for an arbitrary variation $\delta g$ of a curve $g(t)\in G$ in a Lie group $G$, the left-invariant reduced variables are $g^{-1}\dot g\in \mathfrak g$ in the Lie algebra $\mathfrak g= T_eG$. Their variations arise from variations on the Lie group and are given by 
\[\delta \xi = \dot \eta + \mathrm{ad}_\eta\xi\,,\] 
for $\eta:= g^{-1}\delta g$. Here, the operation $\mathrm{ad}:\mathfrak g\times \mathfrak g\to \mathfrak g$ represents the adjoint action of the Lie algebra on itself via the Lie bracket, denoted equivalently as $\mathrm{ad}_\xi \eta= [\xi,\eta]$, and we will freely use either notation throughout the text. 
If the Lagrangian $L(g,\dot g)$ is left-invariant under the action of $G$, the restricted variations $\delta \xi$ of the reduced Lagrangian $L(e,g^{-1}\dot g)=:l(\xi)$ inherited from admissible variations of the solution curves on the group yield the Euler-Poincar\'e equation
\begin{align}
	\frac{d}{dt}\frac{\partial l(\xi)}{\partial \xi} + \mathrm{ad}^*_\xi \frac{\partial l(\xi)}{\partial \xi}=0
	\,.
	\label{EP}
\end{align}
In this equation, ${\rm ad}^*: \mathfrak g\times \mathfrak g^*\to\mathfrak g^*$ is the dual of the adjoint Lie algebra action, ad. That is, $\langle{\rm ad}^*_\xi\mu,\eta \rangle=\langle{\mu,\rm ad}_\xi\eta \rangle$ for $\mu\in\mathfrak{g}^*$ and $\xi,\eta\in\mathfrak{g}$, under the nondegenerate pairing $\langle\,\cdot\,,\, \cdot \,\rangle: \mathfrak g\times \mathfrak g^*\to\mathbb{R}$. 
Throughout this paper, we will restrict ourselves to semi-simple Lie algebras, so that the pairing is given by the Killing form, defined as 
\begin{align}
	\kappa(\xi,\eta) := \mathrm{Tr}\left ( \mathrm{ad}_\xi\mathrm{ad}_\eta\right ).
	\label{Killing}
\end{align}
In terms of the structure constants of the Lie algebra denoted as $c_{jk}^i$ for a basis $e_i,\,i=1,\dots,{\rm dim}(\mathfrak{g})$, so that $[e_i,e_j]=c_{jk}^ie_i$, in which $\xi=\xi^ie_i$ and $\eta=\eta^je_j$, the Killing form takes the  explicit form
\begin{align*}
	\mathrm{Tr}(\mathrm{ad}_\xi\mathrm{ad}_\eta)= c_{im}^nc_{jn}^m \xi^i \eta^j.
\end{align*}
An important property of this pairing is its bi-invariance, written as  
\begin{align}
	\kappa (\xi, \mathrm{ad}_\zeta \eta) = \kappa (\mathrm{ad}_\xi \zeta,\eta),
	\label{bi-invariance}
\end{align}
for every $\xi,\eta,\zeta \in \mathfrak{g}$.
This pairing allows us to identify the Lie algebra with its dual, as the Killing form semi-simple Lie algebras is non-degenerate.
We will also use the property for compact Lie algebras, that the Killing form is negative definite and thus induces a norm on the Lie algebra, $\|\xi\|^2:=-\kappa(\xi,\xi)$. 
This function turns out to always be an invariant function on the coadjoint orbit, for every semi-simple Lie algebra; that is, it is a Casimir function.
Indeed, an invariant function is in the kernel of the Lie-Poisson bracket $\{F,C\}(\mu) := \kappa\left ( \mu, \left [\frac{\partial F}{\partial \mu},\frac{\partial C}{\partial \mu}\right ]\right )$, $\forall F\in C(\mathfrak g^*, \mathbb R)$. 
For $C(\mu)= \frac12 \kappa(\mu,\mu)$ it is straightforward to check using the bi-invariance \eqref{bi-invariance} that this is true for any function $F$. 
In general, it is difficult to find other independent Casimir functions of semi-simple Lie algebra; see \cite{perelomov1068casimir,thiffeault2000classification}.
Of course, the theory of semi-simple Lie algebras is standard and well developed, see for example \cite{varadarajan1984lie}. However, for the sake of clarity, we will express the abstract notations of adjoint and coadjoint actions with respect to the Killing form. 
We may then identify $\mathfrak g^* \cong \mathfrak g$ for each semi-simple Lie algebra we treat here. 

We now turn to the equivalent Clebsch formulation of the Euler-Poincar\'e equations via a constrained Haamilton's principle, which we will use for implementing the noise in these systems. 
The Clebsch formulation of the Euler-Poincar\'e equation and its corresponding Lie-Poisson bracket on the Hamiltonian side has been explored extensively in ideal fluid dynamics \cite{holm1983poisson, marsden1983coadjoint} and more recently in optimal control problems \cite{gay2011clebsch} and stochastic fluid dynamics \cite{holm2015variational}. 
This earlier work should be consulted for detailed derivations of Clebsch formulations of Euler-Poincar\'e equations in the contexts of ideal fluids and optimal control problems. We will briefly sketch the Clebsch approach, as specialised to the applications treated here; since we will rely on it for the introduction of noise in finite-dimensional mechanical systems by following the approach of \cite{holm2015variational} for stochastic fluid dynamics. 
We first introduce the Clebsch variables $q\in \mathfrak g$ and $p\in \mathfrak g^*$, where $p$ will be a Lagrange multiplier which enforces the dynamical evolution of $q$ given by the Lie algebra action of $\xi\in \mathfrak g$, as $\dot q+ \mathrm{ad}_\xi q= 0$.
Note the similarity of this equation with the constrained variations of the Lagrangian reduction theory. 
The Clebsch method in fluid dynamics (resp. optimal control) introduces auxiliary equations for advected quantities (resp. Lie algebra actions on state variables) as constraints in the Hamilton (resp. Hamilton-Pontryagin) variational principle $\delta S = 0$ with constrained action
\begin{align}
	S(\xi,q,p) =\int  l(\xi) dt + \int \langle p, \dot{q} + \mathrm{ad}_\xi q \rangle  dt\,.
	\label{Clebsch-action}
\end{align}
Taking free variations of $S$ with respect to $\xi,q$ and $p$ yields a set of equations for these three variables which can be shown to be equivalent to the Euler-Poincar\'e equation \eqref{EP}. 
The relation between the Lie algebra vector $\xi\in \mathfrak g$ and the phase-space variables $(q,p)\in T_e^*G$ is given by the variation of the action $S$ with respect to the velocity $\xi$ in \eqref{Clebsch-action}. This variation yields the momentum map, $\mu:T_e^*G \to \mathfrak g^*$, given explicitly by 
\begin{align}
	\mu := \frac{\partial l(\xi)}{\partial \xi} = \mathrm{ad}^*_qp.
\end{align}
Unless specified otherwise, we will always use the notation $\mu$ for the conjugate variable to $\xi$. 
This version of the Clebsch theory is especially simple, as the Clebsch variables are also in the Lie algebra $\mathfrak g$. 
In general, it is enough for them to be in the cotangent bundle of a manifold $T^*M$ on which the group $G$ acts by cotangent lifts. 
In this more general case, the adjoint and coadjoint actions must be replaced by their corresponding actions on $T^*M$ but the method remains the same. 
Another generalisation, which will be useful for us later, allows the Lagrangian to depend on both $\xi$ and $q$. 
In this case, the Euler-Poincar\'e equation will acquire additional terms depending on $q$ and the Clebsch approach will be equivalent to semidirect product reduction \cite{holm1998euler}. 
We will consider a simple case of this extension in Section \ref{SM-section} and in the treatment of the heavy top in Section \ref{HT}.

\subsection{Structure preserving stochastic deformations}\label{EP-noise-section}

We are now ready to deform the Euler-Poincar\'e equation \eqref{EP} by introducing noise in the constrained Clebsch variational principle \eqref{Clebsch-action}. 
In order to do this stochastic deformation, we introduce $n$ independent Wiener processes $W_t^i$ indexed by $i=1,2,\dots,n,$ and their associated stochastic potential fields $\Phi_i(q,p) \in \mathbb R$ which are prescribed functions of the Clebsch phase-space variables, $(q,p)$.
{ The stochastic processes used here are standard Weiner processes, as discussed, e.g., in \cite{chen2015constrained, ikeda2014stochastic}. }
The number of stochastic processes can be arbitrary, but usually we will assume it is equal to the dimension of the Lie algebra, $n={\rm dim}(\mathfrak g)$. 
The constrained stochastic variational principle is then given by 
\begin{align}
	S(\xi,q,p) =\int  l(\xi) dt + \int \langle p,\circ  dq + \mathrm{ad}_\xi q\, dt \rangle   + \int \sum_{i=1}^n \Phi_i(q,p)\circ dW^i_t.
		\label{Sto-Clebsch-action}
\end{align}
In the stochastic action integral \eqref{Sto-Clebsch-action} and hereafter, the multiplication symbol $\circ$ denotes a stochastic integral in the Stratonovich sense. 
The Stratonovich formulation is the only choice of stochastic integral that admits the classical rules of calculus (e.g., integration by parts, the change of variables formula, etc.). Therefore, the Stratonovich formulation is indispensable in variational calculus and in optimal control. The free variations of the action functional \eqref{Sto-Clebsch-action} may now be taken, and they will yield stochastic processes for the three variables $\xi,q$ and $p$. 

For convenience in the next step of deriving a stochastic Euler-Poincar\'e equation, we will assume that the Lagrangian $l(\xi)$ in the action \eqref{Sto-Clebsch-action}  is hyperregular, so that $\xi$ may be obtained from the fibre derivative $\frac{\partial l(\xi)}{\partial \xi} = \mathrm{ad}^*_qp$. 
We will also specify that the stochastic potentials $\Phi_i(q,p)$ should depend only on the momentum map $\mu= \mathrm{ad}^*_q p$ so that the resulting stochastic equation will be independent of $q$ and $p$.  
Following the detailed calculations in \cite{holm2015variational}, we then find the stochastic Euler-Poincar\'e equation
\begin{align}
	d \frac{\partial l(\xi)}{\partial \xi} + \mathrm{ad}^*_\xi \frac{\partial l(\xi)}{\partial \xi} dt 
	- \sum_i \mathrm{ad}^*_\frac{\partial \Phi_i(\mu)}{\partial \mu}  \frac{\partial l(\xi)}{\partial \xi}\circ dW_t^i=0\,.
	\label{SEP}
\end{align}
In terms of the stochastic process 
\begin{align}
dX= \xi dt - \sum_i \frac{\partial \Phi_i(\mu)}{\partial \mu} \circ dW_t^i
\,,\quad\hbox{with}\quad
\mu=\frac{\partial l(\xi)}{\partial \xi},
	\label{dX-def}
\end{align} 
the stochastic Euler-Poincar\'e equation \eqref{SEP} may be expressed in compact form, as
\begin{align}
	d\mu+ \mathrm{ad}^*_{dX} \mu=0\,.
	\label{SEP-simple}
\end{align}
The introduction of noise in the Clebsch-constrained variational principle rather than using reduction theory provides a transparent approach for dealing with stochastic processes on Lie groups and constrained variations arising for such processes. 
In this approach, the momentum map plays the same central role in both the deterministic and stochastic formulations. 
See \cite{arnaudon2014stochastic} for a different approach, resulting in the derivation and analysis of deterministic expectation-value Euler-Poincar\'e equations using reduction by symmetry with conditional expectation. 

\begin{remark}[Reduction with noise]
	\label{reduction-noise}
	The stochastic Euler-Poincar\'e equation \eqref{SEP} arises from a stochastic reduction by symmetry, as follows. First, the reconstruction relation $\dot g= g \xi$ in the deterministic case has its stochastic counterpart
	\begin{align}
		dg= g \xi dt + \sum_i  g\sigma_i \circ dW_t^i\, ,
		\label{dg-eqn}
	\end{align}
where $\sigma_i:= -\,\frac{\partial \Phi_i(\mu)}{\partial \mu} $, and the expressions $g \xi $  and $g \sigma_i $ are understood as the tangent of the left action of the group on itself; or equivalently, the left action of the group on its Lie algebra. 
	Equation \eqref{SEP} then results from taking the variation of $g^{-1} dg$ with \eqref{dg-eqn} and comparing with the derivative of $g^{-1}\delta g$ while setting $d(\delta g)= \delta (dg)$.  
	This gives the variation $\delta\xi$ as
	\begin{align}
		\delta\xi = d\eta + \mathrm{ad}_{(g^{-1} dg)}\eta
		= d\eta + \mathrm{ad}_\xi\eta\,  dt + \sum_i \mathrm{ad}_{\sigma_i}\eta  \circ dW^i_t,
		\label{delta-xi-eqn}
	\end{align}	 
	where $ \eta = g^{-1}\delta g\in \mathfrak g$ is arbitrary, except that $\delta g(0)=0=\delta g(T)$ at the endpoints in time $t\in[0,T]$.
	Then, using these constrained variations in the reduced variational principle $\delta \int l(\xi)dt = 0$ yields the stochastic Euler-Poincar\'e equation \eqref{SEP}, by the following calculation, 
	\begin{align}\begin{split}
		0 & = \delta \int l(\xi)dt = \int \left\langle \frac{\delta l}{\delta \xi} \,,\, \delta\xi \right\rangle dt
		    = \int \left\langle \frac{\delta l}{\delta \xi} \,,\, d\eta + \mathrm{ad}_{(g^{-1} dg)}\eta \right\rangle dt 
	        \\& = \int \left\langle -\,d\frac{\delta l}{\delta \xi} + \mathrm{ad}^*_{(g^{-1} dg)} \frac{\delta l}{\delta \xi}
		\,,\, \eta \right\rangle dt + \left\langle \frac{\delta l}{\delta \xi} \,,\, \eta \right\rangle \bigg|_0^T
	\end{split}	
	\label{StochEP-eqn}
	\end{align}	 
upon imposing 	the condition that $\eta$ vanishes at the endpoints in time.
\end{remark}
As in the deterministic case, various generalisations of this theory are possible. 
For example, as mentioned earlier, the Clebsch phase-space variables can also be defined in $T^*M$, and the Lagrangian can depend on $q$ for systems of  semidirect product type \cite{gaybalmaz2016geometric}.
Another generalisation is to let the stochastic potentials $\Phi_i(\mu)$ also depend separately on $q$ in the semidirect product setting, as we will see later. 

After having defined the Stratonovich stochastic process \eqref{SEP}, one may compute its corresponding It\^o form, which is readily given in terms of the $ \mathrm{ad}^*$ operation by 
\begin{align}
	\begin{split}
	d \frac{\partial l(\xi)}{\partial \xi} + \mathrm{ad}^*_\xi \frac{\partial l(\xi)}{\partial \xi} dt &+ \sum_i \mathrm{ad}^*_{\sigma_i}  \frac{\partial l(\xi)}{\partial \xi}dW_t^i - \frac12 \sum_i \mathrm{ad}^*_{\sigma_i}\left ( \mathrm{ad}^*_{\sigma_i}\frac{\partial l(\xi)}{\partial \xi}\right ) dt=0,
	\end{split}
	\label{SEP-simple-ito}
\end{align}
where $\sigma_i:= -\,\frac{\partial \Phi_i(\mu)}{\partial \mu} $. 
Notice that the indices for $\sigma_i$ in the It\^o sum in \eqref{SEP-simple-ito} are the same, and the $\sigma_i$ may be taken as a basis of the underlying vector space. 
In terms of $\mu:= \frac{\partial l(\xi)}{\partial \xi} $ the It\^o stochastic Euler-Poincar\'e equation \eqref{SEP-simple-ito} may be expressed equivalently as 
\begin{align}
	d \mu + \mathrm{ad}^*_\xi \mu \mbox{ }dt + \sum_i \mathrm{ad}^*_{\sigma_i}  \mu  \mbox{ }dW_t^i 
	- \frac12 \sum_i \mathrm{ad}^*_{\sigma_i}\left (\mathrm{ad}^*_{\sigma_i}\mu\right ) dt=0\,.
	\label{SEP-simple-ito-mu}
\end{align}

Another formulation of the stochastic Euler-Poincar\'e equation in \eqref{SEP} which will be used later in deriving the Fokker-Planck equation is the stochastic Lie-Poisson equation
\begin{align}
	d f(\mu) &= \left \langle \mu, \left[ \frac{\partial f}{\partial \mu},  \frac{\partial h}{\partial \mu}\right ] \right \rangle dt + \sum_i\left \langle \mu, \left[ \frac{\partial f}{\partial \mu},  \frac{\partial \Phi_i}{\partial \mu}\right ] \right \rangle \circ dW_t^i\\
	&=: \{f,h\} dt +\sum_i \{ f,\Phi_i\} \circ dW_i\,,
	\label{LP-sto}
\end{align}
where the Lie-Poisson bracket $\{\cdot ,\cdot \}$ is defined just as in the deterministic case, from the adjoint action and the pairing on the Lie algebra $\mathfrak g$. 

\subsection{The extension to semidirect product systems}\label{SM-section}
{
As discussed in \cite{holm1998euler}, ``It turns out that semidirect products occur under rather general
circumstances when the symmetry in $T^*G$ is broken''. The geometric mechanism for this remarkable fact is that under reduction by symmetry, a semidirect product of groups emerges whenever the symmetry in the phase space is broken. The symmetry breaking produces new dynamical variables, living in the coset space formed from taking the quotient $G/G_a$ of the original unbroken symmetry $G$ by the remaining symmetry $G_a$ under the isotropy subgroup of the new variables. These new dynamical variables form a vector space $G/G_a\simeq V$ on which the unbroken symmetry acts as a semidirect product, $G\,\circledS\,V$. In physics, elements of the vector space $V$ corresponding to the new variables are called ``order parameters''. Typically, in physics, the original symmetry is broken by the introduction of potential energy depending on variables which reduce the symmetry to the isotropy subgroup of the new variables. Dynamics on the semidirect product $G\,\circledS\,V$ results, and what had previously been flow under the action of the unbroken symmetry now becomes flow plus waves, or oscillations, produced by the exchange of energy between its kinetic and potential forms. 
The heavy top is the basic example of this phenomenon, and it will be treated in Section \ref{HT}. The semidirect product motion for the heavy top arises in the presence of gravity, when the support point of a freely rotating rigid body is shifted away from its centre of mass. 

With this connection between symmetry breaking and semidirect products in mind, we} now extend the stochastic Euler-Poincar\'e equations to include semidirect product systems. We refer to \cite{holm1998euler} for a complete study of these systems.
Although the deterministic equations of motion in \cite{holm1998euler} are derived from reduction by symmetry, we will instead incorporate noise by simply extending the Clebsch-constrained variational principle used in the previous section.

The generalisation proceeds, as follows. We will begin by assuming that the Clebsch phase-space variables comprise the elements of $T^*V$ for a given vector space $V$ on which the Lie group $G$ acts freely and properly. 
In fact we will have $(q,p)\in V\times V^*$.
However, in this work, we will restrict ourselves to the case where $V$ is the underlying vector space of $\mathfrak g$. 
Following the notation of \cite{ratiu1981euler}, we denote $\overline{ \mathfrak g}= V$ in the sequel. 
Then, from the Killing form on $\mathfrak g$, denoted by $\kappa:\mathfrak g\times \mathfrak g \to \mathbb  R $, there is a bi-invariant extension of the Killing form on $\mathfrak g\circledS \overline {\mathfrak g }$ defined as 
\begin{align}
	\kappa_s \left ( ( \xi_1, \xi_2) , (\eta_1, \eta_2) \right ) := \kappa( \xi_1,\eta_2) + \kappa( \xi_2, \eta_1).
	\label{kappaS}
\end{align}
Although this pairing is non-degenerate and bi-invariant, we will not use it for the definition of the dual of the semidirect algebra $\mathfrak g\circledS V$. Instead, we will use the sum of both Killing forms, namely 
\begin{align}
	\kappa_0 \left ( ( \xi_1, \xi_2) , (\eta_1, \eta_2) \right ) := \kappa( \xi_1,\eta_1) + \kappa( \xi_2, \eta_2).
	\label{kappa0}
\end{align}
The group action is defined via the adjoint representation of $G$ on $V=\overline{\mathfrak g}$, given by  $(g_1,\eta_1)(g_2,\eta_2) = (g_1g_2, \eta_1  + \mathrm{Ad}_{g_1}\eta_2)$. 
We then directly have  the infinitesimal adjoint and coadjoint actions as
\begin{align}
\begin{split}
	\mathrm{ad}_{(\xi_1,q_1)} (\xi_2,q_2) &= ( \mathrm{ad}_{\xi_1}\xi_2 , \mathrm{ad}_{\xi_1}q_2+\mathrm{ad}_{q_1}\xi_2)\,,\\
	\mathrm{ad}^*_{(\xi,q)} (\mu,p) &= ( \mathrm{ad}^*_\xi\mu + \mathrm{ad}^*_qp, \mathrm{ad}^*_\xi p),
	\label{SD-prod}
\end{split}
\end{align}
where the coadjoint action is taken with respect to $\kappa_0$ in \eqref{kappa0}. 

The extended Killing form $\kappa_s$ defined in \eqref{kappaS}, gives, apart from $\kappa(\eta,\eta)$ with $\eta\in \overline {\mathfrak g}$, a second invariant function on the coadjoint orbit
\begin{align*}
	\kappa_s\left ( (\xi,\eta), (\xi,\eta)\right ) = 2\kappa(\xi, \eta),
\end{align*}
found from the bi-invariance of the Killing form $\kappa_s$. 
One then replaces the corresponding Lie algebra actions in the Clebsch-constrained variational principle \eqref{Sto-Clebsch-action}, to obtain the stochastic process with semidirect product
\begin{align}
	\begin{split}
	d \left ( \mu, q \right ) &+ \mathrm{ad}^*_{(\xi,r)}\left (  \mu , q\right ) dt  + \sum_i \mathrm{ad}^*_{\left ( \frac{\partial \Phi_i(\mu,q)}{\partial \mu},\frac{\partial \Phi_i(\mu,q)}{\partial q}\right) } \left (\mu, q \right )\circ dW_t^i=0,
	\end{split}
	\label{SDE-SM}
\end{align}
where $l:\mathfrak g\circledS V\to \mathbb R $, $\Phi_i:\mathfrak g^*\circledS V^*\to \mathbb R$ and 
\begin{align}
	\frac{\partial l(\xi,q)}{\partial \xi}=: \mu
	\qquad \mathrm{and}\qquad 
	 \frac{\partial l(\xi,q)}{\partial q}=: r\,.
\end{align}
Consequently, after taking the Legendre transform of $l$, we have the Hamiltonian derivatives
\begin{align}
	\frac{\partial h(\mu,q)}{\partial \mu}=: \xi
	\qquad \mathrm{and}\qquad 
	 \frac{\partial h(\mu,q)}{\partial q}=: -\,r\,,
	 \label{hr-defs} 
\end{align}
for $h:\mathfrak g^*\circledS V^*\to \mathbb R$.
By substituting into \eqref{SDE-SM} the expressions in \eqref{SD-prod} for the coadjoint action, we obtain the system  
\begin{align}
\begin{split}
	d \mu &+ \left ( \mathrm{ad}^*_\xi \mu  + \mathrm{ad}^*_r q\right ) dt  + \sum_i \left( \mathrm{ad}^*_\frac{\partial \Phi_i(\mu,q)}{\partial \mu} \mu  + \mathrm{ad}^*_\frac{\partial \Phi_i(\mu,q)}{\partial q} q \right )   \circ dW_t^i  =0
	\,,\\
	d q & +  \mathrm{ad}^*_\xi q  dt  + \sum_i  \mathrm{ad}^*_\frac{\partial \Phi_i(\mu,q)}{\partial \mu}q \circ dW_t^i=0\,.
	\label{SDE-SM-system}
\end{split}
\end{align}
Although the number of stochastic potentials $\Phi_i$ which one may consider is arbitrary, we shall find it convenient for our purposes to restrict to a maximum of $n= \mathrm{dim}(\mathfrak g) + \mathrm{dim}(V)$ such potentials. 
In fact, the potentials associated with $V$ will not actually be fully treated here. 

The semidirect product theory we have described here is the simplest instance of it, as we are using a particular vector space $V$.
In general, the advected quantities can also be in a Lie algebra, or an arbitrary manifold, provided the action of the group $G$ on it is free and proper \cite{gaybalmaz2016geometric}.

\subsection{The Fokker-Planck equation and invariant distributions}\label{FP-section}

We derive here a geometric version of the classical Fokker-Planck equation (or forward Kolmogorov equation) using our SDE \eqref{SEP}. 
Recall that the Fokker-Planck equation describes the time evolution of the probability distribution $\mathbb{P}$ for the process driven by \eqref{SEP}. 
We refer to \cite{ikeda2014stochastic} for the standard textbook treatment of stochastic processes. 
Here, we will consider $\mathbb P$ as a function $C(\mathfrak g^*)$ with the additional property that $\int_\mathfrak{g^*} \mathbb P d\mu= 1$. 
First, the generator of the process \eqref{SEP} can be readily found from the Lie-Poisson form \eqref{SEP-simple-ito}  of the stochastic process \eqref{LP-sto} to be 
\begin{align}
	Lf(\mu) = \left \langle \mathrm{ad}^*_\xi\mu ,\frac{\partial f}{\partial \mu}\right \rangle - \sum_i\left \langle \mathrm{ad}^*_{\sigma_i}\mu,\frac{\partial}{\partial \mu} \left \langle \mathrm{ad}^*_{\sigma_i}\mu ,\frac{\partial f}{\partial \mu}\right \rangle\right \rangle,
	\label{FP-gen}
\end{align}
where $\langle\,\cdot \,,\, \cdot \,\rangle$ still denotes the Killing form on the Lie algebra $\mathfrak g$ and $f\in C(\mathfrak g^*$) is an arbitrary function of $\mu$. 
Provided that the  $\Phi_i$'s are linear functions of the momentum $\mu$, the diffusion terms of the infinitesimal generator $L$ will be self-adjoint with respect to the $L^2$ pairing $\int_\mathfrak{g^*} f(\mu)\mathbb P(\mu)d\mu$. 
If the  $\Phi_i$'s are not linear, the advection terms of $L^*$ will contain other terms. However, since we will restrict our considerations to the case of linear stochastic potentials, mainly for practical reasons, we will refer to \eqref{FP-gen} and its analogues as the Fokker-Planck operator $L^*$. 

The Fokker-Planck equation describes the dynamics of the probability distribution $\mathbb P$ associated to the stochastic process for $\mu$, in the standard advection diffusion form. 
Another step can be taken to highlight the underlying geometry of the Fokker-Planck equation \eqref{FP-gen}, by rewriting it in terms of the Lie-Poisson bracket structure as
\begin{align}
	\frac{d}{dt} \mathbb P + \{ h,\mathbb P\} -  \sum_i\{\Phi_i, \{\Phi_i,\mathbb P\}\}=0 \,,
	\label{HamFP}
\end{align}
where $h(\mu)$ is the Hamiltonian associated to $l(\xi)$ by the Legendre transform. 
In \eqref{HamFP}, we recover the Lie-Poisson formulation \eqref{LP-sto} of the Euler-Poincar\'e equation together with a dissipative term arising from the noise of the original SDE in a double Lie-Poisson bracket form. 

This formulation gives the following theorem for invariant distributions of \eqref{FP-gen}:
\begin{theorem}\label{limit-thm}
	The invariant distribution $\mathbb P_\infty$ of the Fokker-Planck equation \eqref{FP-gen}, i.e, $L^*\mathbb P_\infty=0$ is \emph{uniform} on the coadjoint orbits on which the SDE \eqref{SEP} evolves. 
\end{theorem}
\begin{proof}
By a standard result in functional analysis, see for example \cite{villani2009hypocoercive}, a linear differential operator of the form $L = B  + \sum_i A_i^2$ has the property that $\mathrm{ker}(L) = \mathrm{ker}(A_i)\cap \mathrm{ker}(B)$, where here $A_i = \{\Phi_i, \cdot \}$ and $B = \{h, \cdot \}$.
Consequently, for every smooth function $f$, the only functions $g$ which satisfy $\{ f,g \}= 0$ are the Casimirs, or invariant functions, on the coadjoint orbits. 
When restricted to a coadjoint orbit, these functions become constants.
Hence, the invariant distribution $\mathbb P_\infty$ is a \emph{constant} on the coadjoint orbit identified by the initial conditions of the system.
\end{proof}

Since the dynamics is restricted to the coadjoint orbits, for the probability distribution $\mathbb P$ to tend to a constant, yet remain normalisable, satisfying $\int_{\mathfrak g^*} \mathbb P(\mu)d\mu= 1$, the value of the density must tend to the inverse of the volume of the coadjoint orbit. 
Of course the compactness of the coadjoint orbit is equivalent to $\mathbb P_\infty>0$. 
For non-compact orbits, Theorem \ref{limit-thm} is still valid, and it will imply an asymptotically vanishing invariant distribution, in the same sense as for the invariant solution of the heat equation on the real line. 
In this case, a more detailed analysis of the invariant distribution can be performed by studying \emph{marginals}, or projections onto a compact subspace of the coadjoint orbit. 

Examples of non-compact coadjoint orbits arise in the semidirect product setting. 
First, the Fokker-Planck equation for the semidirect product stochastic process \eqref{SDE-SM} is given by
\begin{align}
	\begin{split}
	Lf(\mu,q) &= \left \langle \mathrm{ad}^*_{(\xi,r)} (\mu,q) ,\left ( \frac{\partial f(\mu,q)}{\partial \mu}, \frac{\partial f(\mu,q)}{\partial q}\right ) \right \rangle - \\
	&- \sum_i\left \langle \mathrm{ad}^*_{(\sigma_i,\eta_i)}(\mu,q), \left \{ \frac{\partial}{\partial \mu} \left \langle \mathrm{ad}^*_{(\sigma_i,\eta_i)}(\mu,q) ,\left ( \frac{\partial f(\mu,q)}{\partial \mu},\frac{\partial f(\mu,q)}{\partial q}\right )\right \rangle\right.\right ., \\
	&\hspace{4.2cm}, \left . \left .  \frac{\partial}{\partial q} \left \langle \mathrm{ad}^*_{(\sigma_i,\eta_i)}(\mu,q) ,\left ( \frac{\partial f(\mu,q)}{\partial \mu},\frac{\partial f(\mu,q)}{\partial q}\right )\right \rangle\right \}\right \rangle, 
	\end{split}
	\label{FP-SM}
\end{align}
where $\sigma_i :=- {\partial \Phi_i}/{\partial \mu}$ and $\eta_i := -{\partial \Phi_i}/{\partial q}$. 
The pairing used here is the sum of the pairings on $\mathfrak g$ and on $V$, given by $\kappa_0$ in \eqref{kappa0}. 
Note that for some values of index $i$, the vector fields $\sigma_i$ or $\eta_i$ may be absent. 
One can check that $L^*= L$; so that $L$ generates the Lie-Poisson Fokker-Planck equation for the probability density $\mathbb P(\mu,q)$. 
As before, upon using the explicit form of the coadjoint actions, one finds
\begin{align}
\begin{split}
	Lf(\mu,q) &= \left \langle \mathrm{ad}^*_\xi \mu + \mathrm{ad}^*_q r , \frac{\partial f}{\partial \mu} \right \rangle + \left\langle  \mathrm{ad}^*_\xi q ,\frac{\partial f}{\partial q}\right \rangle -\\
	&- \sum_i\left \langle \mathrm{ad}^*_{\sigma_i} \mu + \mathrm{ad}^*_q \eta_i , \frac{\partial A_i}{\partial \mu} \right \rangle -\sum_i \left\langle \mathrm{ad}^*_{\sigma_i} q  ,\frac{\partial A_i}{\partial q}\right \rangle, \\
	\noindent \mathrm{where}\qquad A_i:&= \left \langle \mathrm{ad}^*_{\sigma_i} \mu + \mathrm{ad}^*_q \eta_i , \frac{\partial f}{\partial \mu} \right \rangle + \left\langle \mathrm{ad}^*_{ \sigma_i} q ,\frac{\partial f}{\partial q}\right \rangle\, .
\end{split}
\end{align}
The Fokker-Planck equation \eqref{FP-SM} provides a direct corollary of Theorem \ref{limit-thm}.
\begin{corollary}
	The invariant probability density $\mathbb{P}_\infty(\mu,q)$ of \eqref{FP-SM} is constant on the coadjoint orbit corresponding to the initial conditions of the stochastic process \eqref{SDE-SM}. 
\end{corollary}

As mentioned earlier, the coadjoint orbit of this system is not compact, even if it had been compact for the Lie algebra $\mathfrak g$. 
Nevertheless, we can study the marginal distributions 
\begin{align}
	\mathbb P^1(\mu) &:= \int \mathbb P(\mu,q)dq\qquad \mathrm{ and}\label{P1}\\
	\mathbb P^2(q) &:= \int \mathbb P(\mu,q)d\mu\, , \label{P2}
\end{align}
which of course extend to invariant marginal distributions $\mathbb P^1_\infty$ and $\mathbb P^2_\infty$.
With these marginal distributions, we can get more information about the invariant distribution of the semidirect product Lie-Poisson  Fokker-Planck equation \eqref{FP-SM}, as summarized in the next theorem.
\begin{theorem}\label{SD-invariant}
	For a semi-simple Lie algebra $\mathfrak g$ and $V= \overline{\mathfrak g}$, the marginal invariant distributions defined in \eqref{P1} and \eqref{P2} of the Fokker-Planck equation \eqref{FP-SM}, with $\eta_i=0$, for all $i$, have the following forms.  
	\begin{enumerate}
		\item The invariant distribution $\mathbb P^2_\infty(q)$ is constant on the $q$-dependent subspace of the coadjoint orbit. 
			If the Lie algebra $\mathfrak g$ is non-compact, the constant is zero. 
		\item The invariant distribution $\mathbb P^1_\infty(\mu)$ restricted to $\kappa(\mu,\mu)$ is constant. 
		\item If $\mathfrak g$ is compact, the invariant distribution $\mathbb P^1_\infty(\mu)$  is linearly bounded in time in the direction perpendicular to $\kappa(\mu,\mu)$.
	\end{enumerate}
\end{theorem}

\begin{proof}
	We will compute the invariant marginal distributions separately, but first recall that the invariant distribution $\mathbb P(\mu,q)$ is constant on the Casimir level sets given by the initial conditions.

$(1)$
By integrating the Fokker-Planck equation \eqref{FP-SM} over $\mu$, one obtains 
\begin{align}
	L\mathbb P^2(q) &=  \int \left\langle  \mathrm{ad}^*_\xi q  ,\frac{\partial \mathbb P(\mu,q)}{\partial q}\right \rangle d\mu - \left\langle  \mathrm{ad}^*_{ \sigma_i} q ,\frac{\partial }{\partial q}
	\left\langle   \mathrm{ad}^*_{\sigma_i} q ,\frac{\partial \mathbb P^2(q)}{\partial q}\right \rangle  \right \rangle,
\label{FP-marginSDP}
\end{align}
where we have used the property that the coadjoint action is divergence-free (because of the anti-symmetry of the adjoint action, when identified with the coadjoint action via the Killing form) and have recalled that the Lie algebra is either compact, or $\mathbb P(\mu,q)= 0$ for the boundary conditions. 

Only the advection term remains in \eqref{FP-marginSDP}, as $\xi=\frac{\partial h}{\partial \mu}$ depends on $\mu$. 
Nevertheless, an argument similar to that for the proof of Theorem \ref{limit-thm} may be applied here to give the result of constant marginal distribution on the $q$ dependent part of the coadjoint orbits.
Again, if the Lie algebra is non-compact, then the probability density $\mathbb P^2_\infty(q)$ must vanish because of the normalisation. 

$(2)$
We first integrate the Fokker-Planck equation \eqref{FP-SM} with respect to the $q$ variable to find
\begin{align}
	L\mathbb P^1(\mu) &= \left \langle \mathrm{ad}^*_\xi \mu , \frac{\partial \mathbb P^1}{\partial \mu} \right \rangle 
	-  \sum_i\left \langle \mathrm{ad}^*_{\sigma_i} \mu  , \frac{\partial}{\partial \mu} \left \langle \mathrm{ad}^*_{\sigma_i} \mu  , \frac{\partial \mathbb P^1}{\partial \mu} \right \rangle \right \rangle,
\end{align}
where we have again used that the coadjoint action is divergence free, the same boundary conditions and the fact that $\langle \mathrm{ad}_q \xi, \frac{\partial \mathbb P}{\partial \mu}\rangle = 0,\: \forall \xi$ since $\frac{\partial \mathbb P}{\partial \mu}$ is aligned with $q$. 
Indeed, $\mathbb P$ is a function of the Casimirs, and thus is a function of $\kappa_s((\mu,q),(\mu,q))$. 
This fact prevents us from directly invoking Theorem \ref{limit-thm} as we would find that $\mathbb P^1$ is indeed constant on $\kappa(\mu,\mu)$, but $\mu$ does not have an invariant norm. 
Nevertheless, we can still use this theorem by restricting  $\mathbb P^1$ to the sphere $\kappa(\mu,\mu)$, or equivalently simply considering polar coordinates for $\mu$ and discarding the radial coordinate. 
In this case, we can invoke Theorem \ref{limit-thm} and obtain the result of a constant marginal distribution $\mathbb P^1_\infty$ projected on the coadjoint orbit of the Lie algebra $\mathfrak g$ alone. 

$(3)$
We compute the time derivative of the quantity $\|\mu\|^2:= -\,\kappa(\mu,\mu)$, which is positive definite and thus defines a norm, to get an upper estimate of the form
\begin{align*}
	\frac{d}{dt}\frac12 \|\mu\|^2 &= \langle \mu, \dot \mu\rangle = \langle \mathrm{ad}_r q, \mu\rangle \leq \|r\| \|q\|\|\mu\|.
\end{align*}
Then, because $\|q\|=\sqrt{-\kappa(q,q)}$ is constant, and provided that $r$ is bounded, we can integrate to find 
\begin{align}
	\|\mu(t)\| \leq \|\mu(0)\| +\alpha t, 
\end{align}
where $\alpha$ is a constant depending on the Lie algebra and the Hamiltonian.
\end{proof}

\begin{remark}[On ergodicity]
	An important question about any given dynamical system is whether its solution is ergodic. 
	This question needs some clarification for the systems considered here. 
	First, notice that the deterministic systems are not ergodic, as they are Hamiltonian systems with extra conserved quantities given by the Casimirs. 
	Some of the systems are even completely integrable, a property which is usually understood as the opposite of ergodicity. 
	Now, if the noise is switched on in the case where the $\sigma_i$ span the entire Lie algebra, then there is a constant invariant measure on the level set of the Casimir given by the initial conditions. 
	This means that we have the ergodicity property on the coadjoint orbits but not on the full Euclidian space in which the coadjoint orbits are embedded. 
	The ergodicity must then only be defined with respect to the coadjoint orbit, otherwise the system will not be seen to be ergodic. 
	Finally, the cases where the $\sigma_i$ do not span the Lie algebra must be treated individually, depending on the system in hand. 
	For example with the rigid body in section \ref{RB}, having two independent non-trivial $\sigma_i$ is sufficient for ergodicity, while having only one sigma will make the system integrable, and thus non-ergodic. 
\end{remark}

\paragraph{\bf Summary.} This section has reviewed the framework for the study of noise in dynamical systems defined on coadjoint orbits, and has illustrated how noise may be included in these systems, so as to preserve the deterministic coadjoint orbits.
This preservation property is seen clearly in the Clebsch formulation, because the deterministic and stochastic systems share the same momentum map, whose level sets define the coadjoint orbits. The systems we have considered are the Euler-Poincar\'e equations on semi-simple finite dimensional Lie groups and the semidirect product structures which appear when the advected quantities are introduced in the underlying vector space of the Lie algebra of the Lie group. 
These structures are not the most general. However, their study has allowed us to use the properties of the natural pairing given by the Killing form to prove a few illustrative results in a simple and transparent way. 
In particular, we showed that the invariant measure of the Fokker-Planck equation, written in Lie-Poisson form, is constant on the coadjoint orbits. 
In the semidirect product setting, a bit more care was needed to obtain similar results for the marginal distributions, as the coadjoint orbits are not compact in this case. 
We will illustrate our approach with the two basic examples of the rigid body and heavy top in sections \ref{RB} and \ref{HT}, where more will be said about these systems, and in particular about their integrability. 

\section{Dissipation and invariant measures}\label{dissipation}
In the previous section we described a structure preserving stochastic deformation of mechanical systems with symmetries. 
The preserved structure is the coadjoint orbit of the deterministic system. Namely, the stochastic process still belongs to one of these orbits, characterised by the initial conditions of the system. 
This preservation is reflected in the strict conservation of particular integrals of motion, called Casimirs. 
In general, these are the only conserved quantities of our stochastic processes. 
Indeed, the energy is not conserved, apart from very particular choices of the energy and the stochastic potentials as we will see for some examples. 
The energy is not strictly decaying either, but is subject to random fluctuations with its own stochastic process coupled to the stochastic process of $\mu$. 
The complexity of the energy evolution has hindered us from studying it in full generality in the previous sections. 
In the present section, however, we will investigate the energy behaviour for particular mechanical examples subject to dissipation and random fluctuations. 
The type of energy dissipation that we will introduce in Section \ref{dissipation-section} also preserves the coadjoint orbits. Consequently, the dissipation is compatible with our stochastic deformation. 
The main outcome after introducing this dissipation is the emergence of a balance between noise and dissipation which will make the invariant measure of the Fokker-Planck equation energy dependent, as we will see in Section \ref{FP-diss-section} and in the proof of existence of random attractors in Section \ref{RA-section}. 

\subsection{Double bracket dissipation}\label{dissipation-section}

To augment the stochastic processes introduced in the previous section, we will add a type of dissipation for which the solutions of the stochastic process will still lie on the deterministic coadjoint orbit.
For this purpose, we will use \emph{double bracket dissipation}, which was studied in detail in \cite{bloch1996euler} and was generalised recently in \cite{gaybalmaz2013selective, gaybalmaz2014geometric}.
We will follow the latter works to include an energy dissipation which preserves the Casimir functions. 
We will not review this theory in detail here. Instead, we refer the reader to \cite{gaybalmaz2013selective} for a detailed discussion of Euler-Poincar\'e selective decay dissipation and \cite{gaybalmaz2014geometric} for the semidirect product extension. 

For the stochastic process \eqref{SEP-simple}, the dissipative stochastic Euler-Poincar\'e equation written in Hamiltonian form is 
\begin{align}		
	d\mu &+ \mathrm{ad}^*_\frac{\partial h}{\partial \mu} \mu\, dt 
	+ \theta\, \mathrm{ad}^*_\frac{\partial C}{\partial \mu} \left [ \frac{\partial C}{\partial \mu}, \frac{\partial h}{\partial \mu} \right ]^\flat dt + \sum_i\mathrm{ad}^*_{\sigma_i} \mu \circ dW_t^i = 0 \,,
	\label{SEP-Diss}
\end{align}
where $\theta>0$ parametrises the rate of energy dissipation and $C$ is a chosen Casimir of the coadjoint orbit. 
For convenience, we are using the isomorphism $\flat:\mathfrak g\to \mathfrak g^*$ defined via the Killing form of $\mathfrak g$.
The converse isomorphism will be denoted $\sharp:\mathfrak g^*\to \mathfrak g$. 
The corresponding generalisation of selective decay for the semidirect product stochastic process \eqref{SDE-SM}, following \cite{gaybalmaz2014geometric}, is given by
\begin{align}
	\begin{split}
	d(\mu,q) + \mathrm{ad}^*_{(\xi,r)} ( \mu,q)\,  dt 
	&+ \theta\, \mathrm{ad}^*_{\left ( \frac{\partial C}{\partial \mu},\frac{\partial C}{\partial q}\right )} \left [ \left ( \frac{\partial C}{\partial \mu},\frac{\partial C}{\partial q}\right ) ,  ( \xi,r) \right ]^\flat dt \\
	&\hspace{3cm} + \sum_i\mathrm{ad}^*_{(\sigma_i,\eta_i)} (\mu,q) \circ dW_t^i = 0,
	\end{split}
	\label{SD-SD-eqns}
\end{align}
where $\xi = \frac{\partial h}{\partial \mu}$, and the quantities $h$ and $r$ are defined in equation \eqref{hr-defs}.
Equation \eqref{SD-SD-eqns} may be written equivalently as a system of equations, by using the actions given in \eqref{SD-prod}. Namely, 
\begin{align}
	\begin{split}
		d\mu + (\mathrm{ad}^*_\xi  \mu +\mathrm{ad}^*_r q)\, dt 
		+ \theta\, \mathrm{ad}^*_{\frac{\partial C}{\partial \mu}} \left [  \frac{\partial C}{\partial \mu}, \xi \right ]^\flat &dt 
		\,+\, \theta\,  \mathrm{ad}^*_\frac{\partial C}{\partial q} \left (\mathrm{ad}_\frac{\partial C}{\partial \mu} r + \mathrm{ad}_\frac{\partial C}{\partial q} \xi  \right )^\flat  dt \\
	 + & \sum_i \left (\mathrm{ad}^*_{\sigma_i}\mu +\mathrm{ad}^*_{\eta_i}q\right ) \circ dW_t^i = 0,\\
	dq + \mathrm{ad}^*_\xi q\, dt + \theta\, \mathrm{ad}^*_\frac{\partial C}{\partial \mu}   \left ( \mathrm{ad}_\frac{\partial C}{\partial \mu} r- \mathrm{ad}_\frac{\partial C}{\partial q} \xi  \right )^\flat dt  	
	&\,+ \sum_i \mathrm{ad}^*_{\sigma_i}q  \circ dW_t^i = 0\,.
	\end{split}
	\label{SP-SD}
\end{align}

Recall for the deterministic equations that the energy decays for $\theta>0$ as
\begin{align}
\frac{d}{dt}h(\mu,q) = - \,\theta \left \| \mathrm{ad}_\frac{\partial C}{\partial \mu}\xi \right \|^2 - \theta \, 	\left \| \mathrm{ad}_\frac{\partial C}{\partial \mu} r+ \mathrm{ad}_\frac{\partial C}{\partial q}  \xi \right \|^2,
	\label{SD-diss-SD}
\end{align}
where the second term is present only in the semidirect product setting \cite{gaybalmaz2014geometric}.

\begin{remark}[Choice of Casimir]
The selective decay approach preserves the entire coadjoint orbit, and the speed of decay depends upon which invariant function $C$ one uses in implementing it. 
Indeed, either the first or second term of \eqref{SD-diss-SD} can vanish depending on the choice of Casimir. 
We refer to the heavy top example in Section \ref{HT} for more details. 
\end{remark}

\begin{remark}[Variational principle and reducion]
	The reader may have noticed already that although we introduced noise via variational principles, the dissipation is added as an extra term in the equations of motions. 
	In fact, it also fits a variational principle, but as it is a dissipative force, it requires another type of variational principle, the so-called Lagrange-d'Alembert principle. Here, we only refer to \cite{gaybalmaz2014geometric} for the exact formulation of the variational principle in this context. 
	If one goes through the computation of the variations in the variational principle with noise, as described in remark \ref{reduction-noise}, then an extra noise term will appear in \eqref{SEP-Diss}, given by 
	\begin{align*}
		\sum _i \theta\, \mathrm{ad}^*_\frac{\partial C}{\partial \mu} \left [ \frac{\partial C}{\partial \mu}, \frac{\partial \Phi_i}{\partial \mu} \right ] ^\flat \circ dW_t^i.
	\end{align*}
	This term would be admissible, since it preserves the coadjoint orbit. However, we do not include it here. Instead, we leave it for treatment elsewhere, as it complicates the calculations to follow without significantly affecting the solution behaviour; since it is proportional to $\theta\sigma^2$, and $\theta$ and $\sigma^2$ are taken as being small compared to the original dynamics, so that the noise and dissipation are viewed as perturbations. 
	The product of the two factors therefore makes this term negligible in our study here.  
\end{remark}

Asymptotically in time, $t\to\infty$, the deterministic equations with selective decay will tend toward a state which is compatible with the state of minimal energy, as shown in \cite{gaybalmaz2014geometric}. 
However, the presence of noise will balance the dissipation due to selective decay and prevent the system from reaching this deterministic equilibrium position. 
This feature will imply a non-constant invariant distribution of the corresponding Fokker-Planck solution to be studied in the next section, as well as the existence of \emph{random attractors}, for which we refer to \cite{kondrashov2015data,schenk1998random} for background information.  

\subsection{The Fokker-Planck equation and invariant distributions}\label{FP-diss-section}
In order to study the balance between multiplicative noise and nonlinear dissipation, we compute the Fokker-Planck equation associated to the process \eqref{SEP-Diss} or, equivalently,  \eqref{SP-SD}, and its invariant solutions. 

The Fokker-Planck equation for the Euler-Poincar\'e stochastic process \eqref{SEP-Diss} is modified by the double bracket dissipative term, to read as,
\begin{align}
	\frac{d}{dt} \mathbb P(\mu) + \{h,\mathbb P\}  +\theta\left \langle\left [\frac{\partial \mathbb P}{\partial \mu}, \frac{\partial C}{\partial \mu}\right], \left [ \frac{\partial h}{\partial \mu}, \frac{\partial C}{\partial \mu}\right]\right \rangle - \frac12 \sum_i  \{\Phi_i,\{\Phi_i,\mathbb P\}\}=0.
	\label{FP-Diss}
\end{align}
The invariant distribution of this Fokker-Planck equation is no longer a constant on the coadjoint orbits. Instead, it now depends on the energy, as summarized in the following theorem. 
\begin{theorem}\label{FP-diss-thm}
	Let the noise amplitude be of the form $\sigma_i= \sigma e_i$ for an arbitrary $\sigma \in \mathbb R$, where the $e_i$'s span the underlying vector space of the dual Lie algebra $\mathfrak g^*\cong \mathfrak g$. 
	The invariant distribution of the Fokker-Planck equation \eqref{FP-Diss} associated to \eqref{SEP-Diss} with Casimir $C= \kappa(\mu,\mu)$ is given by 
	\begin{align}
		\mathbb P_\infty(\mu) = Z^{-1} e^{-\frac{2\theta}{\sigma^2} h(\mu)},  
	\label{MaxwellianDist}	
	\end{align}
	where $Z$ is the normalisation constant that enforces $\int \mathbb P_\infty(\mu) d\mu= 1$. 
\end{theorem}

\begin{proof}
	The invariant distribution is given by solving $\frac{d}{dt} \mathbb P_\infty(\mu)= 0$. From the structure of the Fokker-Planck equation in double bracket form \eqref{FP-Diss}, the advection term must vanish independently of the other terms. (See the argument of Theorem \ref{limit-thm}.) 
	We therefore use the Ansatz $\mathbb P_\infty(\mu) = f(h(\mu))$, where the function $f$ is to be determined. 
Consequently, only the selective decay and the double bracket term still remain. 
The selective decay is first rewritten, using the bi-invariance property of the Killing form \eqref{bi-invariance}, as 
\begin{align*}
	\theta\left \langle\left [\frac{\partial \mathbb P}{\partial \mu}, \frac{\partial C}{\partial \mu}\right], \left [ \frac{\partial h}{\partial \mu}, \frac{\partial C}{\partial \mu}\right]\right \rangle &= 
\theta\left \langle\frac{\partial \mathbb P}{\partial \mu}, \mathrm{ad}_\frac{\partial C}{\partial \mu} \left [ \frac{\partial C}{\partial \mu}, \frac{\partial h}{\partial \mu}\right] \right \rangle \\
&= \theta\, \mathbf d\left( f(h) \mathrm{ad}_\frac{\partial C}{\partial \mu} \left [ \frac{\partial C}{\partial \mu}, \frac{\partial h}{\partial \mu}\right]\right ),
\end{align*}
where we have used the property that the coadjoint action for semi-simple Lie algebras is divergence-free. 
(Notice that the exterior derivative $\mathbf d$ is a divergence operation here.)
Since $\kappa(\mu,\mu)$ is a Casimir and $\mu^\sharp=\frac{\partial C}{\partial \mu}$, we can rewrite the double bracket due to the noise as 
\begin{align*}
	-\frac12 \sum_i  \{\Phi_i,\{\Phi_i,\mathbb P\}\} &= -\sigma^2\frac12 \sum_i\left \langle \mathrm{ad}^*_{e_i}\frac{\partial C}{\partial \mu}^\flat , \frac{\partial }{\partial \mu} \left \langle \mathrm{ad}^*_{e_i}\frac{\partial C}{\partial \mu}^\flat, \frac{\partial \mathbb P}{\partial \mu} \right \rangle\right \rangle\\
	\hbox{(From bi-invariance of $\kappa$)}\quad
	&= \sigma^2\frac12 \sum_i \mathbf d \left (f'(h)\mathrm{ad}_{e_i}\frac{\partial C}{\partial \mu} \left \langle \mathrm{ad}_\frac{\partial h}{\partial \mu}\frac{\partial C}{\partial \mu}, e_i \right \rangle\right)\\
	&= \sigma^2\frac12  \mathbf d \left (f'(h)\mathrm{ad}_{\mathrm{ad}_\frac{\partial h}{\partial \mu}\frac{\partial C}{\partial \mu}}\frac{\partial C}{\partial \mu}  \right).
\end{align*}
We have used the bi-invariance of the pairing to enforce the relation $\mathrm{ad}^\dagger_\xi\eta:=\mathrm{ad}^*_\xi \eta^\flat= - \mathrm{ad}_\xi\eta$. See for example \cite{varadarajan1984lie} for more details.
The result \eqref{MaxwellianDist} for the equilibrium distribution then follows by comparing the selective decay term with the double bracket term and noticing that the two terms will cancel, provided $f(x) = e^{-{2\theta\,x}/{\sigma^2}}$. 
\end{proof}

The Fokker-Planck equation with dissipation in the semidirect-product setting directly gives
\begin{align}
	\begin{split}
	\frac{d}{dt} \mathbb P(\mu,q) &+ \{h,\mathbb P\} - \frac12 \sum_i  \{\Phi_i,\{\Phi_i,\mathbb P\}\}\, +\\
	&+\theta\, \left \langle\left [ \left (\frac{\partial \mathbb P}{\partial \mu}, \frac{\partial \mathbb P}{\partial q}\right ), \left ( \frac{\partial C}{\partial \mu}, \frac{\partial C}{\partial q}\right ) \right], \left [ \left ( \frac{\partial h}{\partial \mu},\frac{\partial h}{\partial q}\right ),\left (  \frac{\partial C}{\partial \mu},\frac{\partial C}{\partial q}\right )\right]\right \rangle=0 \,.
	\end{split}
	\label{SD-SD-FP}
\end{align}
Consequently, for semidirect products, we have the analogue of the previous theorem, but for the marginal invariant distribution on the advected quantities. 
\begin{theorem}
	Provided the Hamiltonian is of the form $h(\mu,q)= K(\mu)+V(q)$ for two functions $K$ and $V$, the invariant marginal distribution $\mathbb P_\infty^2(q)$ with the selective decay from the Casimir $\kappa(\mu,q)$ is given by
	\begin{align}
		\mathbb P_\infty^2(q)= Z^{-1} e^{-\frac{2\theta}{\sigma^2} V(q)},
	\end{align}
	where $Z$ is the normalisation constant.
\end{theorem}
\begin{proof}
The proof here is similar to the proof for Theorem \ref{FP-diss-thm}. Thus, we only show the main calculations. 	
First, the selective decay term is given explicitly, using the Casimir $\kappa(\mu,q)$, by
\begin{align*}
\theta\, \left \langle\mathrm{ad}_\frac{\partial \mathbb P}{\partial \mu} q , \mathrm{ad}_\xi q \right \rangle+  
\theta\, \left \langle \mathrm{ad}_\frac{\partial \mathbb P}{\partial \mu}\mu + \mathrm{ad}_\frac{\partial \mathbb P}{\partial q}q , 
\mathrm{ad}_\xi \mu+ \mathrm{ad}_rq  \right \rangle. 
\end{align*}
Integrating the selective decay term of \eqref{SD-SD-FP} in $\mu$ and assuming $\mathbb P^2(q)= f(V(q))$, gives 
\begin{align*}
	\theta\, \left \langle   \mathrm{ad}_\frac{\partial \mathbb P^2}{\partial q} q ,  \mathrm{ad}_r q \right \rangle &= - \theta \left \langle   \mathrm{ad}^*_{\mathrm{ad}_r q}  q,  \frac{\partial \mathbb P^2}{\partial q} \right \rangle = 	- \theta\mathbf d \left (  \mathrm{ad}^*_{\mathrm{ad}_r q}  q f  \right ),
\end{align*}
where we have used the bi-invariance property of $\kappa$ \eqref{bi-invariance}, as well as the divergence-free property of the coadjoint action. 
Then, after integration over $\mu$, the double bracket term becomes
\begin{align*}
	-\frac12 \sigma^2 \mathbf d \left (  \mathrm{ad}^*_{e_i} q  \left\langle  \mathrm{ad}^*_{e_i}q ,\frac{\partial \mathbb P^2}{\partial q}\right\rangle   \right) &=  \frac12 \sigma^2 \mathbf d \left (f'  \mathrm{ad}^*_{e_i}q \left\langle  \mathrm{ad}_q r,  e_i \right\rangle   \right) \\
&= - \frac12 \sigma^2 \mathbf d \left (f'  \mathrm{ad}^*_{\mathrm{ad}_rq  }q    \right),
\end{align*}
upon again using bi-invariance. 
Thus, the result follows, as $f$ must satisfy $\theta f =\frac12 \sigma^2 f'$.   
\end{proof}

In the Euler-Poincar\'e setting, the invariant distribution was concentrated around the positions of minimum energy, and here the advected quantity $q$ is concentrated around the position of minimal \emph{potential} energy. 
We conjecture that the complete invariant distribution is concentrated around the minimal energy region, as in the Euler-Poincar\'e setting. However, we will not investigate this conjecture here, as we will be mainly interested in the dynamics of the advected quantities. 
\begin{remark}[Gibbs measure]
	This calculation only uses the bi-invariance of the Killing form, which holds in general for semi-simple Lie algebras. Therefore, the same conclusion applies for other Lie algebras which admit a bi-invariant pairing. 
	In statistical physics, the invariant measure \eqref{MaxwellianDist} is often called a Gibbs measure. This association provides a natural identification of the quantity $\sigma^2/ 2k_B\theta$ with a Kelvin temperature $T$, where $k_B$ is the Boltzmann constant. 
	This notion of temperature arises via coupling the system with a heat bath, at temperature $T$. 
Such an open system in statistical physics is referred to as a canonical ensemble, whereas the system without dissipation is closed, and hence fits in the traditional category of micro-canonical ensembles. 
\end{remark}

\begin{remark}[On ergodicity]
	The dissipative stochastic systems are also ergodic on the level sets of the Casimirs determined by the initial conditions, whereas the dissipative systems without noise are not ergodic as they will rapidly converge to the minimum energy positions. 
\end{remark}

\begin{remark}[Time reversal]
	The results of this section also hold when evolving backward in time, using the change of variable $t\to -t$. Indeed, since the noise $dW$ is centred, only the dissipation will be affected by time reversal, and it will have the opposite effect, namely the system will tend toward the highest energy equilibrium position. 
\end{remark}

\section{Random attractors}
\label{RA-section}

We now turn to the study of the existence of random attractors (RAs) in our stochastic dissipative systems, in connection with the theory of random dynamical systems (RDS). 
The classic approach in studying the effect of stochastic forcing of nonlinear dynamical systems proceeds by integrating the system forward in time and performing averages, then studying the Fokker-Planck equation, as we have done up to now. 
Another approach studies random dynamical systems via the so-called \emph{pull-back method}.
We will not fully explore the theory of random dynamical systems and pull-back attractors here. Instead, we will only invoke the main results from the theory and refer the interested reader to \cite{crauel1994attractors,crauel1997random, arnold1995random,bonatti2006dynamics,kloeden2011nonautonomous} for in-depth accounts of these subjects.
In a nutshell, for a given fixed realisation of the noise, the average is taken over the initial conditions. 
The noise makes the system time-dependent; so the notion of an attractor should be defined in the pull-back sense, such that for large times the attractive set does not depend on time. 
That is, the pull-back attractor is defined by pulling back a given set of initial conditions from $t=0$ to $t\to -\infty$ and letting the system evolve to $t=0$. 
In the limit, the set obtained at $t=0$ is the pull-back attractor. 
In random dynamical systems theory, the pull-back attractor is usually called a random attractor, and if it is not singular, it may admit a particular type of measure, the Sinai-Ruelle-Bowen measure (SRB), which is also called a physical measure, see \cite{young2002srb}. 
We will denote the physical measure by $\mathbb{P}_\omega(\mu)$ for a given realisation of the noise $\omega$.
There is a fundamental relation between this SRB measure and the invariant measure of the Fokker-Planck equation which was first discovered in \cite{crauel1991markov} and later in a theorem of \cite{crauel1998additive}. This relation is informally given by 
\begin{align}
	\int_\Omega \mathbb{P}_\omega(\mu) d\omega = \mathbb{P}_{\infty}(\mu),  
\end{align}
for the probability space $\Omega$. 
Here we are referring to probability densities, and the SRB measure can be seen as the invariant measure most compatible with volume, although volume in phase space is not preserved, because of dissipation. For more explanation, see \cite{young2002srb}.

\begin{remark}[Periodic kicking]
	We are only considering here the interaction of noise and dissipation. 
	However, if the noise were replaced by a simpler deterministic forcing, similar results would emerge. 
	In particular, periodic forcing or kicking of dissipative dynamical systems has been studied in great detail in numerous works, e.g., in \cite{lin2010dynamics,lu2013strange}. 
	In section \ref{kicking}, we will implement periodic kicking and damping in the rigid body, and will numerically demonstrate the existence of non-singular attractors and chaos. 
	We have left deeper theoretical studies of these systems for future investigations.
\end{remark}

\subsection{Existence of attractors}
We first determine that the stochastic processes \eqref{SEP-Diss} and \eqref{SP-SD} do indeed admit random attractors, provided the top Lyapunov exponent is positive. See \cite{kondrashov2015data,schenk1998random} and references therein for more details about this type of approach. Then we will estimate the value of the top Lyapunov exponent using numerical simulations for the rigid body in section \ref{RB}.

\begin{theorem}\label{RA-EP-thm}
	The stochastic process \eqref{SEP-Diss} admits a random attractor, for every Lie group $G$.  
\end{theorem}

\begin{proof} 
The SDE \eqref{SEP-Diss} may be recast as a random dynamical equation  (RDE) by using the following vector Wiener processes $z_i$,
\begin{align}
	dz_i = \sigma_i dW_t^i\,,
\end{align}
where $z_i$ is understood as a vector process in the direction along $\sigma_i$. 
In the sequel, we will denote $z(t,\omega)= \sum_i z_i(t,\omega)\in \mathfrak g$.
The process $z(t)$ thus defines a random path in the Lie algebra $\mathfrak g$ and, via the exponential map, a random path in the group $G$ as $g(t,\omega)= e^{z(t,\omega)}$. 

We then define a new variable $\widetilde \mu(t)= g(t)\mu(t) := \mathrm{Ad}^*_{g(t)}\mu(t)$ and we have, from \eqref{SEP-Diss} (see for example \cite{marsden1999intro}),
\begin{align*}
	d\widetilde \mu(t) &= \mathrm{Ad}^*_{g(t)} \left ( -\sum_i\mathrm{ad}^*_{\sigma_i}\mu\circ dW + d\mu\right )
	= \mathrm{Ad}^*_{g(t)} \left (F(\mathrm{Ad}^*_{g(t)^{-1}}\widetilde \mu(t))\right) dt,
\end{align*}
where our stochastic process is generically written $d\mu = F(\mu)dt + G_i(\mu)\circ dW_i$ for convenience. 
From here, we have the RDE associated to \eqref{SEP-Diss} of the form
\begin{align}
	\frac{d}{dt}\widetilde \mu(t) = \widetilde F(\widetilde \mu(t),g(t)),
	\label{RDE}
\end{align}
where $\widetilde F$ is defined in the previous calculation as the drift part of the process. 
Recall that from the theory of selective decay we have \cite{gaybalmaz2013selective}
\begin{align*}
	\frac{d}{dt}h(\mu)  = -\,\theta \left \|\left [ \frac{\partial C}{\partial \mu},\frac{\partial h}{\partial \mu}\right ]\right \|^2,
\end{align*}
and $h(\mathrm{Ad}^*_g\mu) = h(\mu)$, because $h$ is $G$-invariant,  so that this equality becomes for \eqref{RDE},
\begin{align}
	\frac{d}{dt}h(\widetilde \mu)  = -\theta \left \|\left [ \mathrm{Ad}_{g^-1}\frac{\partial C(\widetilde \mu)}{\partial \widetilde \mu}, \mathrm{Ad}_{g^-1}\frac{\partial h(\widetilde \mu)}{\partial \widetilde \mu}\right ]\right \|^2 \le0\,.
\end{align}
This inequality assures that the energy decays at a random strictly negative rate. 
The existence of the random attractor then follows from standard arguments, demonstrated, for example, in the linear case by \cite{schenk1998random} and in a more general nonlinear setting by \cite{chekroun2011stochastic}.
\end{proof}

The idea of this proof is to generalise the linear change of variables used to recast the original stochastic process  as a random dynamical equation, by using a nonlinear group theoretical change of variable. 
The dissipative property follows from the selective decay theory and the invariance of the Hamiltonian under the group action. 
This theorem is general, in that no specific assumptions on the Lie group need to be imposed. 
In particular, modulo difficulties in analysis, the theorem should also apply for the diffeomorphism group used in the description of compressible fluid equations. However, we have no intention of investigating the infinite dimensional theory here.

The same result persists in the semidirect-product theory, as developed earlier. 
\begin{corollary}
	Theorem \ref{RA-EP-thm} applies to semidirect-product stochastic processes \eqref{SP-SD}.
\end{corollary}
\begin{proof}
	The proof follows the same argument, upon using the action of the group $G$ and the Lie algebra $\mathfrak g$ and the advected quantities in $V$ to define the change of variables. 
	The decay rate of the energy is given by using the deterministic selective decay formulae \eqref{SD-diss-SD}.
\end{proof}

\subsection{Existence of the SRB measure}
We now turn to the existence of the SRB measure. 
Theorem \ref{SRB-thm} below for the existence of SRB measures will invoke 
H\"ormander's theorem about the smoothness of transition probabilities for a diffusion satisfying the so-called H\"ormander (Lie) bracket conditions. 
The Lie bracket $[v,w](x)$ of two vector fields $v(x),w(x)$ in $\mathbb{R}^n$ is defined as 
\begin{align}
	[v,w](x) = Dv(x)w(x) - Dw(x)v(x), 
\end{align}
where we denote by
$Dv$ the derivative matrix given by $(Dv)_{ij} = \partial_j v_i = v_{i,j}$. Given an SDE of the form 
\begin{align}
	dx = A_0(x)dt  + \sum A_i(x)\circ dW_t^i,
\label{SDEgen}
\end{align}
the H\"ormander condition we use states that if the following condition is satisfied 
\begin{align}
\cup_{k\geq 1}\, V_k(x) = \mathbb{R}^n,\quad\hbox{for all}\quad x\,,
\label{hormander}
\end{align}
where
\begin{align}
	\begin{split}
	V_k(x) &= V_{k-1}(x) \cup \mbox{span}\{ [v(x),A_j(x)]: v\in V_{k-1},j\geq 0\}\quad \mathrm{and}\\
	V_0(x) &= \mbox{span}\{ A_j, j \geq 1 \} \,,
	\end{split}
\end{align}
then the invariant measure of \eqref{SDEgen} is smooth with respect to the Lebesgue measure. 

The H\"ormander condition implies the following standard theorem for stochastic dissipative systems.
\begin{theorem}\label{SRB-thm}
If the largest Lyapunov  exponent of \eqref{SEP-Diss} is positive, the random attractor is the support of a Sinai-Ruelle-Bowen (SRB) measure. 
\end{theorem}

\begin{proof}
The proof uses the corollary of Theorem B in \cite{ledrappier1988entropy}, which assumes the existence of a random attractor.  
The only point left to show here is that the parabolic H\"ormander condition \eqref{hormander}  is fulfilled.
Given the Stratonovich process \eqref{SDEgen} in $\mathfrak{g}^*$, we only need to check that the vector fields $A_1,\dots,A_N$ will span the tangent space to the coadjoint orbits as long as $N$ is sufficiently large. 
Since $A_i(\mu):= \mbox{ad}^*_{\sigma_i} \mu$, they are tangent to the coadjoint orbits. The minimal number of $A_i$ needed cannot be found, in general, as it will depend on the Lie symmetry algebra and the form of the Hamiltonian. 
Nevertheless, in our case the $ \sigma_i $ span the vector space ${\mathfrak g}$, and the H\"ormander condition is fulfilled. 

\end{proof}
\begin{corollary}
	 Theorem \ref{SRB-thm} also applies for the semidirect product case, even with $\eta_i= 0 $. 	
\end{corollary}
\begin{proof}
	The same argument applies here, even if $\eta_i=0$, as the semidirect product structure will automatically span the whole space, provided $\mathfrak g$ is already spanned and $h$ is not too degenerate on $V$. 
\end{proof}

\subsection{Estimation of Lyapunov exponents}

In principle, it is possible to compute the value of the top Lyapunov exponent as a function of the parameters of the system. However, this turns out to be a very challenging computation.   
Nonetheless, positivity of the top Lyapunov exponent is important to determine, as it allows us to use the previous Theorem \ref{SRB-thm} to prove the existence of a non-singular random attractor with an SRB measure and positive entropy. 
We will restrict ourselves to the first step of the calculation and explain why it is a difficult problem. 
We will then estimate the top Lyapunov exponent numerically in the example section, and leave the rigorous proof as an open problem.  

The very first step is to estimate the sum of the Lyapunov exponents using the multiplicative ergodic theorem (MET), that we state here in its simplest form. 

\begin{theorem}[MET theorem]
	\label{MET}
	Suppose that the stochastic process of $\mu$ has an ergodic invariant measure $\mathbb P_\infty$. Then there exists a subset $\Delta$ of the phase space which is invariant under the flow of $\mu$ and ordered Lyapunov exponents $\lambda_i$ for $i= 1, \ldots , n$ where $n$ is the dimension of the phase space such that the following properties hold for all $(\mu, \omega) $ in the invariant set $\Delta$:
		\begin{enumerate}
			\item The Lie algebra can be decomposed into a direct sum 
				\begin{align*}
					T_\mu \mathfrak g=  E_1(\omega, \mu) \oplus \ldots \oplus E_n(\omega, x),
				\end{align*}
				with 
				\begin{align}
					\delta \mu \in  E_i \Leftrightarrow \lim_{t\to \infty} \frac{1}{t}\mathrm{log}\| DF(t,\omega, \mu) \delta \mu \| = \lambda_i,
					\label{lyap-formula}
				\end{align}
				where $d(\delta \mu) = DF(t,\omega, \mu) \delta \mu$ is the linearisation of the flow equation $d\mu = F(t,\omega, \mu)$ along a particular nonlinear flow, $\mu(t)$. 
			\item For a generic $\delta \mu$, the associated Lyapunov exponent is $\lambda(\omega, x, v) = \lambda_+$, the largest one. 
			\item The sum of the Lyapunov exponents is given by 
				\begin{align}
					\lim_{t\to \infty} \frac{1}{t} \mathrm{log}\,  \mathrm{det}\, DF (t, \omega, \mu) = \sum_i \lambda_i.
					\label{sum-lyap-formula}
				\end{align}
		\end{enumerate}
\end{theorem}
For simplicity here we haveassumed that the multiplicity of the Lyapunov exponents is always $1$; that is, they are all distinct. 
We refer to \cite{arnold1995random} for more details on this theorem and its generalisations. 
The stochastic systems which we consider here are written on compact semi-simple Lie algebras, such that $c^2=\|\mu\|^2$ is constant and defines a bounded set. 
The energy functional $h(\mu)$ is also a generic quadratic kinetic energy term, with a given inertia tensor $\mathbb I^{-1}$, corresponding to the Hessian matrix of $h(\mu)$. 
We can then prove the following formula for the sum of the Lyapunov exponents. 
\begin{proposition}\label{sum-lyap}
	Provided the Lie algebra is compact and semi-simple, the sum of the Lyapunov exponents is estimated from below by
	 \begin{align}
		 \sum_i d_i\lambda_i \geq -  \frac12 |\epsilon| n\sigma^2  -\theta |\epsilon|\left (c^2 \mathbb I_\mathrm{min}^{-1}-  \mathbb E_\infty h(\mu)\right ),
	\label{sum-lambdas}
	\end{align}
	where $c= \|\mu\|^2$, $\epsilon$ is the Killing form constant, $n$ is the number of $\sigma_i=\sigma e_i$ spanning the Lie algebra and $d_i$ are the multiplicity of each Lyapunov exponent. Thus, the dimension of the Lie algebra, $n={\rm dim}(\mathfrak{g})$. 
	The quantity $\mathbb{I}_\mathrm{min}^{-1}$ is the largest eigenvalue of the Hessian of the Hamiltonian. 
	The expectation $\mathbb E_\infty $ is taken with respect to the invariant measure $\mathbb P_\infty$. An estimation from above is also available, upon using $\mathbb I_\mathrm{max}^{-1}$, the minimal eigenvalue, instead of $\mathbb I_\mathrm{min}^{-1}$. 
\end{proposition}
\begin{proof}
	Let us denote the stochastic process \eqref{SEP-Diss} in It\^o form by
	\begin{align*}
		d\mu = F(\mu) dt + \sum_iG_i(\mu)dW_t^i.
	\end{align*}
	We can now directly apply the MET theorem \ref{MET} to compute the sum of the Lyapunov exponents
	\begin{align}
		\sum_i\lambda_i = \lim_{t\to \infty} \frac{1}{t} \mathrm{log}\, \mathrm{det}\,  \delta \mu (t,\omega,x)\, .
	\label{MET-thm}
	\end{align}
	We can then use Jacobi formula and ergodicity to rewrite \eqref{MET-thm} as 
	\begin{align}
		 \lim_{t\to \infty} \frac{1}{t} \mathrm{log}\, \mathrm{det}\, \delta \mu (t,\omega,x)=  \lim_{t\to \infty} \frac{1}{t} \mathrm{Tr}\int^t  DF(\varphi(t,x,\omega))ds\, .
	\end{align}
	Here, we have intentionally dropped the linearisation of the noise amplitude, since this term will vanish (the trace of the adjoint action always vanishes).
	Finally, ergodicity of this process gives
	\begin{align}
		\sum_i\lambda_i = \int \mathrm{Tr}(DF(\mu)) \mathbb P_\infty(\mu) d\mu\,, 
	\end{align}
	where $\mathbb P_\infty$ is the invariant measure of the underlying stochastic process. 
	The calculation of the trace simplifies in the case of a compact semi-simple Lie algebra with the Killing form  $\mathrm{Tr}(\mathrm{ad}_A\mathrm{ad}_B)=  \epsilon A\cdot B$, where $\epsilon< 0 $ depends on the Lie algebra.  
	Then, using the explicit form of $F$ along with semi-simplicity for $\mathfrak g$, yields 
	\begin{align*}
		F(\mu) = \mathrm{ad}_\frac{\partial h}{\partial \mu} \mu + \theta\, \mathrm{ad}_\mu\mathrm{ad}_\mu\frac{\partial h}{\partial \mu}  + \frac12\sum_i  \mathrm{ad}_{\sigma_i}\mathrm{ad}_{\sigma_i}\mu\, .
	\end{align*}
	Consequently, we arrive at  
	\begin{align}
		\mathrm{Tr}(DF(\mu)) &= \mathrm{Tr}\left (- \theta\, \mathrm{ad}_\mu \mathrm{ad}_\frac{\partial h}{\partial \mu} + \theta\, \mathrm{ad}_\mu\mathrm{ad}_\mu\frac{\partial^2h }{\partial \mu^2} +  \frac12\sum_i \sigma^2 \mathrm{ad}_{e_i}\mathrm{ad}_{e_i} \right )\\
		&= |\epsilon| \theta h(\mu)+\theta A(\mu,\mu)  -  \frac 12|\epsilon| n \sigma^2, 
		\label{Aterm}
	\end{align}
	 where $n$ is the number of $\sigma_i$ fields. 
	 The $A(\mu,\mu)$ term depends on the Lie algebra structure constants, and is difficult to obtain explicitly for every compact semi-simple Lie algebra. 
	 However, we can estimate it here using
	 \begin{align}
		  |\epsilon |\theta \kappa(\mu,\mu) \mathbb I^{-1}_\mathrm{max}\geq \theta \mathrm{Tr}\left ( \theta \mathrm{ad}_\mu \mathrm{ad}_\mu \frac{\partial^2 h}{\partial \mu^2}\right )\geq |\epsilon |\theta \kappa(\mu,\mu)\mathbb I^{-1}_\mathrm{min}. 
		 \label{A-estimate}
	 \end{align}
	 Collecting terms then gives a lower and upper bound for the sum of the Lyapunov exponents \eqref{sum-lambdas}.
\end{proof}

Unfortunately, Proposition \ref{sum-lyap} only gives a negative lower bound for the sum of the Lyapunov exponents of an arbitrary compact semi-simple Lie algebra. 
Indeed, the last term can be bounded from above by 
\begin{align*}
	\mathbb E_\infty h(\mu) 
	= \int h(\mu) e^{-\frac{2\theta}{\sigma^2}h(\mu)}d\mu
	\leq c^2 \mathbb I_{\mathrm{min}}^{-1}.
\end{align*}
Thus, the sum of the Lyapunov exponents is found to be negative. 

A precise value can be computed explicitly for each Lie algebra, by using the structure constants to calculate the term $A(\mu,\mu)$ in \eqref{Aterm}. 
We will show this calculation in the case of $SO(3)$ in the free rigid body example in Section \ref{RB}. 

\begin{remark}[Non-compact Lie algebras]
This argument does not apply for non-compact semi-simple Lie algebras, as the Killing form is not sign-definite, thus it does not provide us with a norm. 
\end{remark}

Having only a negative lower bound for the sum of the Lyapunov exponent means we must proceed further by estimating the top Lyapunov exponent, in order to obtain an SRB measure. 
Obtaining a precise estimate for our general class of system is still an open problem, but we will show here the main steps toward this result. 
We will then numerically estimate it for the simple case of $SO(3)$ in the section \ref{RB}.

First recall the linearisation of the flow, in It\^o form
\begin{align}
	d\delta \mu = DF(\mu) \delta \mu\, dt + \sum_i DG_i(\mu)  \delta \mu \,dW_t^i
\end{align}
where $DG_i(\mu) = \mathrm{ad}_{\sigma_i}$.
We then want to apply the Furstenberg-Kasminskii formula. (See \cite{arnold1995random} for details.) For this purpose, we introduce the following change of variables 
\begin{align}
	R= \|\delta \mu\|\in \mathbb R\quad\mathrm{and}\quad  \Phi= \frac{\delta \mu}{\|\delta \mu\|} \in \mathbb S^{n-1},
\end{align}
where $\mathbb S^{n-1}$ is the ($n$-1)-sphere. 
The corresponding stochastic processes are
\begin{align}
	d R&= \langle \Phi, (DF(\mu)dt + DG_idW^i_t)\Phi\rangle R := Q(\mu,\Phi) R\,,\\	
	d \Phi&= (DF(\mu)dt + DG_idW^i) \Phi  - \langle \Phi, (DF(\mu)dt + DG_idW^i ) \Phi\rangle \Phi\nonumber \\
	&:= P(\mu,\Phi)dt + \sum_i P_i(\mu,\Phi)dW_t^i\, .
\end{align}
As a direct consequence of ergodicity and the multiplicative ergodic theorem \ref{MET} [Item 1 and 2], the Furstenberg-Kasminskii formula then gives the top Lyapunov exponent by evaluating the following integral
\begin{align}
	\lambda_+ = \int Q(\mu,\Phi) \mathbb Q(\mu,\Phi)d\mu d \Phi,
	\label{FK-formula}
\end{align}
where $\mathbb Q(\mu,\Phi)$ is the joint invariant distribution of the processes for $\mu$ and $\Phi$, where the processes for $\Phi$ depend on $\mu$, but not the reverse. 

One can see that a rough estimate of the integrand from below would always be negative. Consequently, we need more work to obtain a positive bound.  
Let us evaluate the quantity $Q$ at the positions $(\mu_i,\Phi_j)$, where we use the basis corresponding to the eigenvalues of the Hessian of $h$, i.e., the inverse of the moment of inertia $\mathbb I^{-1}$. 
We also order the eigenvalues as $\mathbb I_1>\ldots >\mathbb I_n$. 
Note that because the dynamics of $\mu$ takes place on the coadjoint orbit, the linearisation is tangent to it and thus $i\neq j$ for the choice of positions $(\mu_i,\Phi_j)$. 
A direct calculation gives the following simplifications of the quantity $Q(\mu_i,\Phi_j)$
\begin{align*}
	Q(\mu_i,\Phi_j)&= \langle \Phi_j, DF(\mu_i)\Phi_j\rangle \\
	&=  \langle \mathrm{ad}_{\mu_i}\Phi_j, \mathbb I^{-1} \Phi_j\rangle-\frac12\sigma^2\sum_k  \langle \mathrm{ad}_{e_k}\Phi_j,\mathrm{ad}_{e_k}\Phi_j\rangle \\
	&+ \theta\, \langle \mathrm{ad}_{\mu_i}\Phi_j,\mathrm{ad}_{\mathbb I^{-1}\mu_i}\Phi_j\rangle 
	- \theta\, \langle \mathrm{ad}_{\mu_i}\Phi_j,\mathrm{ad}_{\mu_i}\mathbb I^{-1} \Phi_j\rangle\\ 
	&=  \langle \mu_i,\mathrm{ad}_{\Phi_j} \mathbb I^{-1} \Phi_j\rangle-\frac12\sigma^2\sum_k \| \mathrm{ad}_{e_k}\Phi_j\|^2 \\
	&+ \theta\, \langle \mathrm{ad}_{\mu_i}\Phi_j,\mathrm{ad}_{\mu_i}(\mathbb I^{-1}_i\mathrm{Id}-\mathbb I^{-1}) \Phi_j\rangle\\ 
	&=  -\frac12\sigma^2\sum_k  \| \mathrm{ad}_{e_k}\Phi_j\|^2+ \theta\, (\mathbb I^{-1}_i -\mathbb I^{-1}_j)\|\mathrm{ad}_{\mu_i}\Phi_j\|^2\\
	&=  -\frac{n-1}{2}\sigma^2  +  \theta\, c^2(\mathbb I^{-1}_i -\mathbb I^{-1}_j),
\end{align*}
where we have used the Casimir sphere radius $c$  (thus $\mu_i = c\, e_i$) and the fact that $\|\Phi\|^2= 1$. 
The last formula conveys a lot of information about the possible sign for $\lambda_+$. 
Indeed, depending  on the choice of $i$ and $j$, it is possible that $Q(\mu_i,\Phi_j)$ is positive, provided the difference between the moments of inertia is large enough with respect to the $\sigma$, $\theta$ and $c$. 
That difference between the moments of inertia is important and is tied to the nature of the random attractor. 
Indeed, differences in the moments of inertia imply different speeds for nearby orbits, and thus a shear effect. 
This shear effect is a common source of random attractors, which has appeared in a number of recent papers such as \cite{wang2003strange, lin2008shear}. 
The necessity of sufficiently large shear for the existence of the random attractor is thus implied by the last formula in the computation above. 
Of course, this is not the whole story, as the original dynamics of the system is also an important factor. 
Indeed, without noise and dissipation, the system is Hamiltonian. Thus the sum of Lyapunov exponents vanishes in this case, and the top Lyapunov exponent must be positive for a large set of initial conditions. 
In the present case, integrating the deterministic part of $Q$ against the Gibbs measure is already analytically difficult and no particular sign can be easily expected from examining this term.
The only thing we can expect at this stage is that for large dissipation, when the Gibbs measure is localised around the equilibrium positions of minimum energy, the dynamics of the deterministic system is negligible, whereas for a small dissipation the deterministic dynamics will be important. 
We refer to section \ref{RB}, where we will evaluate the top Lyapunov exponent in the rigid body example by numerical simulation.  

We now turn to the semidirect product structure and also estimate the sum of the Lyapunov exponents in the following proposition, where we choose to use the Casimir $C(\mu,q)= \frac{1}{\epsilon}\kappa(q,q)=c^2$ for the dissipative term for simplicity only. 
\begin{proposition}\label{LEs-SP}
	The sum of the Lyapunov exponents for the semidirect stochastic process \eqref{SP-SD} with Casimir $C(\mu,q)= \frac{1}{\epsilon}\kappa(q,q)=c^2$ is given by
	 \begin{align}
		 \sum_i \lambda_i \geq - |\epsilon|\, n\sigma^2  -\theta c^2 \mathbb I_\mathrm{min}^{-1}.
		 \label{LEbound-SP}
	\end{align}
\end{proposition}
\begin{proof}
We follow closely the proof for the Euler-Poincar\'e case. 	
Let us denote the stochastic process \eqref{SP-SD} in It\^o form by
	\begin{align*}
		d(\mu,q) = \left [F^\mu(\mu,q) +F^q(\mu,q)\right ]dt + \sum_i\left [G_i^\mu(\mu,q)+G_i^q(\mu,q)\right ]dW_t^i,
	\end{align*}
	where we denoted $F^\mu$ (resp. $F^q$) the $\mu$ (resp. $q$) component of $F$.
	The MET theorem, Jacobi's formula and ergodicity of this process gives
	\begin{align}
		\sum_i\lambda_i = \int \mathrm{Tr}\left (D_\mu F^\mu (\mu,q) + D_qF^q(\mu,q)\right ) \mathbb P_\infty(\mu,q) d(\mu,q), 
	\end{align}
	where $\mathbb P_\infty(\mu,q)$ is the invariant measure of the underlying stochastic process, and $D_\mu$ and $D_q$ denotes the Jacobian matrices taken with respect to $\mu$ or $q$ respectively.   
	Consequently, after substituting the Casimir $C(q)$ in the general formula \eqref{SP-SD}  we obtain
\begin{align*}
	\mathrm{Tr}(D_\mu F^\mu+D_qF^q) = \mathrm{Tr}\left ( -\theta\,  \mathrm{ad}_q\mathrm{ad}_q\mathbb I^{-1}   -\sigma^2 \sum_i\mathrm{ad}_{e_i}\mathrm{ad}_{e_i} \right ), 
\end{align*}
and, using again \eqref{A-estimate}, we have the result \eqref{LEbound-SP}. 
\end{proof}

As before, the sum of the Lyapunov exponents is negative. 
Thus, we must estimate the top Lyapunov exponent in order to prove the existence of a non-singular SRB measure for this system. 
The same difficulty as before remains in this case for estimating the top Lyapunov exponent using the Furstenberg-Kasminskii formula. 

\paragraph{\bf Summary.} In this section, we have studied the interaction of multiplicative noise and nonlinear dissipation on coadjoint orbits. 
For this purpose, we added a double bracket dissipation mechanism to the previously derived  stochastic process in order to preserve the coadjoint orbit structure on which the solutions of the stochastic process are supported. 
In the case of semi-simple Lie algebras, we obtained the invariant measure of the Fokker-Planck equation and found the associated Gibbs measure on the coadjoint orbits.  
In the semidirect product case, this result was shown to hold for the marginal distribution of the advected quantity only, where the Gibbs distribution depends only on the potential energy. 
We then proved the existence of random attractors for a wide class of systems by using the dissipative property of the double bracket and the H\"ormander condition on the generating vector fields, provided the top Lyapunov exponent is positive. 
Unfortunately, we were not able to derive an exact lower positive bound for the top Lyapunov exponent. However, numerical investigations in the example sections \ref{RB} and \ref{HT} will provide us with strong evidences of the positivity of the top Lyapunov exponent for some region of the parameter space $(\theta, \sigma)$. 
In the next two sections we will study two specific examples of stochastic deformations of the Euler-Poincar\'e dynamical equation, for the free rigid body and the heavy top, using both analytical and numerical tools. 

\section{Euler-Poincar\'e example: the stochastic free rigid body}\label{RB}

This section introduces stochastic dynamics for the classic example of the Euler-Poincar\'e dynamical equation; namely, the equation for free rigid body motion. 
Stochastic rigid body models have arisen in various fields of application, such as nanoparticles \cite{shrestha2015simulation,blum2006measurement}, molecular biology \cite{gordon2009algorithm}, polymer dynamics \cite{chirikjian2012stochastic2}[Section 13.7], filtering in aeronautics: guidance and tracking \cite{willsky1974estimation}. 
We refer to \cite{chirikjian2012stochastic1,chirikjian2012stochastic2} for more applications. 
One source of models for stochastic dynamics stems from the so-called rotational Brownian motion of molecules. Rotational Brownian motion comprises the random change in the orientation of a polar molecule due to collisions with other molecules and is an important element in the theory of dielectric materials. 
Perrin and Debye's non-inertial theories are the most well-known models, see for example \cite{chirikjian2012stochastic2}[section 16.3]. 
Rotational Brownian motions have also been observed in a laboratory setting and have been properly documented in \cite{han2006brownian}. 
Much of the current research in rotational Brownian motions has been devoted to inertial models, non-spherical molecules and possibility of dipole-dipole interactions. 
Walter et al. \cite{walter2010stochastic} took a step further in proposing an inertial, Langevin type of generalisation to the rigid body equations aiming at studying systems of rigid bodies as models for polymer dynamics. 
The coupling between linear and rotational dynamics was important in this case, to capture the motion features of long polymeric chains. 
Their models assume linearity in the noise for both linear and angular momentum variables, whereas the model used here is fully nonlinear with multiplicative noise and preserves strong geometrical features such as the coadjoint orbits. 

\begin{remark}[The LLG equation]
We mention that the stochastic Landau-Lipschitz-Gilbert (LLG) equation studied for example in \cite{garanin1997fokker, brzezniak2012large,kohn2005magnetic} has the same structure as our stochastic dissipative rigid body equation.
Indeed, we preserve the coadjoint orbit, thus the amplitude of the momentum variable, which corresponds to the strength of magnetic moments in the LLG model. 
We will not study this link further here as the LLG equation is a PDE in two or three dimensions and requires different analytical methods than the rigid body equation.  
\end{remark}

\subsection{The stochastic rigid body}
The canonical example for illustrating the Euler-Poincar\'e reduction by symmetry is the free rigid body, whose configuration space is the group of rotations $SO(3)$.  
For a complete treatment from the viewpoint of reduction we refer to \cite{marsden1999intro}, 
For simplicity here, we rely on the isomorphism $\mathfrak{ so}(3)\cong \mathbb R^3$ which translates the commutator in the Lie algebra to the cross product of three-dimensional vectors, via $[A , B ] \to \boldsymbol A\times\boldsymbol B$, where $\mathbb R^3$ vectors are denoted with bold font.
This map allows us to use a slightly different Killing form than the canonical one. Namely, we shall use the scalar product as our pairing, via the formula $\boldsymbol A\cdot \boldsymbol B= -\frac12 \kappa(A,B)$. 

The reduced Lagrangian of the free rigid body is written in terms of the angular velocity $\Omega\in \mathfrak{so}(3)$ and a prescribed moment of inertia $\mathbb I\in \mathrm{Sym}(3)$ as
\begin{align}
	l(\boldsymbol \Omega) = \frac12 \boldsymbol \Omega \cdot\mathbb{ I} \boldsymbol \Omega
	: = \frac12 \boldsymbol \Omega\cdot \boldsymbol \Pi\,,
	\label{H-RB}
\end{align}
where the angular momentum $\boldsymbol \Pi$ is defined accordingly and the Legendre transform gives the reduced Hamiltonian $h(\boldsymbol \Pi) = \frac12 \boldsymbol \Pi \,\mathbb I ^{-1}\boldsymbol \Pi$.
We take the stochastic potential to be linear in the momentum variable $\boldsymbol \Pi$
\begin{align}
	\Phi_i(\boldsymbol \Pi) = \sum_{i=0}^3 \boldsymbol{\sigma}_i\cdot \boldsymbol\Pi\,,
	\label{Phi-RB}
\end{align}
where the constants $\boldsymbol{\sigma}_i$ generically span $\mathbb R^3$ but can be chosen in various ways. 
The stochastic process for $\boldsymbol \Pi$ is then computed from \eqref{SEP} to be 
\begin{align}
	d\boldsymbol\Pi + \boldsymbol\Pi\times \boldsymbol\Omega\, dt + \sum_i\boldsymbol\Pi\times \boldsymbol{\sigma}_i \circ dW^i_t=0 ,
	\label{Sto-RB-stra}
\end{align}
and the corresponding It\^o process is
\begin{align}
	d\boldsymbol\Pi + \boldsymbol\Pi\times \boldsymbol\Omega\, dt +\frac{1}{2} \sum_i(\boldsymbol\Pi\times \boldsymbol{\sigma}_i)\times \boldsymbol{\sigma}_i\, dt +  \sum_i \boldsymbol\Pi\times \boldsymbol{\sigma}_i\,  dW^i_t= 0. 
	\label{Sto-RB-ito}
\end{align}

The coadjoint orbit defined by a level set of the quadratic Casimir $\|\boldsymbol\Pi\|^2=\mathrm{c}^2$ is preserved  in our geometrical construction, as may be checked by a direct computation in both the Stratonovich and the It\^o stochastic representations. 
Although the Casmir is conserved, the energy $h(\boldsymbol\Pi) = l(\boldsymbol\Omega)$ is not a conserved quantity in general.  
Indeed, since the moment of inertia $\mathbb I$ is a symmetric matrix, the stochastic process associated to $h$ can be found to be 
\begin{align}
	dh= \sum_i(\boldsymbol \Pi\times \boldsymbol{\sigma}_i)\cdot [\mathbb I^{-1}(\boldsymbol \Pi\times \boldsymbol{\sigma}_i)-(\boldsymbol \Omega\times \boldsymbol{\sigma}_i) ]\, dt+2\sum_i   (\boldsymbol \Pi\times \boldsymbol{\sigma}_i)\cdot \boldsymbol \Omega\,  dW_t^i.
	\label{SDE-energy}
\end{align}
In the general case, one only has bounds for the energy given by the two stable equilibrium positions of the rigid body, namely $E_\mathrm{min} = \frac {1}{2\mathrm{I_3}}|\Pi_3(0)|^2$ and $E_\mathrm{max} = \frac {1}{2\mathrm{I_1}}|\Pi_1(0)|^2$ if $I_1\leq I_2\leq I_3$. 
Thus, the energy may randomly fluctuate within these bounds. 

Apart from the obvious case of $\mathbb I=Id$, one can check that the system with $\mathbb I= (I_1,I_1,I_3)$ and  $\boldsymbol \sigma= (0,0,\sigma_3)$ conserves the energy for every values of $I_1,I_3$ and $\sigma_3$ . 
In this case, the stochastic rigid body reduces to the Kubo oscillator of \cite{kubo1991statistical}
\begin{align*}
	d\Pi_1 = \Pi_2(a\Pi_3  dt +  \chi_3\circ dW), \quad d\Pi_2 =  -\Pi_1(a\Pi_3 dt +  \chi_3\circ dW)\quad \mathrm{and}\quad  d\Pi_3 =0,  
\end{align*}
where $a:= \frac{I-I_3}{I\, I_3}$. 
This system is integrable by quadratures and a solution is 
\begin{align*}
	\Pi_1(t) &= \Pi_1(0)\, \mathrm{cos}(\gamma t + \chi_3 W_t) - \Pi_2(0)\, \mathrm{sin}(\gamma t + \chi W_t), \\
	\Pi_2(t) &= \Pi_2(0)\, \mathrm{cos}(\gamma t + \chi_3 W_t) + \Pi_1(0)\, \mathrm{sin}(\gamma t + \chi W_t),
\end{align*}
where $\gamma:= a \Pi_3$.
Although the deterministic free rigid body is integrable, the only known integrable stochastic rigid body is this particular case, which reduces to the Kubo oscillator. 

\subsection{Fokker-Planck equation}

The Fokker-Planck equation of the process \eqref{Sto-RB-stra} is simply given for a probability density $\mathbb P$ by 
\begin{align}
	\frac{d}{dt}\mathbb P + (\boldsymbol\Pi\times \boldsymbol\Omega) \cdot \nabla \mathbb P  +\frac12\sum_i (\boldsymbol\Pi\times \boldsymbol{\sigma}_i)\cdot \nabla[(\boldsymbol\Pi\times \boldsymbol \sigma_i) \cdot \nabla \mathbb P ]=0 \,,
\label{FPRB}
\end{align}
where $\nabla:=\nabla_{\boldsymbol\Pi}$ is the gradient with respect to the independent variable  $\boldsymbol\Pi\in\mathbb{R}^3$.
According to Theorem \ref{limit-thm}, the invariant, or limiting distribution $\mathbb P_\infty$ is constant on coadjoint orbits, which are spheres.  

Based on this result, more can be said about the energy evolution of the stochastic rigid body, without embarking on any deeper studies into the coupled stochastic processes \ref{SDE-energy} and \eqref{Sto-RB-stra}. 
For example, by ergodicity of \eqref{Sto-RB-stra}, the long time average of the stochastic rigid body motion follows the limiting distribution $\mathbb P_\infty$. 
In terms of energy, the distribution is not uniform, but will be proportional, at a given energy, to the length of the deterministic trajectory of the rigid body with this energy. 
The energy will thus randomly oscillate between two bounds, with maximum probability to be near the energy of the unstable equilibrium.

\subsection{Double bracket dissipation}
The double bracket dissipation for the rigid body involves the only Casimir $\|\boldsymbol \Pi\|^2$ and yields, with noise, the dissipative stochastic process 
\begin{align}
	d\boldsymbol\Pi + \boldsymbol\Pi\times \boldsymbol\Omega\,  dt + \theta\, \boldsymbol \Pi\times( \boldsymbol \Pi\times \boldsymbol \Omega)\, dt+ \sum_i\boldsymbol\Pi\times \boldsymbol{\sigma}_i \circ dW^i_t=0 .
	\label{RB-diss}
\end{align}
The It\^o formulation is similar to \eqref{Sto-RB-ito} and will not be written here. The corresponding 
the Fokker-Planck equation is 
\begin{align}
	\begin{split}
	\frac{d}{dt}\mathbb P& + (\boldsymbol\Pi\times \boldsymbol\Omega)\cdot  \left ( \nabla \mathbb P - \theta\, \boldsymbol \Pi \times  \nabla \mathbb P\right )+\frac12\sum_i (\boldsymbol\Pi\times \boldsymbol{\sigma}_i)\cdot \nabla[(\boldsymbol\Pi\times \boldsymbol \sigma_i) \cdot \nabla \mathbb P ]=0 .
	\end{split}
	\label{FP-RB-SD}
\end{align}

The Fokker-Planck equation for stochastic rigid body dynamics with selective decay may be found by specialising the general proof given for Theorem \ref{FP-diss-thm}. 
Indeed, we can rewrite the Fokker-Planck equation as 
\begin{align}
	\frac{d}{dt}\mathbb P + (\boldsymbol\Pi\times \boldsymbol\Omega) \cdot \nabla \mathbb P + \nabla\cdot \left (  \theta\, \boldsymbol \Pi \times (\boldsymbol\Pi\times \boldsymbol\Omega) \mathbb P   -\frac12\sigma^2\, \boldsymbol\Pi\times( \boldsymbol\Pi\times \nabla \mathbb P )\right )=0 \,,
\end{align}
where we have used $\nabla\cdot (\boldsymbol\Pi\times (\boldsymbol\Pi\times \boldsymbol\Omega))= 0$. The last term in \eqref{FP-RB-SD} simplifies as
\begin{align*}
	\sum_i (\boldsymbol\Pi\times \boldsymbol e_i) [( \boldsymbol\Pi\times  \boldsymbol e_i)\cdot \nabla\mathbb P] 
	&=  \sum_i (\boldsymbol\Pi\times \boldsymbol e_i) [( \nabla\mathbb P \times \boldsymbol\Pi)\cdot  \boldsymbol e_i] 
	= \boldsymbol\Pi\times (\nabla\mathbb P\times \boldsymbol\Pi), 
\end{align*}
since the sum over $i$ is simply the decomposition of the vector $( \nabla\mathbb P \times \boldsymbol\Pi)$ into its $\boldsymbol e_i$ components. 
Consequently, the asymptotic equilibrium solution tends to 
\begin{align}
	\mathbb P_\infty = Z^{-1}e^{-\frac{2\theta}{\sigma^2}h(\Pi)}\,,
\end{align}
in which the overall sign of the exponential argument is negative, since $\theta>0$. 

\subsection{Random attractor}
For $\mathfrak{so}(3)$, we can go beyond Theorem \ref{sum-lyap} to obtain the exact value of the sum of the Lyapunov exponents.
\begin{proposition}
	 The sum of the Lyapunov exponents can be given exactly as 
	 \begin{align}
		 \sum_i \lambda_i = -3\sigma^2  - \theta\left (  c^2 \mathrm{Tr}\,\mathbb I^{-1}-6 \mathbb E_\infty h\right ) ,
	\label{sum-lambdas-so3}
	\end{align}
	where $c$ is the value of the Casimir function,  and $\theta>0$. 
\end{proposition}
\begin{proof}
	We can compute the term $A$ of Theorem \ref{sum-lyap} exactly with the structure constants $c^{ij}_k= \epsilon_{ijk}$
	 \begin{align*}
		 -A(\boldsymbol \Pi,\boldsymbol \Pi)&=-\mathrm{Tr}(\mathrm{ad}_\Pi\mathrm{ad}_\Pi\mathbb I^{-1})= -c^{im}_nc^{jn}_m\mathbb I^{-1} \Pi_i \Pi_j\\
		 &= \Pi_1^2(\mathbb I_2^{-1}+\mathbb I_3^{-1}) +  \Pi_2^2 ( \mathbb I_1^{-1}+\mathbb I_3^{-1})+\Pi_3^2 (\mathbb I_1^{-1}+\mathbb I_2^{-1}),
	 \end{align*}
	 which, when combined with the Hamiltonian, yields the result in equation \eqref{sum-lambdas-so3}. 
\end{proof}

We now turn to the numerical estimation of the top Lyapunov exponent for the stochastic damped rigid body. 
More explanation of the numerical scheme used can be found in the appendix \ref{numerics-lyap}. 
The result is displayed in figure \ref{fig:top-lyap} where we sampled $\theta$ and $\sigma$ with $0.1$ steps and used a spline interpolation for smoothing data. 
This result must not be taken to be exact, since, for example, the regions of large or small noise are the least accurate, as larger noise requires smaller time steps and a sufficiently small noise loses the ergodicity property sooner, as the simulations are run for a finite time. 
Nevertheless, these results numerically demonstrate that the top Lyapunov exponent is positive over a large region of the parameter space $(\theta,\sigma^2)$. 
Based on these numerical results, it is of course not possible to show that the observed chaos is not transient, but the longer runs suggest that the positive top Lyapunov exponent reaches a stable constant value. 
\begin{figure}[htpb]
	\centering
	\includegraphics[scale=0.6]{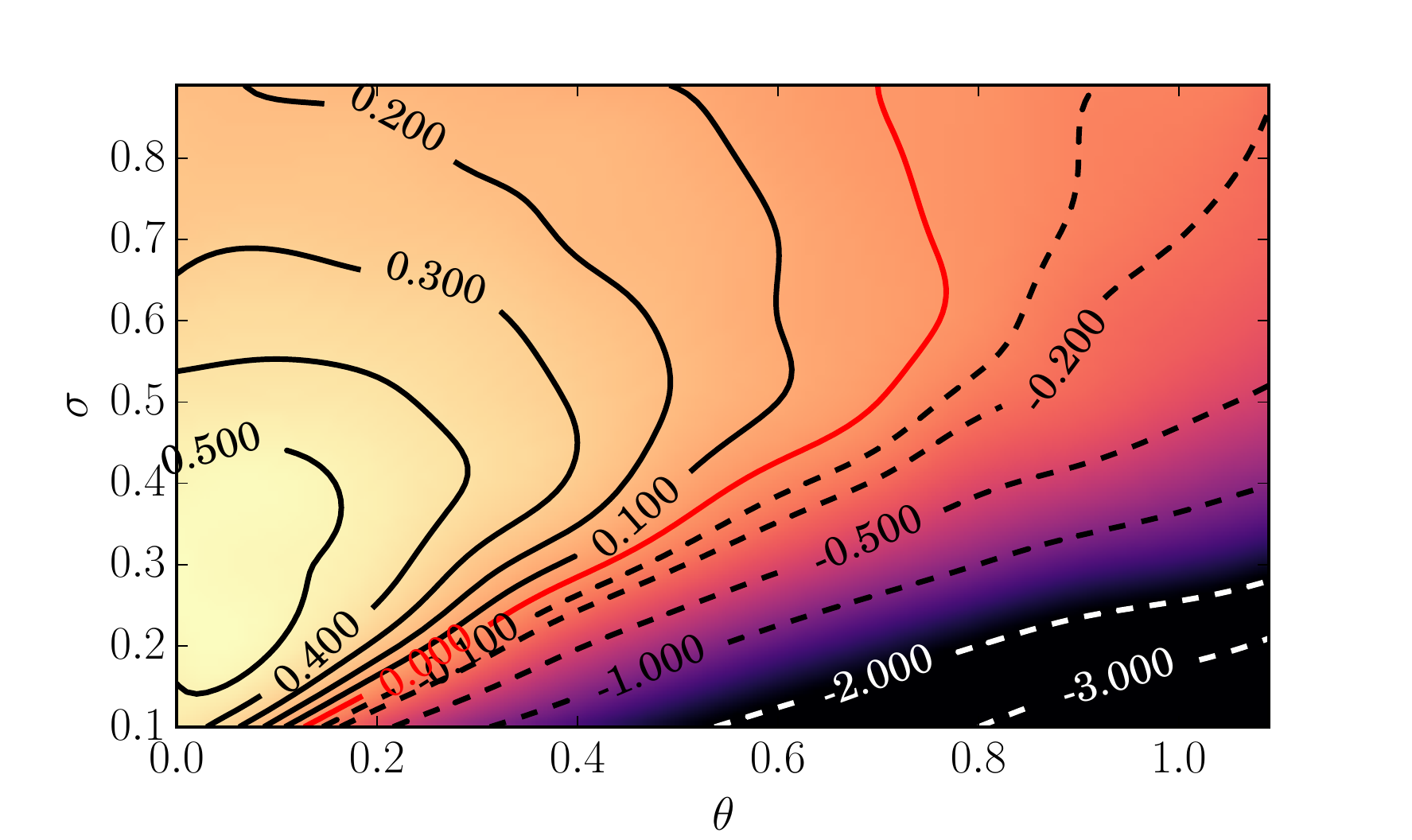}
	\caption{{\small This figure displays the value of the top Lyapunov exponent of the stochastic damped rigid body with $\mathbb I= (1,2,3)$ and $c=1$ in the parameter space $(\theta,\sigma)$. 
	This clearly shows a large region of positive Lyapunov exponent, implying chaotic behaviour in the system. See the text for a more detailed discussion. }}
	\label{fig:top-lyap}
\end{figure}
Apart from the demonstration of a positive top Lyapunov exponent, this figure also provides us with a better understanding of this system, which we now discuss briefly. 
\begin{itemize}
	\item For $\theta = 0$, the attractor is the entire space, that is the momentum sphere, and the invariant distribution is uniform on it. No random attractor or SRB measure exists in this case. 
		Nevertheless, we learn that by increasing the noise, we first observe an increase of the amount of chaos, and then a decrease. 
		The decrease for large noise is rather slow and it is not expected that the top Lyapunov exponent will ever become negative under further increase of the noise.
		This is because the shear of the rigid body is bounded by the difference between the two opposite moment of inertia. 
		In turn, the magnitude of the top Lyapunov exponent is bounded to the extent that it is linked to this shear. 
		This is in contrast with random attractors in the plane, or more generally in non-compact spaces, which can possess an arbitrarily large shear. 
		For example in the case of planar systems with limit cycles studied numerically by \cite{lin2008shear} and analytically by \cite{engel2016syncronisation}, the top Lyapunov is not bounded for large noise amplitudes.

	\item The limit $\sigma\to 0$ cannot be numerically computed, but it is easy to extrapolate it from this graphic. 
		First, notice that if $\theta=\sigma=0$ we are in the Hamiltonian case of rigid body dynamics, hence the sum of the Lyapunov exponents must be zero and the top Lyapunov exponent is therefore positive. 
		Furthermore, the Lyapunov exponents will depend on the initial conditions, as the system is not ergodic anymore, thus the limit $\sigma\to 0$ is not very well defined. 
		Nevertheless, if $\theta>0$, the top Lyapunov exponent converges to a single negative value. 
	\item Upon examining the plots for various values of the Lyapunov exponents, one notices that the dark region of negative Lyapunov exponent varies rapidly with the parameters $\theta$ and $\sigma$ whereas the rest of the plot shows slower variations. 
		It is interesting to remark that the slope of the line $\lambda^+=-1$, for example, is close to $0.4$, which is smaller than $1$. 
		That means that in order to balance an increase in the noise for some value, the damping must be increased by a larger value.  
\end{itemize}

To conclude, in light of this numerical result, and along with the theorem for existence of SRB measure, provided $\boldsymbol \sigma_i$ spans $\mathbb R^3$, the proper dissipation of energy will imply the existence of a non-singular random attractor.

The final analysis at this stage of the investigation concerns the nature of the random attractors of this system. 
From numerical simulations, we display in Figure \ref{fig:RB-RA} a realisation of a random attractor of the rigid body.\footnote{See \url{http://wwwf.imperial.ac.uk/~aa10213/} for a video of this random attractor.} 

\begin{figure}[h]
	\centering
	\subfigure{\includegraphics[scale=0.35]{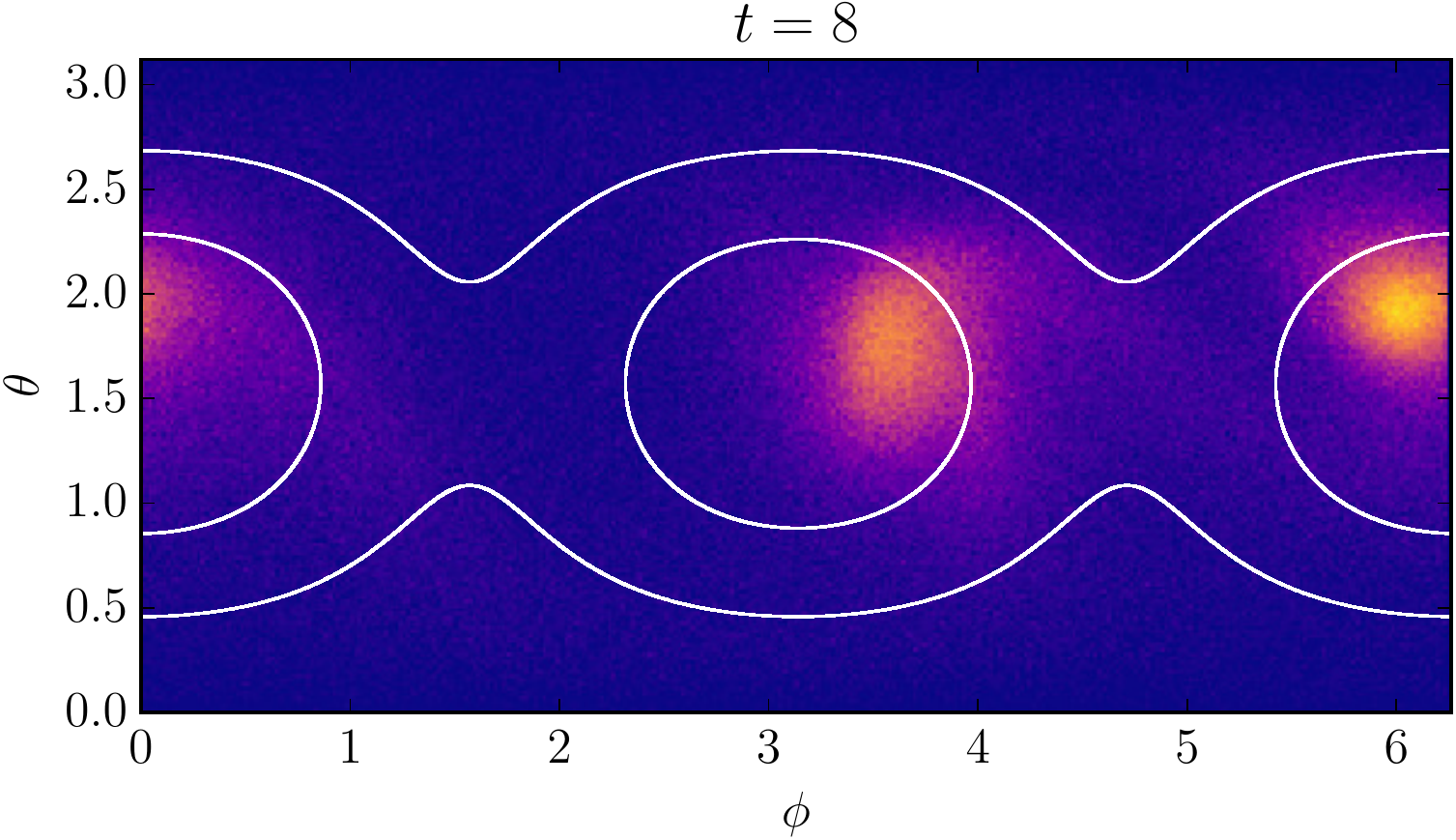}}
	\subfigure{\includegraphics[scale=0.35]{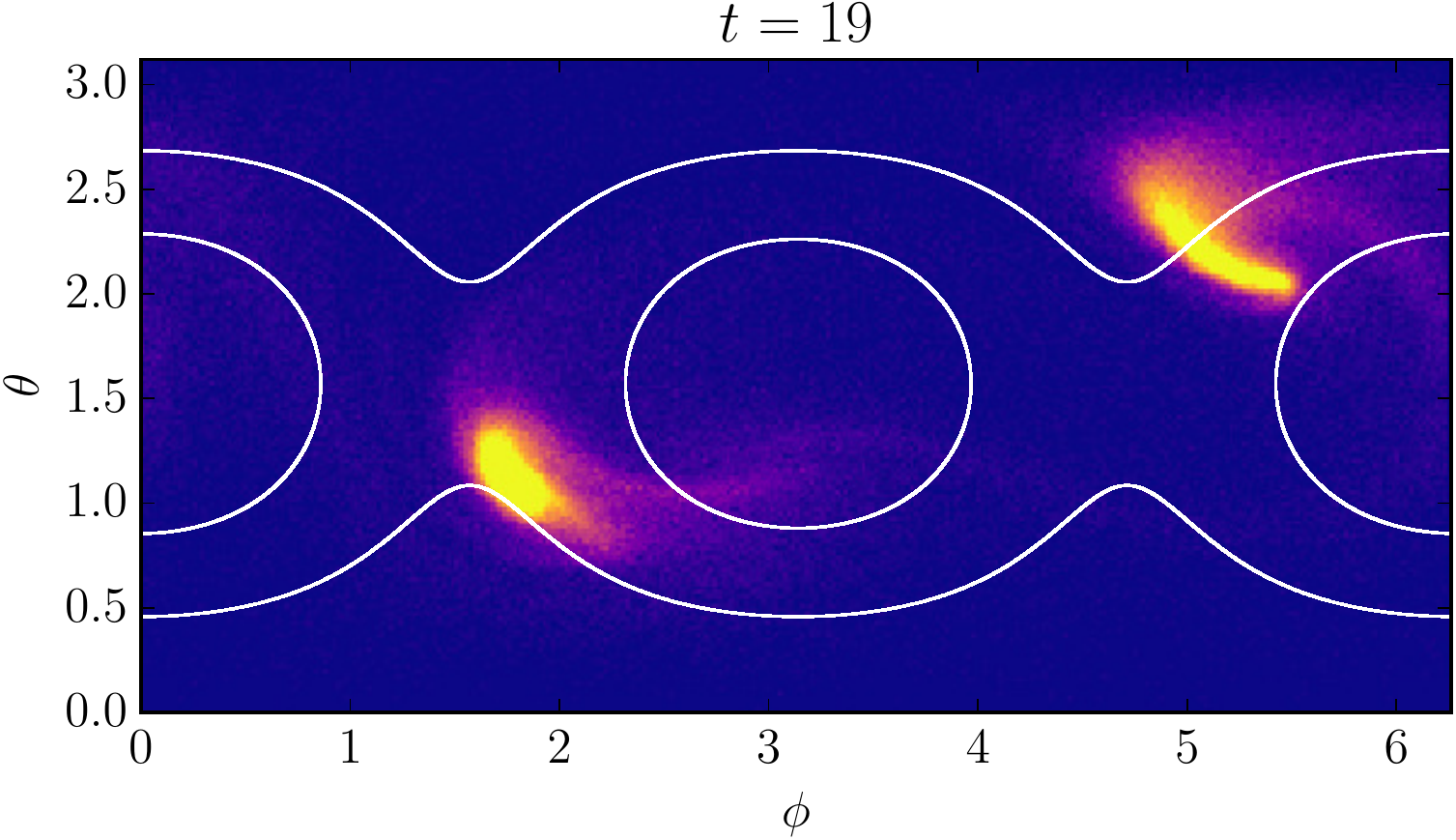}}
	\subfigure{\includegraphics[scale=0.35]{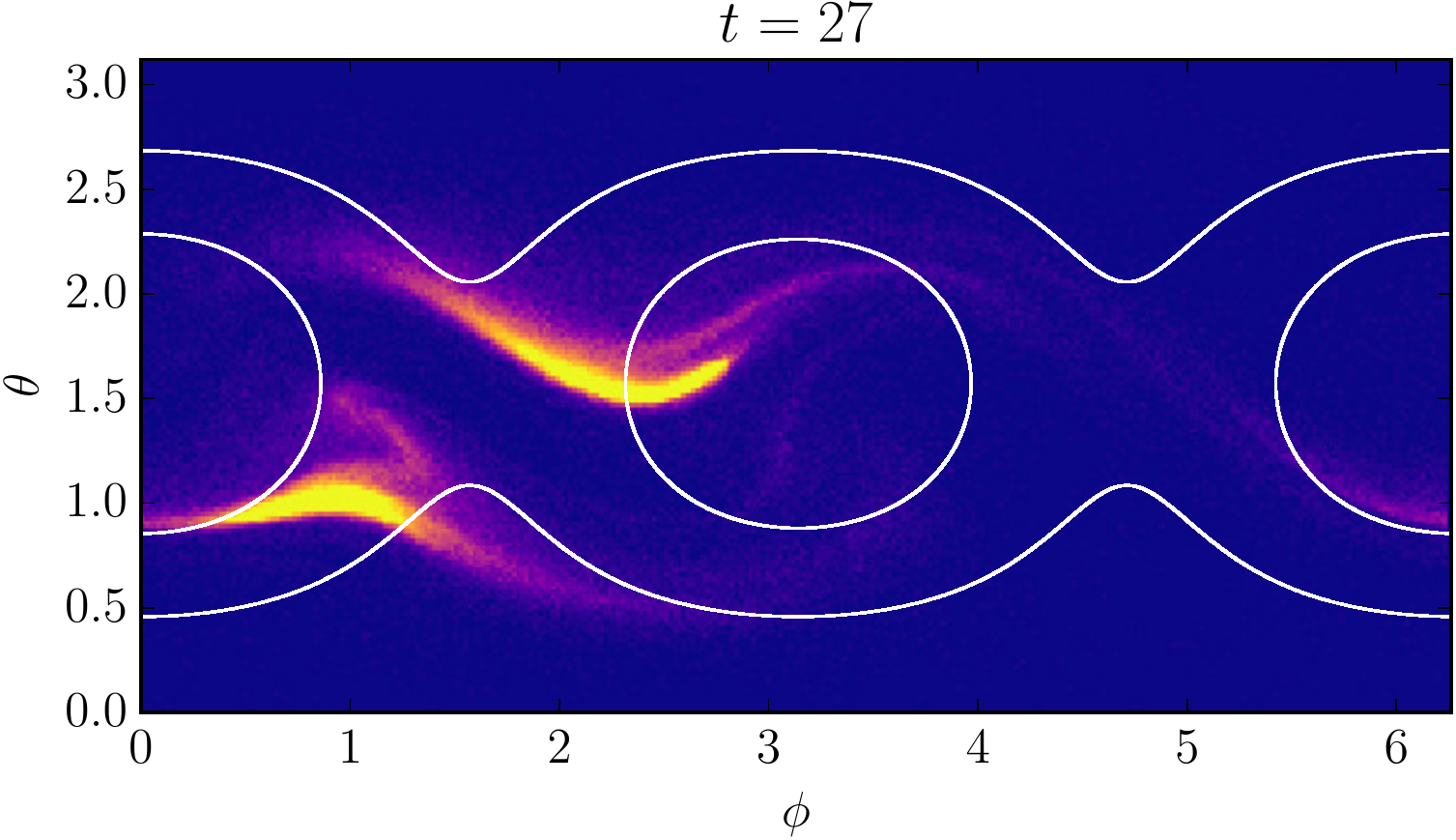}}
	\subfigure{\includegraphics[scale=0.35]{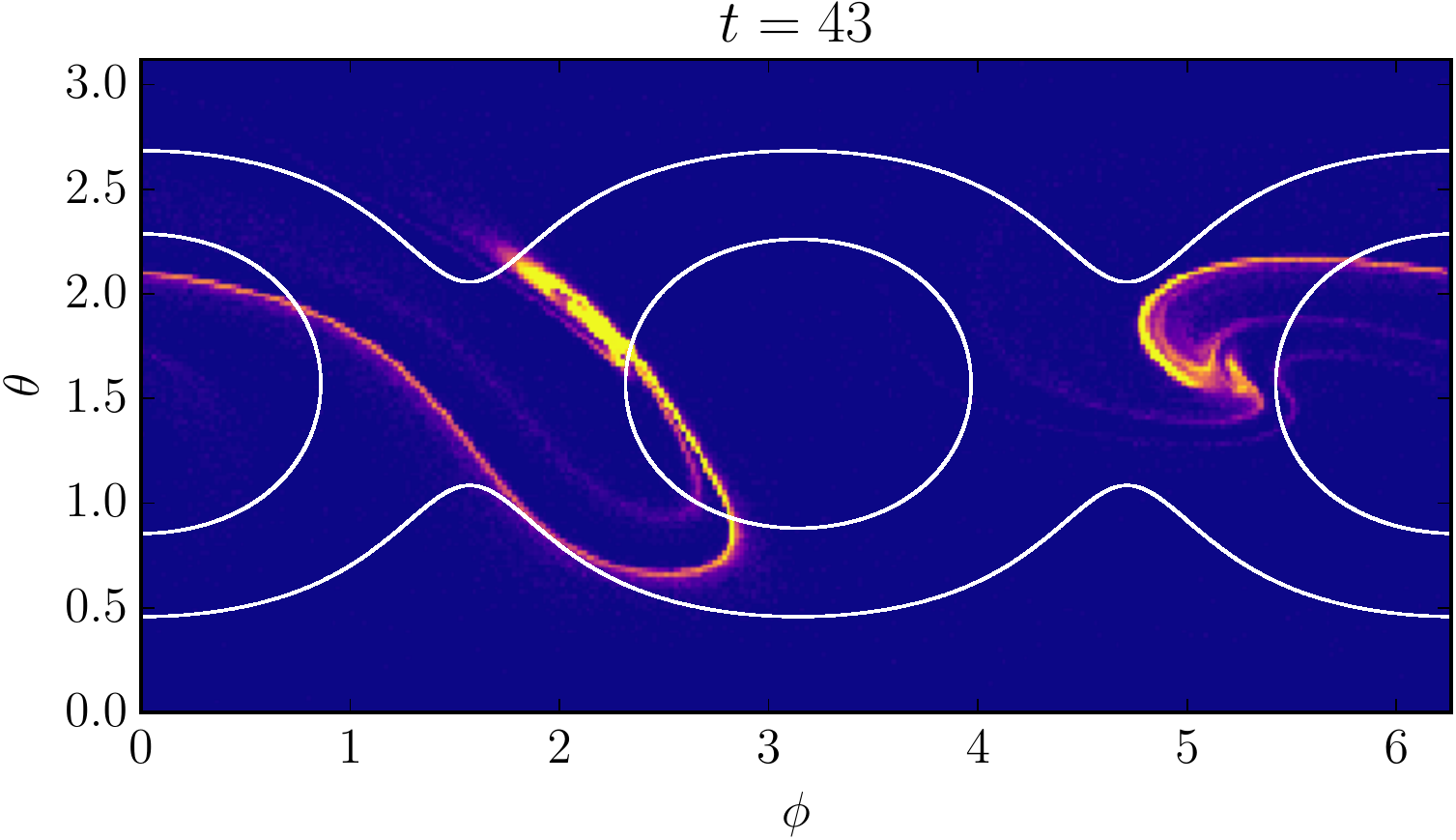}}
	\caption{{\small The four panels display snapshots of the same rigid body random attractor with $\mathbb I= \mathrm{diag}(1,2,3)$, $\theta=0.5$ and $\sigma= 0.5$. 
	The simulation started from a uniform distribution of rigid bodies on the momentum sphere and create finer and finder structures. The color is in log scale and we simulated 400 000 rigid body initial conditions with a split step scheme. }}
	\label{fig:RB-RA}
\end{figure}
The plots in Figure \ref{fig:RB-RA} show the SRB measure, in log scale and exhibit the phenomena of stretching and folding, typical of strange attractors with a positive and negative Lyapunov exponents. 
The positive exponents produce the stretching mechanism and the negative ones produce the folding process. 
Asymptotically in time, these mechanisms may create a fractal structure, similar to the Smale horseshoe structure for the Duffing attractor.
The mechanisms for the creation of the rigid body random attractor can be understood from the underlying rigid body dynamics. 
The heteroclinic orbits, linking the two saddle equilibrium points which correspond to the direction of the second moment of inertia are the longest orbits, with the fastest dynamics along them. 
The speed of the motion for each orbits then decreases to reach the stable equilibrium points, associated to the largest or smallest moment of inertia. 
This change in speed creates a shear on a given non-singular set evolving with the rigid body dynamics. 
Together with the compactness of the sphere, the combination of noise and dissipation produces the complicated structure of the attractor. 
One can remark that the attractor of the Duffing oscillator is similar and provides a good example in the deterministic context of the creation of these structures. 

A more detailed study of this random attractor will certainly be interesting, but is out of the present scope of this work as it would require deeper dynamical systems analysis.
Nevertheless, we will briefly study in the next section a simplification of this model which considers periodic kicking instead of Brownian motion.

\subsection{Periodic kicking}\label{kicking}

We finish this section devoted to the stochastic three dimensional rigid body by making the simple replacement of the noise by a periodic kicking. 
It turns out that this type of forced dynamical systems also shows chaotic behaviour, and furthermore the theoretical understanding of these phenomenon is more advanced that for the pure noise case. 
We refer to \cite{wang2003strange, lin2008shear} and references therein for such studies. 

For us, the periodic kicking is achieved by simply replacing the noise $dW^i_t$ by the sum of dirac delta functions
\begin{align}
	dW_t^i\Rightarrow \sigma_i \sum_{n=1}^\infty \delta ( t- nT), 
\end{align}
where $T$ is the period for the kicking and $\sigma\in \mathbb R^3$ represents the amplitude and direction of the kick, as in the noisy system. 
It is interesting to remark that the kicking corresponds to a rotation around the fixed axis $\sigma$ and with an angle proportional to $\|\sigma \|$. 
Notice that in the noisy case, the axis of rotation was also random, thus it was not relevant to consider it for our studies. 
The double bracket dissipation term is still there, thus some attractors are expected to emerge but they will surely not be random. 
In fact, due to the periodic kicking, the notion of attractor must be modified slightly. 
Recall that in the noisy case, in order to have a fixed attractor in time, we needed the notion of pullback attractor. Here, we will fix the attractor by just observing it at discrete times $nT$. 
Indeed, between each kick, the dynamics relaxes following the damped deterministic rigid body equation. 

We will not attempt an theoretical study of this system but rather illustrate its complexity with the help of numerical investigations. 
Since the parameter space is rather large, we will just highlight the most typical behaviour of the system while only varying the amplitude of the kick, $\|\sigma\|$ for fixed $T=1$, $\theta= 0.2$, $\sigma= \|\sigma\|(1,1,1)$ and $\mathbb I = \mathrm{diag}(1,2,3)$. 
The fact that $\sigma$ is not aligned with any eigenvalue of $\mathbb I$ makes the configuration of the rigid body generic enough for the present study. 
Remarkably, this simple system can undergo many different types of dynamical behaviour by simply varying the kicking amplitude $||\sigma||$. 

We have numerically scanned the attractors for various values of $\|\sigma\|$ and have displayed the results in figure \ref{fig:RB-kick}, which we will analyse qualitatively below. 
\begin{figure}[htpb]
	\centering
	\subfigure{\includegraphics[scale=0.35]{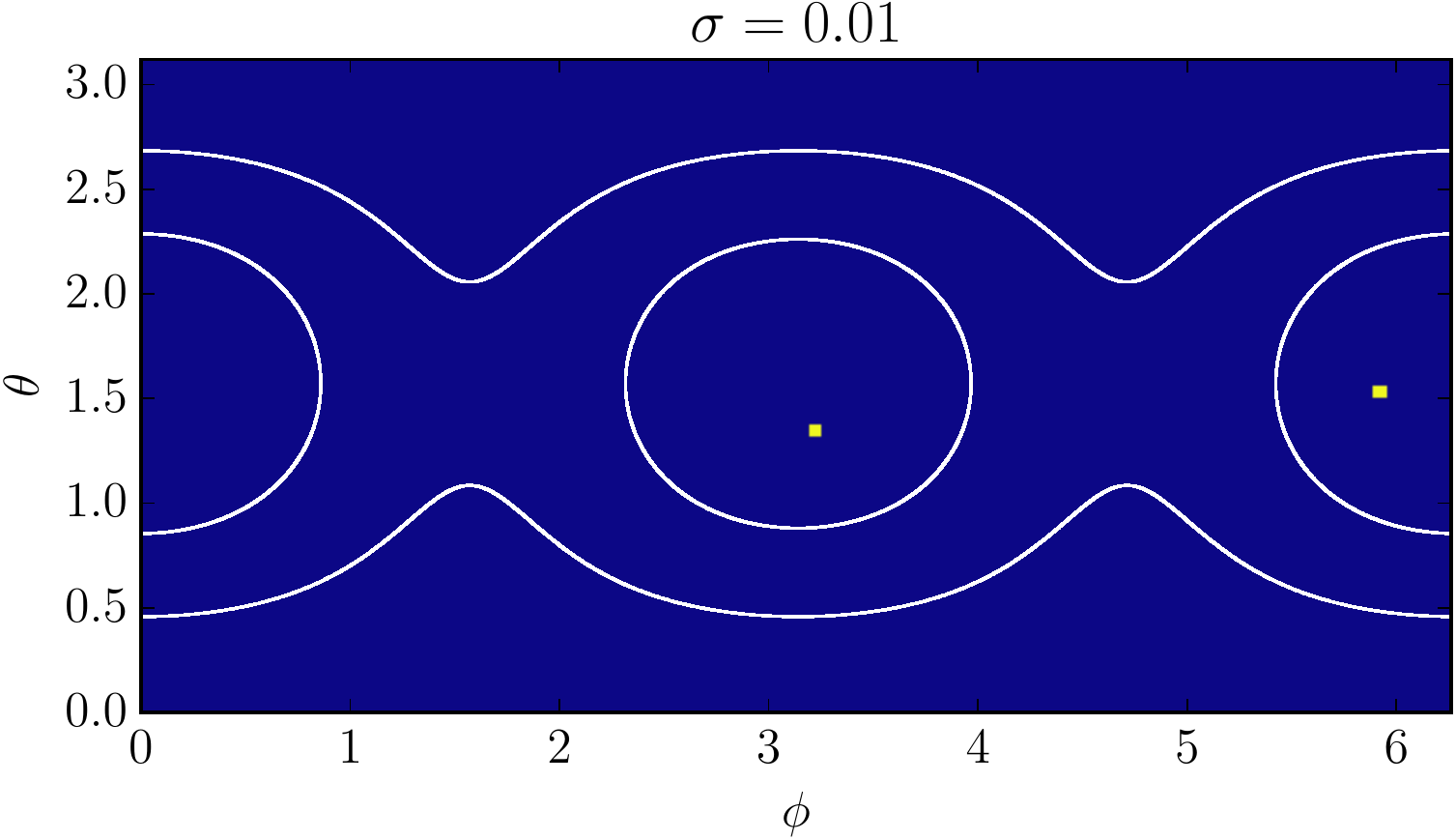}}
	\subfigure{\includegraphics[scale=0.35]{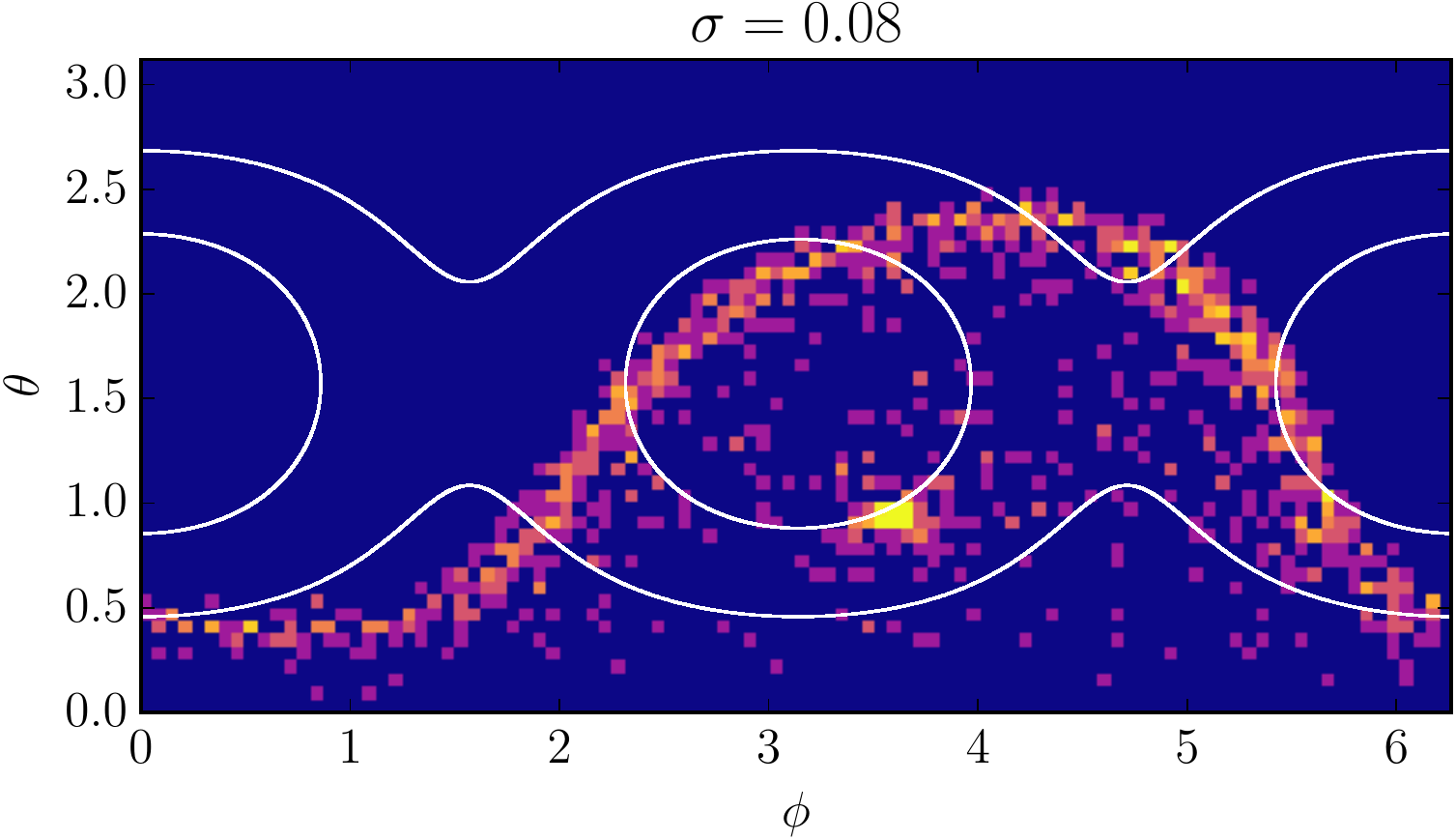}}
	\subfigure{\includegraphics[scale=0.35]{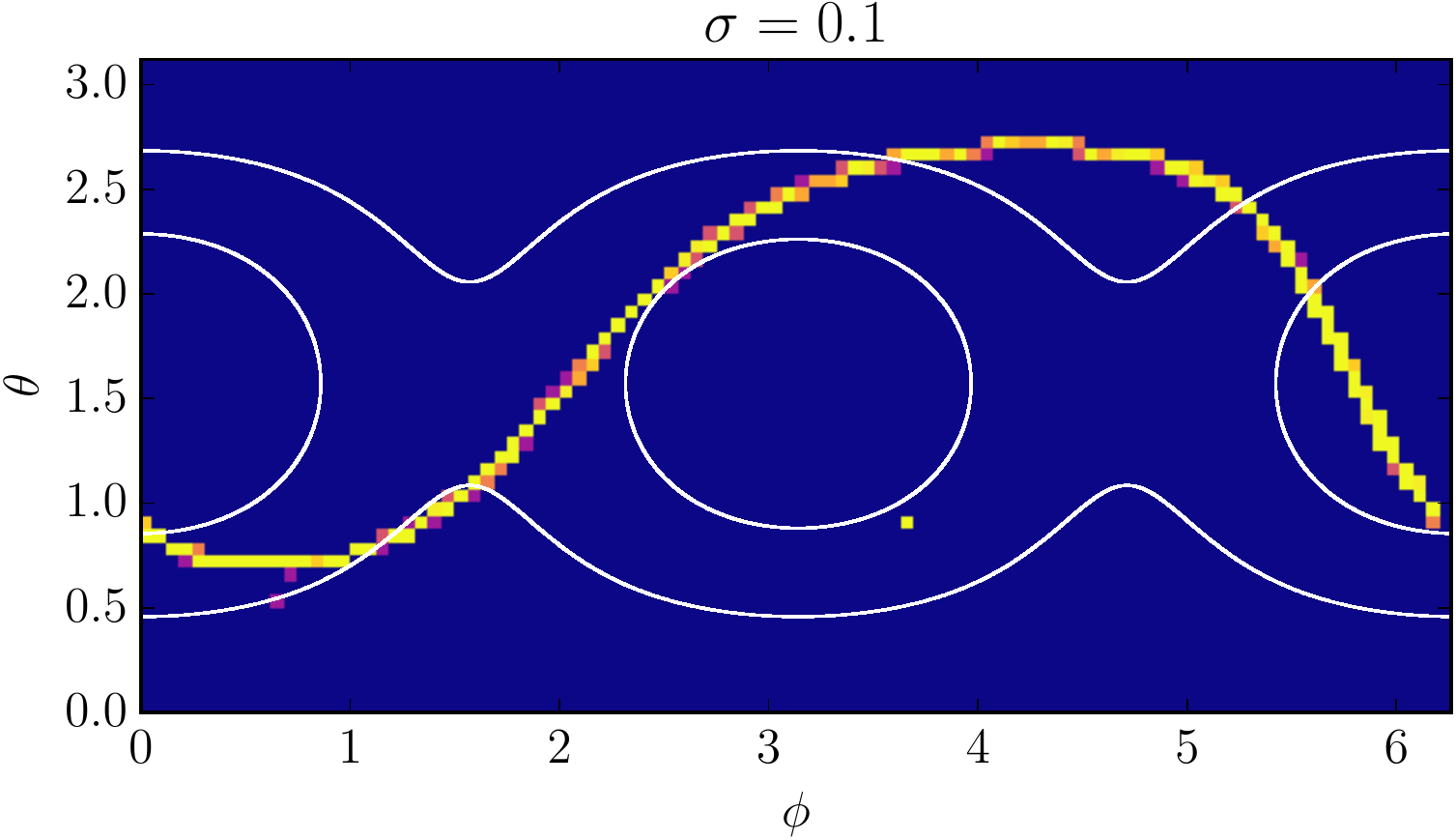}}
	\subfigure{\includegraphics[scale=0.35]{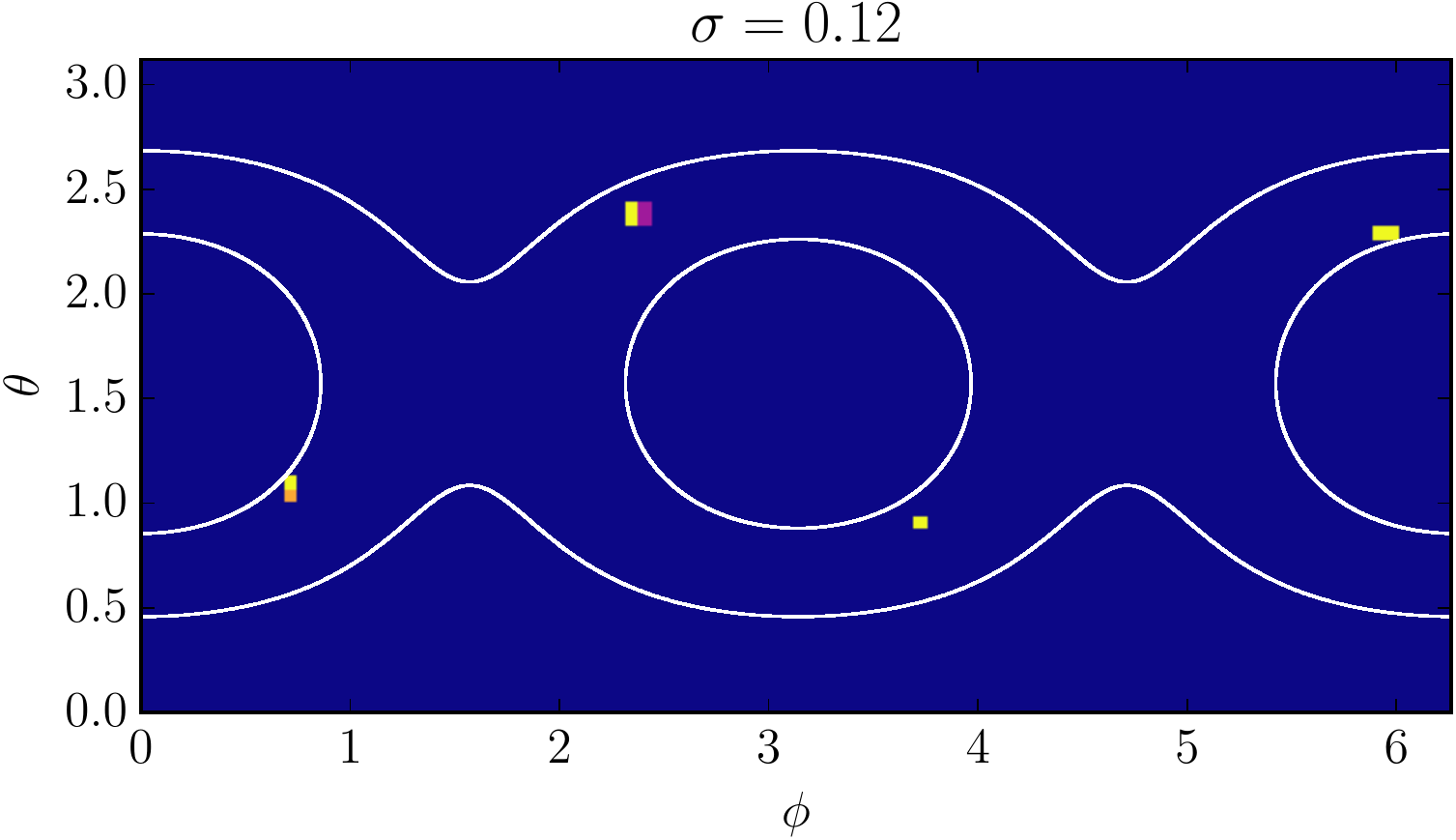}}
	\subfigure{\includegraphics[scale=0.35]{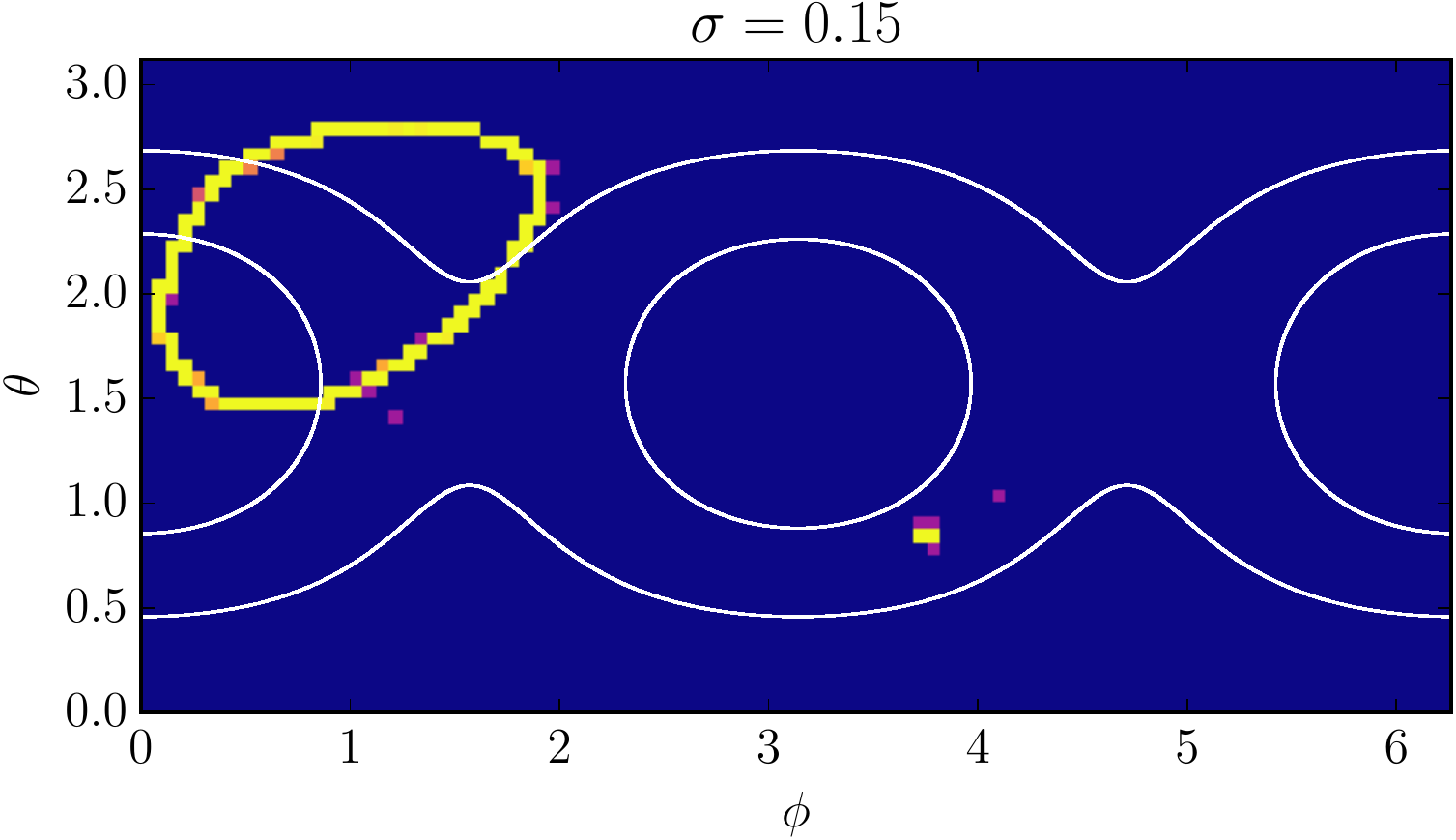}}
	\subfigure{\includegraphics[scale=0.35]{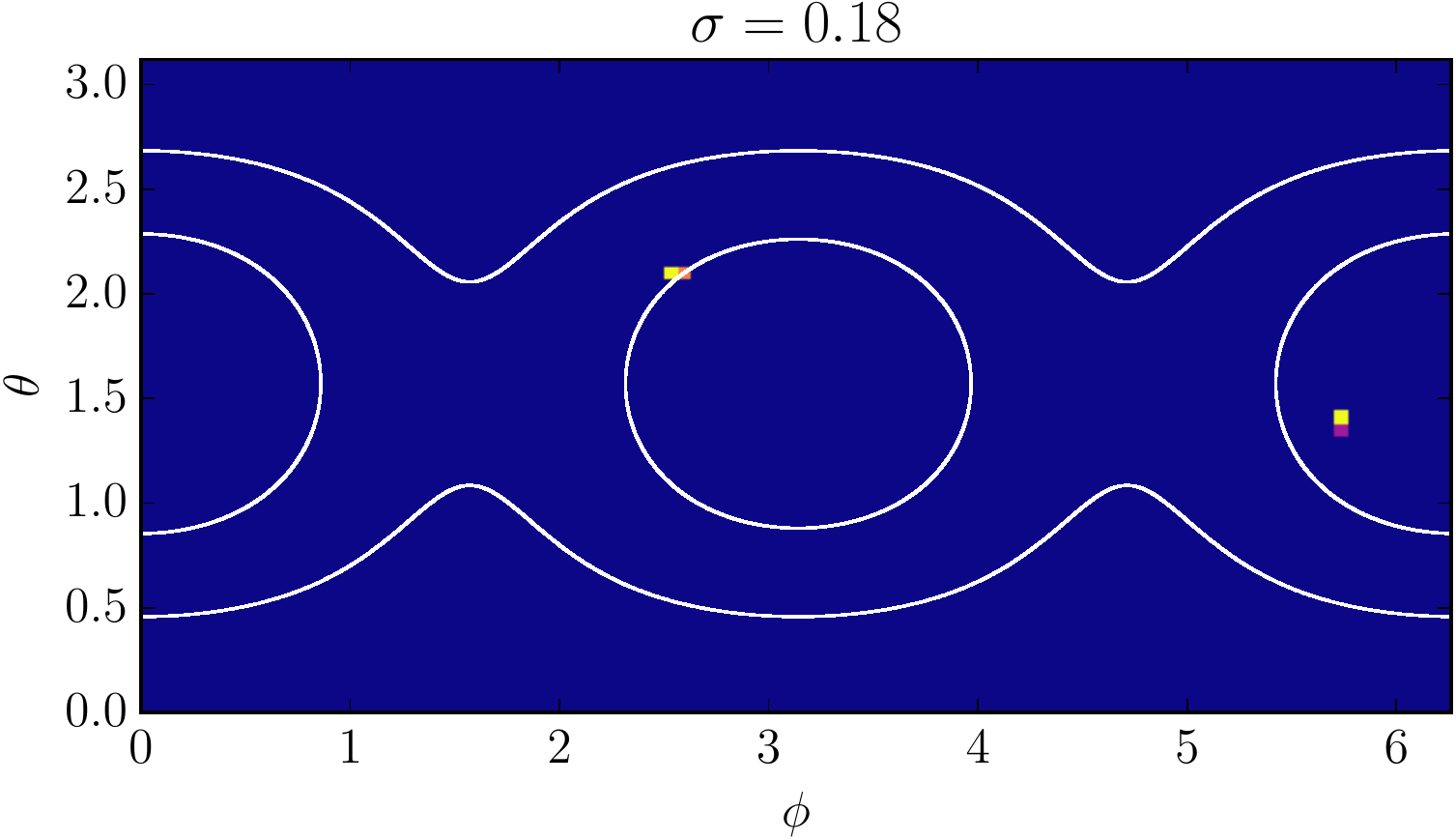}}
	\caption{ {\small This figure displays several attractors of the periodically kicked rigid body for the different regimes parametrised by the kicking amplitude $\|\sigma\|$ as described in the text.  
	Each snapshot is taken after the system has reached an equilibrium position, except for $\|\sigma\|=0.8$.
	In this case, we have illustrated the transient limit cycle structure which quickly disappears as the solution  converges to a singular attractor.  
For $\sigma=0.12$, the three singular attractors (except the lower right one) which are slightly thicker are in fact a single attractor and the kicking makes these points periodically switch amongst each other. 
The last panel with $\sigma=0.18$ shows a similar behaviour for a single attractor composed of two points that switch between each other. }}
	\label{fig:RB-kick}
\end{figure}

The first thing to remark in order to understand these different types of behaviour is that the kick will rotate the momentum of the rigid body in the same direction as the original rigid body flow for the lower right region, and in the opposite direction for the upper left region in figures \ref{fig:RB-kick}. 
For this reason, the effect of the periodic kicking will be different in the opposite sides of the sphere and this will lead to two different types of attractors near these regions. 
We now describe the different types of attractors which appear in this system, upon varying the kicking amplitude.
Let $\mathcal A^+$ denote the attractor where the dynamical flow and the kicking are in the same direction and let $\mathcal A^-$ denote the other attractor. 
They will be understood as being the same if only one attractor emerges. 
Recall that all the parameters are fixed, except the kicking amplitude $\|\sigma\|$. 
We thus classify the different regimes according to an ordered sequence $\epsilon_0<\ldots <\epsilon_5$ of values of $\|\sigma\|$, that we will find later. 

The analysis was done by looking at individual trajectories of rigid bodies and the values of the Lyapunov exponents, which we display in figure \ref{fig:lyap-kick}.
\begin{figure}[htpb]
	\centering
	\includegraphics[scale=0.5]{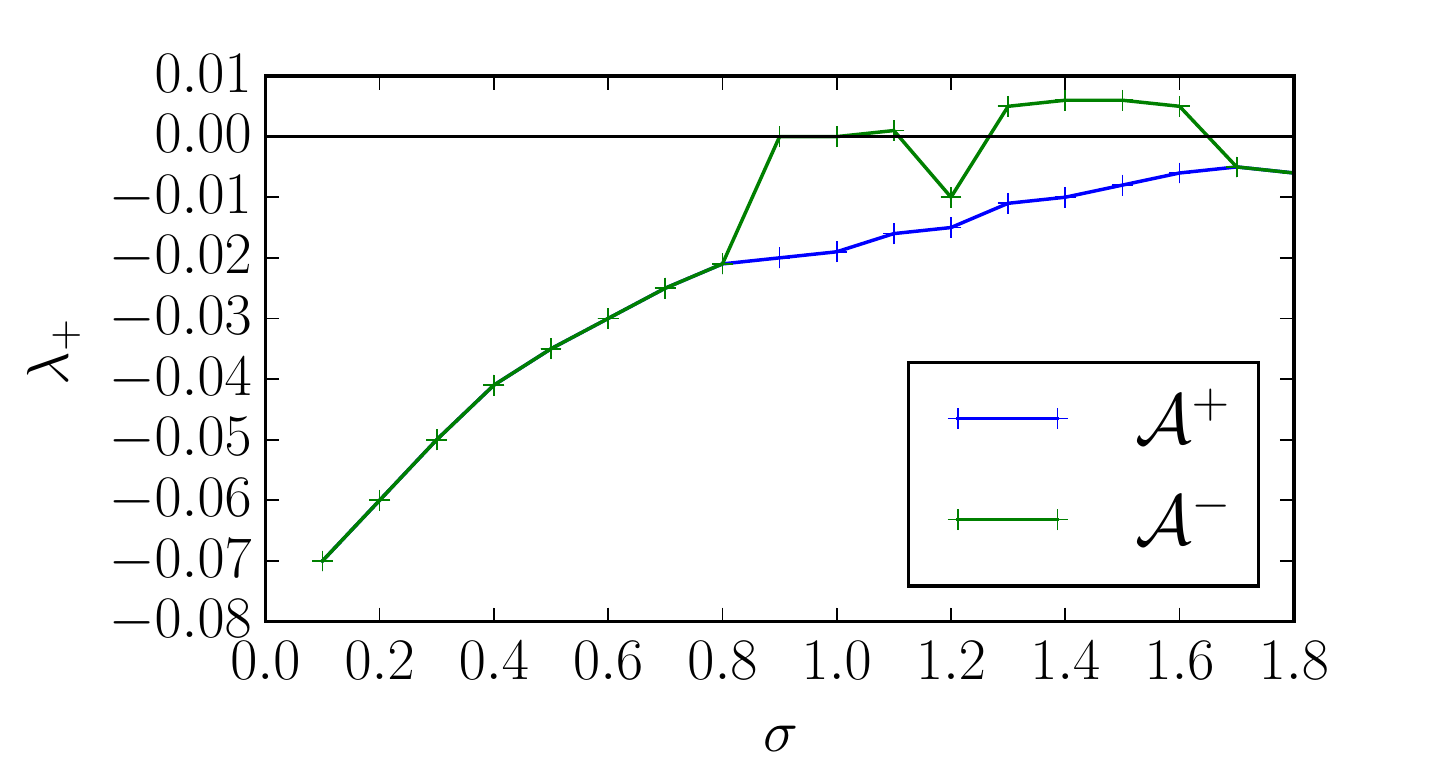}
	\caption{{\small We display the result of the numerical computation of the top Lyapunov exponent for the kicked rigid body as a function of the kicking amplitude $\|\sigma\|$ for the possibly two attractors.  
	The $\mathcal A^+$ attractor is always singular, thus always has a negative top Lyapunov exponent whereas the other attractor sometimes shows chaos on the limit cycle. }}
	\label{fig:lyap-kick}
\end{figure}
From all this data we can estimate the values of the $\epsilon$ in the ordered sequence as approximatively $(0.2,0.8,1.1,1.3,1.7)$, as illustrated in figure \ref{fig:lyap-kick}.

The different regimes are as follow.
\begin{enumerate}
	\item $\sigma< \epsilon_0$: Both attractors $\mathcal A^+$ and $\mathcal A^-$ are singular and near the original rigid body equilibrium positions. 
	\item $\epsilon_0<\sigma<\epsilon_1$: The points near the region of the previous $\mathcal A^-$ eventually reach the $\mathcal A^+$ region to form one singular attractor near the previous $\mathcal A^+$. 
	Before the collapse of point near $\mathcal A^-$ to the region near $\mathcal A^+$, we observe a short transient chaos, close to being a limit cycle. 
	The chaos is revealed by a positive Lyapunov exponent during this period. 
	\item $\epsilon_1<\sigma<\epsilon_2$: The attractor $\mathcal A^-$ is a limit cycle which passes near both original  equilibrium of minimum energy, and thus is driven by the kicks on both sides of the sphere.  
	The chaos on this limit cycle is not clear from the Lyapunov exponent computation, as the top Lyapunov exponent is very close to $0$, see figure \ref{fig:lyap-kick}. 
		The other attractor $\mathcal A^+$ is singular, near its previous position. 
	\item $\epsilon_2<\sigma<\epsilon_3$: The attractor $\mathcal A^-$ consists of $3$ singular points at roughly equal distance from the kicking axis $(-1,-1,-1)$. 
		These three points form a single attractor as the kick makes then periodically switch between themselves.  The other attractor $\mathcal A^+$ is still singular. 
	\item $\epsilon_3<\sigma<\epsilon_4$ : The attractor $\mathcal A^-$ consists of a chaotic limit cycle centred around the kicking axis $(-1,-1,-1)$  and $\mathcal A^+$ is still singular. 
		As compared to the previous limit cycle, this one remains near the region where the kicking is opposite to the flow direction and has a stronger chaos, as seen from the Lyapunov exponent computation, see figure \ref{fig:lyap-kick}. 
	\item $\epsilon_4<\sigma<\epsilon_5$ : The last region explored here shows that both attractors merge to a single attractor that consists of two fix points.  
		The periodic kicking switches one to the other, as in case $(4)$.  
\end{enumerate}

From these findings, the most remarkable result is not the existence of chaos on the limit cycles, but rather the existence of the limit cycles themselves. 
The chaos can be understood in the same way as for the stochastic case, namely by the shear of the system. 
The existence of a stable limit cycle is in fact rather subtle as it requires a fine balance between the kicking, the shear and the damping of the rigid body. 
A precise analytical estimation for the emergence of such limit cycles is of course out of the scope of this work and  we leave it for further studies. 

\section{Semidirect product example: the stochastic heavy top}\label{HT}
The basic example of semidirect product motion is the heavy top, which arises in the presence of gravity, when the support point of a freely rotating rigid body is no longer at its centre of mass. The starting phase space for the heavy top is $T^\ast {\rm SO}(3)$, just as for the free rigid body.  When the support point is shifted away from the centre of mass, gravity breaks the symmetry, and the system is no longer ${\rm SO}(3)$ invariant. Consequently, the motion can no longer be written entirely in terms of the body angular momentum $\boldsymbol{\Pi}\in \mathfrak{so}(3)^*$.
One also needs to keep track of the unit vector $\boldsymbol{\Gamma}$, the ``direction of gravity'' as seen from the body $(\boldsymbol{\Gamma}  = \mathbf{R}^{-1}\mathbf{k} $ where the unit vector $\mathbf{k}$ points upward in space and $\mathbf{R}$ is the element of ${\rm SO}(3) $ describing the current configuration of the  body). The variable $\boldsymbol{\Gamma}$ may be identified with elements in the coset space $SO(3)/SO(2)$, where $SO(3)$ is the symmetry broken by introducing a special vertical direction for gravity, and $SO(2)$ is the remaining symmetry. This $SO(2)$ is the isotropy subgroup of $SO(3)$ corresponding to rotations around the unit vector $\mathbf{k}$ which leave the direction of gravity invariant. 
\\

\subsection{The stochastic heavy top}

The Lagrangian for the heavy top is the difference of the kinetic energy and the work against gravity, where the fixed vector $\boldsymbol \chi$ represents the position of the centre of mass of the body with respect to the fixed point. In body coordinates, the reduced Lagrangian is
\begin{align}
	l(\boldsymbol \Omega, \boldsymbol \Gamma) = \frac{1}{2} \boldsymbol \Omega\cdot  \mathbb{I}\boldsymbol \Omega - mg\boldsymbol \Gamma\cdot \boldsymbol \chi\, .
	\label{HT-lag}
\end{align}
We refer to see \cite{holm1998euler, marsden1999intro} for a complete description of the semidirect product reduction  for the heavy top, which we will not explain here.  
The stochastic potential will be taken to be linear in both the $\boldsymbol \Gamma$ and $\boldsymbol \Pi$:
\begin{align}
	\Phi_i(\boldsymbol \Gamma,\boldsymbol \Pi) = \boldsymbol{\sigma}_i\cdot \boldsymbol \Pi+\boldsymbol{\eta}_i\cdot \boldsymbol \Gamma\,,
\end{align}
where $\boldsymbol \sigma_i$ and $\boldsymbol \eta_i$ need not span $\mathbb R^3$.
The stochastic process describing the stochastic heavy top is then 
\begin{align}
	\begin{split}
	d\boldsymbol \Pi &+(\boldsymbol \Omega dt + \sum_i \boldsymbol{\sigma}_i\circ dW^i_t) \times\boldsymbol \Pi +mg (\boldsymbol \Gamma \times \boldsymbol \chi)dt+\sum_i mg(  \boldsymbol \Gamma \times \boldsymbol{\eta}_i)\circ dW^i_t =0 \,,\\
d\boldsymbol \Gamma &+(\boldsymbol \Omega dt + \sum_i\boldsymbol{\sigma}_i\circ dW^i_t)\times\boldsymbol \Gamma= 0 \,,
	\end{split}
\end{align}
and the corresponding It\^o process is 
\begin{align}
	\begin{split}
		d\boldsymbol \Pi &+ (\boldsymbol \Omega dt + \sum_i \boldsymbol{\sigma}_i dW^i_t) \times\boldsymbol \Pi + (\boldsymbol \Gamma\times mg\boldsymbol \chi)dt  \\
		&+ \sum_i mg( \boldsymbol \Gamma\times \boldsymbol{\eta}_i)\circ dW^i_t-\frac12\sum_i\boldsymbol{\sigma}_i\times(\boldsymbol{\sigma}_i\times \boldsymbol \Pi)dt=0 \,,\\
		d\boldsymbol \Gamma &+(\boldsymbol \Omega dt + \sum_i\boldsymbol{\sigma}_i dW^i_t)\times \boldsymbol \Gamma -\frac12 \sum_i\boldsymbol{\sigma}_i\times(\boldsymbol{\sigma}_i\times \boldsymbol \Gamma)dt=0 \,.
	\end{split}
\end{align}

The two Casimirs of the heavy top are conserved, $\|\boldsymbol \Gamma\|^2=k$ and $\boldsymbol\Pi\cdot\boldsymbol \Gamma=c$. However, the energy is not conserved, as it satisfies the following stochastic process
\begin{align}
	\begin{split}
	\frac{d}{dt} E 
	&= \frac14 \sum_i \left [\boldsymbol \Omega\cdot (\boldsymbol{\sigma}_i\times (\boldsymbol{\sigma}_i\times \boldsymbol \Pi)) +  \boldsymbol \Pi\cdot \mathbb I^{-1} ( \boldsymbol{\sigma}_i\times (\boldsymbol{\sigma}_i\times \boldsymbol \Pi))\right ]dt\\
	&+ \frac12 \sum_i \left [ (\boldsymbol \Pi\times \boldsymbol{\sigma}_i)\cdot \mathbb I^{-1} (\boldsymbol \Pi\times \boldsymbol{\sigma}_i) - mg (\boldsymbol{\sigma}\times \boldsymbol \Gamma)\cdot (\boldsymbol \chi\times \boldsymbol{\sigma}_i) \right ] dt \\
	&+ \frac12 \sum_i\left [ \boldsymbol \Omega\cdot (\boldsymbol \Pi\times \boldsymbol{\sigma}_i) +  \boldsymbol \Pi\cdot \mathbb I^{-1} (\boldsymbol \Pi\times \boldsymbol{\sigma}_i) + 2 \boldsymbol \chi\cdot (\boldsymbol \Gamma\times \boldsymbol{\sigma}_i)\right ] dW^i_t \,.
	\end{split}
	\label{HT-energy}
\end{align}
The energy being only bounded from below, this stochastic process can lead to arbitrary large value for the energy, over a long enough time. 

\subsection{The integrable stochastic Lagrange top}
When $\mathbb I$ is of the form $\mathbb I=\mathrm{diag}(I_1,I_1,I_3)$ and $\chi=(0,0,\chi_3)$, the deterministic heavy top is called the Lagrange top and is integrable.
The integrability comes from the extra conserved quantity $\boldsymbol \Pi\cdot \boldsymbol \chi$, in this case.
For noise, the stochastic process for this quantity is 
\begin{align}
	\frac{d}{dt}(\boldsymbol \Pi\cdot \boldsymbol \chi) = -\frac12 \sum_i(\boldsymbol \chi\times\boldsymbol{\sigma}_i)\cdot (\boldsymbol{\sigma}_i\times \boldsymbol \Pi)\,  dt - \sum_i\boldsymbol \chi\cdot (\boldsymbol \Pi\times \boldsymbol{\sigma}_i)dW_t^i,
\end{align}
which is not a conserved quantity in general.
However, the form of this equation implies that if one selects $\boldsymbol \sigma_i=\boldsymbol \chi$ then $\boldsymbol \Pi\cdot \boldsymbol \chi$ is a conserved quantity. 
It is remarkable that with this choice of noise, the energy is also a conserved quantity, as one can check from equation \eqref{HT-energy}. 
We thus have a stochastic integrable Lagrange top, with a stochastic Lax pair given by
\begin{align}
d(\lambda^2\boldsymbol \chi + \lambda \boldsymbol \Pi + \boldsymbol \Gamma) = ((\lambda \boldsymbol \chi  + \boldsymbol \Omega )dt + \boldsymbol \chi\circ dW)\times (\lambda^2\boldsymbol \chi + \lambda \boldsymbol \Pi + \boldsymbol \Gamma),
\end{align}
where $\lambda$ is arbitrary and called a spectral parameter. 
We refer to \cite{ratiu1981euler} for more details about the integrability of the Lagrange top.
Following the framework of integrable hierarchies, further developed for infinite dimensional integrable hierarchies in \cite{arnaudon2015integrable}, there exists another integrable stochastic Lagrange top where the stochastic potential is the same as the Hamiltonian. 
The explanation for the integrability is straightforward, as the change of variable $t\to t+W_t$ maps the stochastic Lagrange top to the deterministic one; so we will not discuss it in more detail here.
\begin{figure}[htpb]
	\centering
	\subfigure[$\Gamma$  motion]{ \includegraphics[scale=0.45]{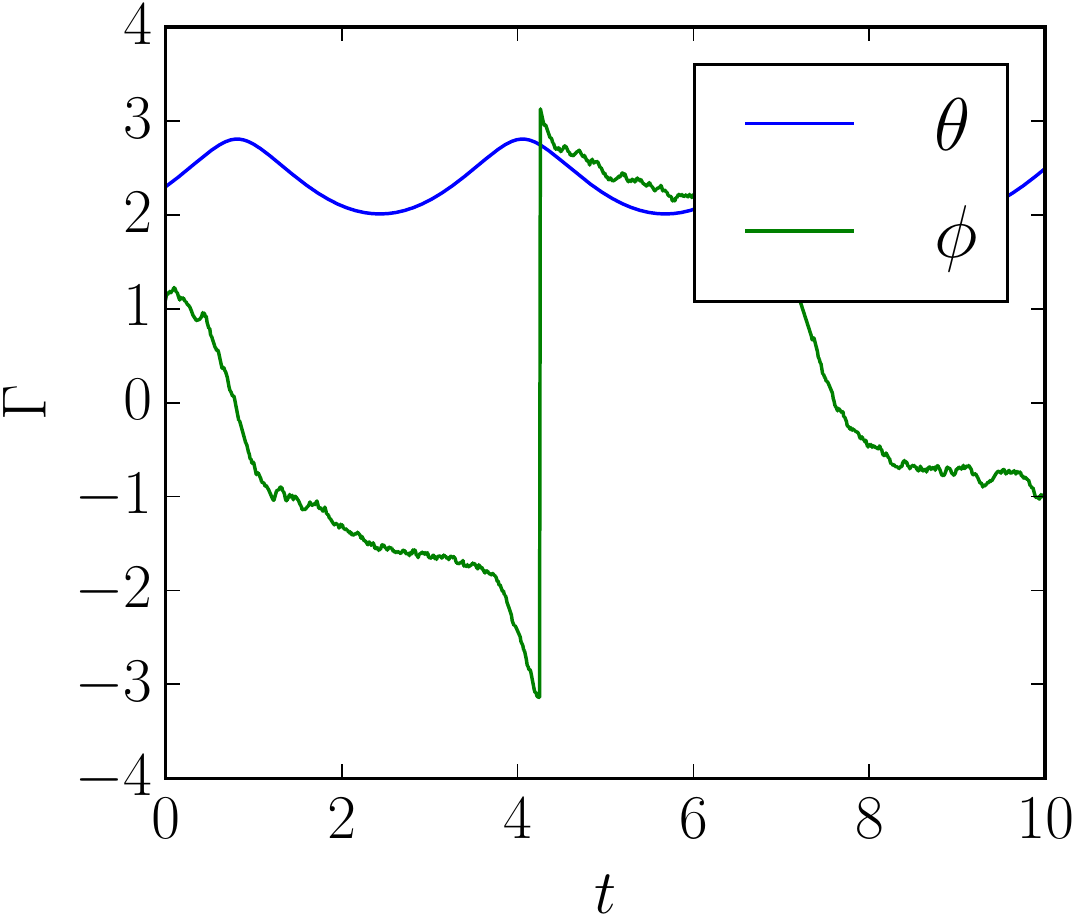}}
	\subfigure[Conserved quantities]{ \includegraphics[scale=0.45]{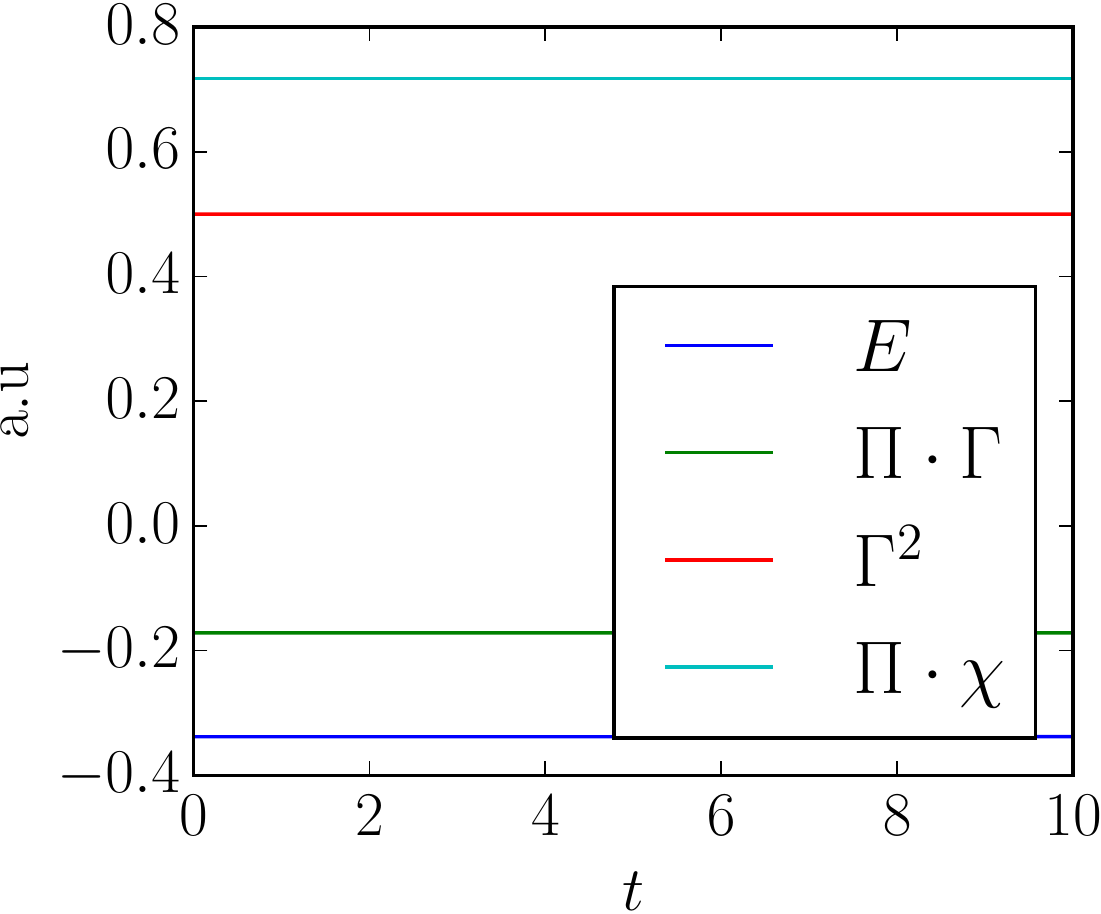}}
	\caption{{\small This figure displays a realisation of the motion of the integrable stochastic Lagrange top.  The conserved quantities are displayed in the right panel. }}
	\label{fig:lagrange}
\end{figure}

We want to study this stochastic system further, as integrability means that an explicit solution can be found. 
Indeed, from the standard theory of the heavy top, see for example \cite{arnold89mechanics,audin1996tops}, the equation for $\Gamma_3$ can be found to be of the form $\dot \Gamma_3^2= f(\Gamma_3)$, where $f$ depends only on the constants of motion $k$ and $c$. 
Then, a straightforward calculation with Euler angles gives 
\begin{align}
	\begin{split}
	\dot \psi  &= \frac{c-k\Gamma_3  }{(1- \Gamma_3^2) I}\\
	d \phi  &= \left [ \frac{c}{I_3\Gamma_3} - \frac{c-k\Gamma_3  }{I_3\Gamma_3(1- \Gamma_3^2) I} ((1-\Gamma_3^2) I  -  I_3\Gamma_3^2)\right ]dt -\chi_3\circ dW\,,
	\end{split}
\end{align}
where $\mathrm{cos}(\theta)= \Gamma_3$ gives the third Euler angle. 
Surprisingly, only $\phi$ has a stochastic motion, while $\psi$ and $\theta$ follow the deterministic Lagrange top motion. 
This is illustrated in Fig.~\ref{fig:lagrange} via a numerical integration of the stochastic Lagrange top equations.
The conservation of all the Lagrange top quantities is reproduced, as well as the fact that the noise only influences the $\phi$ component of the Euler angles.

\subsection{The Fokker-Planck equation and invariant measures}
We now analyse the associated Fokker-Planck equation for the stochastic heavy top, which is given by
\begin{align}
	\begin{split}
	\frac{d}{dt}\mathbb P(\boldsymbol \Pi, \boldsymbol \Gamma) &= (\boldsymbol \Pi\times \boldsymbol \Omega)\cdot \nabla_\Pi \mathbb P  + (\boldsymbol \Gamma\times \boldsymbol \Omega)\cdot \nabla_\Gamma \mathbb P \\
	&+ \sum_i\frac12(\boldsymbol \Pi\times \boldsymbol{\sigma}_i)\cdot\nabla_\Pi  \left [(\boldsymbol \Pi\times \boldsymbol{\sigma}_i)\cdot  \nabla_\Pi\mathbb P \right ]  \\
	&+  \sum_i \frac12 (\boldsymbol \Gamma\times \boldsymbol{\sigma}_i)\cdot \nabla_\Gamma \left [(\boldsymbol \Gamma\times \boldsymbol{\sigma}_i)\cdot \nabla_\Gamma\mathbb P\right ] \\
	&+ \sum_i\frac12(\boldsymbol \Pi\times \boldsymbol{\sigma}_i)\cdot\nabla_\Pi  \left [(\boldsymbol \Gamma\times \boldsymbol{\sigma}_i)\cdot  \nabla_\Gamma \mathbb P \right ]  \\
	&+   \sum_i\frac12 (\boldsymbol \Gamma\times \boldsymbol{\sigma}_i)\cdot \nabla_\Gamma \left [(\boldsymbol \Pi\times \boldsymbol{\sigma}_i)\cdot\nabla_\Pi \mathbb P\right ],  
	\end{split}
	\label{BK-HT}
\end{align}
where in our notation $\nabla_\Pi$ denotes the gradient with respect to the $\boldsymbol \Pi$ variable only and similarly for $\nabla_{\Gamma}$.
By using the semidirect product Lie-Poisson structure of the heavy top
\begin{align}
	\{H,G\}_{HT} := 	
	\begin{bmatrix}
		\nabla_\Pi H &\nabla_\Gamma H
	\end{bmatrix}
	\begin{bmatrix}
		\boldsymbol \Pi\times & \boldsymbol \Gamma \times\\
		\boldsymbol \Gamma\times & 0 
	\end{bmatrix}
	\begin{bmatrix}
		\nabla_\Pi G\\
		\nabla_\Gamma G
	\end{bmatrix},
	\label{LP-HT}
\end{align}
the Fokker-Planck equation \eqref{BK-HT} can be written in the double bracket form
\begin{align}
	\frac{d}{dt}\mathbb P = \{ h,\mathbb P\}_{HT} +\frac12 \{\Phi,\{\Phi,\mathbb P \}_{HT}\}_{HT}, 
\end{align}
where $h(\boldsymbol \Pi, \boldsymbol \Gamma)$ is the Legendre transform of \eqref{HT-lag}. 

Recall that  the invariant marginal distribution on the $\Gamma$ sphere is constant. 
We study here the distribution in the $\Pi$ coordinate, following the general argument of Theorem \ref{SD-invariant}, which gives the bound
\begin{align}
	0\leq \|\Pi\|(t)\leq \|\Pi_0\| + (mgc) t.
	\label{HT-momentum-bound}
\end{align}
This bound increases linearly with time and is unbounded only when $t\to \infty$.
This effect is clearly illustrated in the Figure \ref{fig:HT-momentum} where the probability distribution of $\|\Pi\|^2$ is plotted. 
The initial conditions are uniform distribution on the $\Gamma$ sphere and a single position for all the momentum, with unit norm. 
\begin{figure}[htpb]
	\centering
	\includegraphics[scale=0.5]{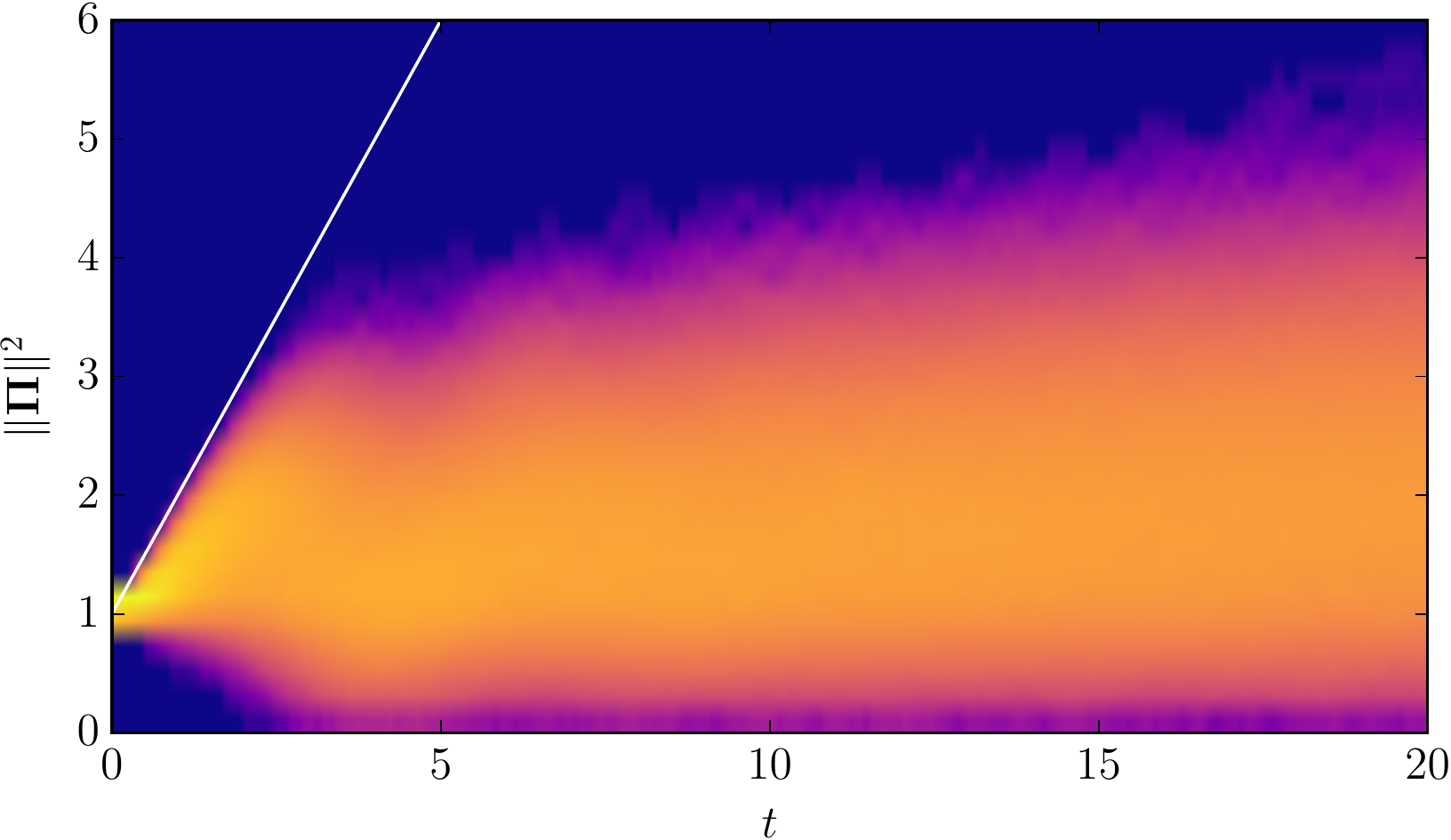}
	\caption{ {\small We display the probability distribution of the norm of the momentum of the heavy top, as a function of time. 
		The distribution tends to $0$ as time goes to $\infty$, but only linearly as shown by equation \eqref{HT-momentum-bound} and the white line in this Figure. 
	The expansion is larger for small time, as the distribution is not yet uniform on the angles of the momentum but linearly bounded in time. 
	After this rapid early expansion, the diffusion slows considerably.}}
	\label{fig:HT-momentum}
\end{figure}
Our system parameters are $m=g=c=1$. Consequently, the linear bound is directly proportional to the time. 
According to Figure \ref{fig:HT-momentum}, the bound is reached almost immediately in the first stage of the diffusion, where the $\Gamma$ and $\Pi$ sphere are not yet uniformly covered. 
After this first short temporal regime, however, the diffusion rate slows considerably below this linear bound. 

\subsection{Random attractor}

The dissipative heavy top equations can be computed directly from the semidirect theory (see also \cite{bloch1996euler}) and in Stratonovich form they read, when the Casimir $\boldsymbol \Pi\cdot\boldsymbol  \Gamma$ is used, 
\begin{align}
	\begin{split}
	d\boldsymbol \Pi &+(\boldsymbol \Omega dt + \sum_i \boldsymbol{\sigma}_i\circ dW^i_t) \times\boldsymbol \Pi +mg (\boldsymbol \Gamma \times \boldsymbol \chi)dt \\
	&+\theta \boldsymbol\Gamma\times(\boldsymbol\Omega\times\boldsymbol\Gamma)dt +\theta \left [ mg\boldsymbol \Pi\times (\boldsymbol\chi\times \boldsymbol \Gamma) - \boldsymbol\Pi\times (\boldsymbol\Pi\times \boldsymbol\Omega)\right]dt= 0  \,,\\
	d\boldsymbol \Gamma &+ (\boldsymbol \Omega dt + \sum_i\boldsymbol{\sigma}_i\circ dW^i_t)\times\boldsymbol \Gamma + \theta \left [ mg\boldsymbol\Gamma\times (\boldsymbol \chi\times \boldsymbol\Gamma) - \boldsymbol \Gamma\times(\boldsymbol\Pi\times \boldsymbol\Omega)\right ]dt = 0 \,.
	\end{split}
\end{align}
Notice that the two Casimirs which define the coadjoint orbits are preserved by both the noise and the dissipation, as expected.  
Also recall the form of the deterministic energy decay 
\begin{align}
	\frac{dh}{dt} 
	&=
	-\theta 
	 \left\|  \boldsymbol\Omega \times \boldsymbol{\Gamma} \right\|^2
	 -\,\theta 
	  \left\|   \boldsymbol\Omega\times \boldsymbol{\Pi}  + mg \boldsymbol\chi \times \boldsymbol{\Gamma} \right\|^2 \,,
   \label{Enorm_decay}
\end{align}
which was used earlier to prove the existence of the random attractor after a nonlinear change of variables. 
The other Casimir $\|\boldsymbol \Gamma\|^2$ can also be used to derive dissipative equations, but energy dissipation will be slower, as only the first term in \eqref{Enorm_decay}  and the first decay term of the $d\boldsymbol \Pi$ are left.
The equilibrium solution of the purely dissipative system are found by setting the right hand side of \eqref{Enorm_decay} to $0$ and are always of the form $\boldsymbol \Omega= \boldsymbol \Gamma= \boldsymbol \chi$ if $\boldsymbol \chi$ is an eigenvalue of $\mathbb I$. 
If not, the equilibrium is aligned to another direction that we will not compute here, as we will stick to the simple generic case. 

We can compute the lower bound for the value of the sum of the Lyapunov exponents using Theorem \ref{LEs-SP} to find
\begin{align}
	 \sum_i \lambda_i \geq -6\sigma^2  -\theta  c^2 \mathrm{Tr}(\mathbb I^{-1})\,,
\end{align}
which is always negative. 
We will not study here the parameter space of the top Lyapunov exponent as done for the rigid body, but only display two instances of an attractor of the heavy top in figure \ref{fig:HT-RA}.
\footnote{See \url{http://wwwf.imperial.ac.uk/~aa10213/} for a video of this random attractor.} 
The formation of the attractors seems not to be of horseshoe type, as occurs for the rigid body. 
This may be explained by the higher dimensionality of the coadjoint orbit (dimension $4$) on which the attractor is supported. 
\begin{figure}[htpb]
	\centering
	\subfigure{\includegraphics[scale=0.53]{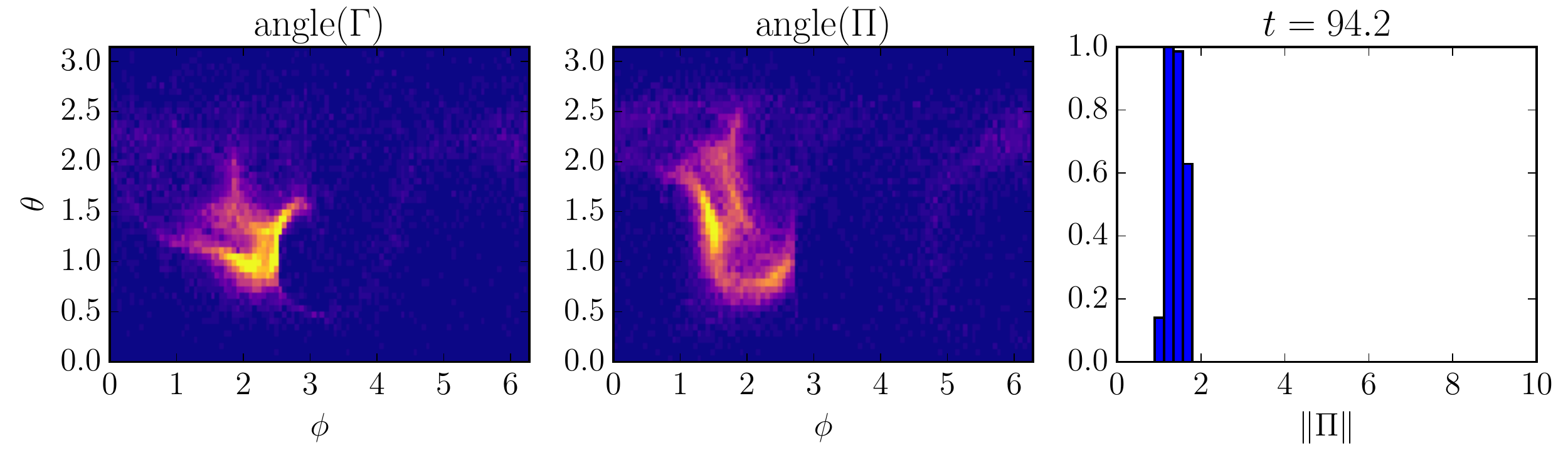}}
	\caption{ {\small  We display here three projections of the attractor on the coadjoint orbit of the heavy top at two different times. 
		The left panel is a projection on the sphere $\|\Gamma\|^2=1$, the second panel on the sphere $\|\Pi\|^2= 1$ and the third panel is the amplitude of the momentum. 
		We used $\theta= 0.2$, $\sigma=0.1$, $\mathbb I= \mathrm{diag}(1,2,3)$, $g=1$ and $20,000$ realisations of the stochastic heavy top, with initial conditions uniformly distributed on a subset of the coadjoint orbit defined by $\|\Pi\|^2= 1$. 
	}}
	\label{fig:HT-RA}
\end{figure}
\section{Two other examples}\label{others}
This section briefly sketches two other stochastic symmetry-reduced examples of the present theory which follow immediately from the examples of the $SO(3)$ rigid body and the heavy top, treated in the previous sections. These are the $SO(4)$ rigid body and the spring pendulum. 


\subsection{The $SO(4)$ rigid body}
For a complete study of the rigid body motion on $SO(4)$ we refer to \cite{birtea2012stability} and references therein. 
We use the generic elements 
\begin{align*}
	X =  
	\begin{pmatrix}
		0 & x_1 & x_2 & x_3 \\
		-x_1 & 0 & x_4 & -x_5 \\
		-x_2 & -x_4 & 0 & x_6\\
		-x_3 & x_5 & -x_6 & 0 ,
	\end{pmatrix}
\end{align*}
or $X= (X_1,X_2)\in \mathbb R^6$. 

In terms of vectors $(X_1,X_2)\in \mathbb R^6$ and $(X_1',X_2')\in \mathbb R^6$ we have
\begin{align*}
       [(X_1,X_2),(X_1',X_2')] = \left ( X_1\times X_1' + X_2\times X_2', X_1\times X_2' + X_2\times X_1'\right ).     
\end{align*}
The coadjoint action is the same, under the trace-pairing. 

The Casimir for $SO(4)$ are given by
\begin{align*}
	C_1&= \mathrm{Tr}(X^2)=\sum_i x_i^2= \|X_1\|^2+\|X_2\|^2 \,,\\
	C_2&= \sqrt{\mathrm{det}(X)}= x_1x_6 + x_2x_5 + x_3x_4= X_1\cdot X_2 \,.
\end{align*}
The first Casimir is a $4$-dimensional sphere and the second is the Pfaffian, or scalar product between two vectors.  

The momentum-velocity relation is $\Pi= J\Omega + \Omega J$ where $J= \mathrm{diag}(\lambda_1,\ldots,\lambda_6)$ and the Hamiltonian $H(\Pi)= \frac12 (\Pi_1\cdot \Omega_1 + \Pi_2\cdot \Omega_2) $.

We thus have the following stochastic $4$-dimensional rigid body equations
\begin{align}
	\begin{split}
	d (\Pi_1,\Pi_2)&= \left ( \Pi_1\times \Omega_1 + \Pi_2\times \Omega_2, \Pi_1\times \Omega_2 + \Pi_2\times \Omega_1\right ) dt \\
	&+  \sum_i \left ( \Pi_1\times \sigma_1^i + \Pi_2\times \sigma_2^i, \Pi_1\times \sigma_2^i + \Pi_2\times \sigma_1^i\right )\circ  dW_i\,,
	\end{split}
\end{align}
which preserve the coadjoint orbit. 

We now look at the selective decay term for the Casimir $C_2(\Pi)= \Pi_1\cdot \Pi_2$.
This term reads, upon using semi-simplicity, 
\begin{align*}
	SD & = \mathrm{ad}_{(\Pi_2,\Pi_1)}\mathrm{ad}_{(\Pi_2,\Pi_1)} (\Omega_1, \Omega_2)\\
         &=\left ( \Pi_2\times (\Pi_2\times \Omega_1 + \Pi_1\times \Omega_2) + \Pi_1\times (\Pi_2\times \Omega_2 + \Pi_1\times \Omega_1),\right . \\
         &=\left ( \Pi_2\times \Pi_2\times \Omega_1 +\Pi_2\times  \Pi_1\times \Omega_2 + \Pi_1\times \Pi_2\times \Omega_2 +  \Pi_1\times\Pi_1\times \Omega_1,\right . \\
	 &\left.  , \Pi_2\times \Pi_2\times \Omega_2 + \Pi_2\times \Pi_1\times \Omega_1 + \Pi_1\times  \Pi_2\times \Omega_1 + \Pi_1\times\Pi_1\times \Omega_2 \right ).     
\end{align*}
One can directly check that the first Casimir $C_1$ is also preserved by this flow. 

\begin{proposition}
	This stochastic dissipative $SO(4)$ free rigid body admits a random attractor.  
\end{proposition}

\begin{proof}
	This is a direct application of the theory developed in Section \ref{dissipation}. 
\end{proof}

The invariant distribution will be centred around the minimal energy position, associated to the direction of the maximal moment of inertia. 
We will not numerically investigate the random attractors for this system here. However, further theoretical studies are indeed possible and these would be interesting to discuss elsewhere, especially the integrable case, with a particular choice of the noise. 

\subsection{Spring pendulum}\label{SP}

From the heavy top equation one can derive the spherical pendulum by letting one of the components of the diagonal inertia tensor in body coordinates tend to zero, e.g., $I_3\to0$. 
This follows, because the spherical pendulum is infinitely thin and, hence, does not have any inertia for rotations around its axis. 
We shall choose $\mathbb I= \mathrm{diag}(I,I,\epsilon)$ in the heavy top equations and then take the limit $\epsilon\to 0 $ so that the dynamics on $\Pi_3$ vanishes. 
The  similarity of this system with the rigid body allows us to consider an extension of the spherical pendulum which is called the spring pendulum \cite{lynch2002swinging}. 
To include the dynamics of the length of the spring pendulum, we introduce a new variable $R(t)\in \mathbb R\setminus\{0\}$ and enforce its dynamical evolution in the variational principle by adding  $ P(\dot R- v )dt$ where $v$ denotes the velocity of the mass along the pendulum and $P$ denotes its associated momentum. 
The Lagrangian is then found to be 
\begin{align}
l(\boldsymbol \Omega,\boldsymbol \Gamma, R,v) = \frac{m}{2}R^2\boldsymbol \Omega\cdot \mathbb{I}\boldsymbol \Omega - mgR\boldsymbol \Gamma\cdot \boldsymbol \chi+ \frac12 m v^2  - \frac{k}{2}(R - 1)^2
\,,
\label{SP-lag}
\end{align}
where $\boldsymbol \chi$ represents the initial position of the pendulum which is taken to be $(0,0,1)$ in accordance with our choice of inertia tensor. 
In \eqref{SP-lag}, we denote the spring constant by $k$ and the mass of the pendulum bob by $m$. 

We shall assume a general linear stochastic potential of the form,
\begin{align}
\Phi(\boldsymbol \Pi,\boldsymbol \Gamma,R,P):= \boldsymbol \sigma\cdot \boldsymbol \Pi 
+\boldsymbol \eta\cdot \boldsymbol \Gamma+ \alpha R + \beta P
\,,
\end{align}
for constant vectors $\boldsymbol \sigma,\boldsymbol \eta,$ and constant scalars $\alpha,\beta$.
Consequently, the stochastic spring pendulum equations are given by
\begin{align}
	\begin{split}
	d\boldsymbol \Pi &= \boldsymbol \Pi\times \boldsymbol \Omega dt 
	+ mgR \boldsymbol \Gamma \times  \boldsymbol \chi dt 
	+ \boldsymbol \Pi\times \boldsymbol \sigma_i\circ dW_t^i 
	+ \boldsymbol \Gamma\times \boldsymbol{\eta}_i\circ dW_t^i
	\,,\\
	d\boldsymbol \Gamma &= \boldsymbol \Gamma\times \boldsymbol \Omega dt 
	+ \boldsymbol \Gamma \times \boldsymbol \sigma_i\circ dW_t^i
	\,,\\
	dR&= \frac{P}{m |\boldsymbol \chi |^2}dt + \beta  dW_t
	\,,\\
	dP&= -m g \boldsymbol \Gamma\cdot \boldsymbol \chi dt - k(R-1) |\boldsymbol \chi |^2 dt +\frac{1}{mR^3} \boldsymbol \Pi\cdot \mathbb I^{-1} \boldsymbol \Pi- \alpha dW_t
	\,.
	\end{split}
	\label{SP-eq}
\end{align}
The analysis above is valid, provided $\epsilon>0$ in the inertia tensor. In the limit $\epsilon\to 0 $, we may set $\boldsymbol \Omega_3=0$ and thereby recover the stochastic elastic spherical pendulum equations.

The equation set in \eqref{SP-eq} consists of two parts: the stochastic heavy top equations, coupled to a pair of stochastic canonical  Hamilton equations for the $(R,P)$ variables. 
The coupling between the two subsets of equations occurs through the dependence on $R$ together with $\boldsymbol \Omega$ and $\boldsymbol \Gamma$  in the Lagrangian \eqref{SP-lag}. 

The Fokker-Planck equation is now easily derived and it reads
\begin{align}
\begin{split}
	\frac{d}{dt}\mathbb P&= \{H,\mathbb P\}_{HT}+ \{H,\mathbb P\}_{\mathrm{can}} + \frac12 \{\Phi,\{\Phi,\mathbb P\}_{HT} \}_{HT}  \\
	&\hspace{3cm}+   \{\Phi,\{\Phi,\mathbb P\}_{HT} \}_{\mathrm can}  +   \frac12 \{\Phi,\{\Phi,\mathbb P\}_{\mathrm can } \}_{\mathrm can} \,,
\end{split}
\end{align}
where $\{\,\cdot \,,\, \cdot\,\}_{\mathrm{can}}$ is the canonical Poisson bracket with respect to the $(R,P)$ variables. 
The coupling between the elastic and pendulum motions is too complicated to extract any information from the Fokker-Planck equation. 
Indeed, inspection of the motion on $(R,P)$ shows that the advection equation for $(R,P)$ depends on the other variables. This inextricable complex dependence precludes finding the limiting distribution explicitly, despite the simple Laplacian form of the diffusion operator. 

As pointed out by \cite{lynch2002swinging}, the deterministic elastic spherical pendulum system is a toy model for the lowest modes of atmosphere dynamics. 
For this application, the motion of the spring oscillations encoded in $R$ is considerably faster than the pendulum motion and smaller in amplitude. Averaging the deterministic Lagrangian over the relatively rapid oscillations of the spring yields a nonlinear resonance between the modes of a type which also appears in the atmosphere. 
The noise can be included in either of the two types of dynamics and each will influence the other through the nonlinear coupling.
Also, for small oscillations around the equilibrium, the deterministic nonlinear coupling produces star shaped orbits \cite{holm2002spring,lynch2002swinging}, which can be perturbed, or even entirely destroyed, by the introduction of the noise, depending on its amplitude.

\section{Conclusion and open problems}

 Before stating some open problems arising from this work, we will briefly summarise it. 
In the first section we reviewed and developed the new machinery of stochastic geometric mechanics, in the context of finite dimensional systems which admit a group of symmetry of semi-simple type. 
The main results emerged from the introduction of a particular type of noise that preserves the coadjoint motion of the deterministic equations. 
The associated Fokker-Planck equation was found to possess interesting geometrical properties, related to the Lie-Poisson formulation of the equation of motion. The Lie-Poisson formulation was used to derive its invariant solution which is constant on the level set of the Casimirs prescribed by the initial conditions. 
The second section was devoted to the introduction of dissipation with the double bracket term, for which the coadjoint orbits are still preserved by the flow of the equation. 
This particular combination of multiplicative noise and nonlinear dissipation on coadjoint orbits yields non constant invariant measures, often referred to as Gibbs measures.
This type of invariant measure makes an interesting connection to statistical physics, and naturally provides us with a notion of temperature for these systems.  
The second outcome of the noise-dissipation interaction is the existence of so-called random attractors, which are mathematical objects deeply connected to the theory of random dynamical systems. 
We demonstrated, by adapting the standard tools from the random dynamical system theory, that such objects do exist in the mechanical systems studied here. 
Furthermore, we gave conditions on the dissipation and noise amplitudes for the existence of non-singular such attractors which will in turn support an SRB measure. 
The next two sections were devoted to the application of this theory to the standard examples in geometric mechanics, which are the free rigid body and the heavy top for the semi-direct product extension, also developed here.
We studied these explicit stochastic processes in detail, and in particular with illustrative numerical simulations.  
The final section touched upon other related examples such as a spring pendulum which can be viewed as an extension of the heavy top with a direct product structure, and a higher dimensional rigid body, written on $SO(4)$. 

We now end by listing several open problems which have been formulated during the course of this work. 
\begin{itemize}
	\item Some of the results presented here relied on the assumption of compactness of the coadjoint orbits; for example, in estimating the sum of the Lyapunov exponents, and for the study of the invariant distributions. 
		This assumption is probably unnecessary, but properly addressing its removal would require more advanced mathematical tools than we have used here. 
	\item We were only able to obtain a numerical demonstration that the top Lyapunov exponent is positive. This numerical demonstration could be made considerably more refined, and possibly analytical results could also be derived in future studies.
	\item We only touched upon the analysis of random attractors via numerical simulations, but much more may be said about these objects by, for example, studying their Lyapunov exponents in more detail, and studying the underlying dynamical process of their formation. 
		This is motivated by the fact that even in the two simple illustrative examples of random attractors treated here, we observed two rather different solution behaviours.
	\item Although we restricted ourselves to semisimple Lie algebras, we were able to write most of the equations for more general Lie algebras. 
		Hence, the present line of reasoning should be valid for other similar systems by modifying the proofs accordingly. 
		Examples of such systems would include semi-direct products with arbitrary advected quantities, solvable or nilpotent Lie algebras, the Toda lattice and possibly infinite dimensional Lie groups such as the diffeomorphism group.  
\end{itemize}

\subsection*{Acknowledgements}
{\small
We are grateful to many people for fruitful and encouraging discussions, including S. Albeverio, J.-M. Bismut, N. Bou-Rabee, M. D. Chekroun, G. Chirikjian, D. O. Crisan, A. B. Cruzeiro, J. Eldering, M. Engel, N. Grandchamps,  P. Lynch, J. Newman, J.-P. Ortega, G. Pavliotis, V. Putkaradze, T. Ratiu and C. Tronci.
The simulations were run with the Imperial College High Performance Computing Service.
We also acknowledge the Bernoulli Centre for Advanced Studies at EPFL where parts of this work were elaborated.  
AA acknowledges partial support from an Imperial College London Roth Award and AC from a CAPES Research Award BEX 11784-13-0. 
All the authors are also supported by the European Research Council Advanced Grant 267382 FCCA held by DH.
}

\bibliographystyle{alpha}
\bibliography{sgm-2015}

\appendix

\section{Numerical scheme}
\label{numerics-lyap}

Here we briefly discuss our numerical scheme for the integration of the stochastic rigid body equations. 
The heavy top being a direct extension of this scheme, we will not treat it here.

We use a split step scheme where the deterministic part of the equation is integrated with the python solver \emph{odeint} of Scipy that used the classic scheme \emph{lsoda}. 
The stochastic term is integrated via the exact solution of $d\boldsymbol \Pi = \boldsymbol \Pi \times \circ \mathbf{dW}$ where $\mathbf{dW}= \sum_i \boldsymbol \sigma_i dW_t^i$, as 
\begin{align}
	\boldsymbol \Pi(t+dt) = e^{\widehat{ \mathbf{dW}}} \boldsymbol \Pi(t),
\end{align}
where $\widehat{\mathbf{dW}}$ is corresponding anti-symmetric matrix, or $\mathfrak {so}(3)$ Lie algebra element, and the exponential is the matrix exponential.  

The computation of the top Lyapunov exponent is done with this scheme extended to also compute the evolution of the linearisation, $\delta \Pi$. 
The form of the equation for $\delta \Pi$ is similar to the rigid body equation, and the previous split step algorithm is implemented in the same way. 
Despite having a good preservation of the coadjoint orbit for the original stochastic process, this scheme does not accurately preserve the restriction of the linearisation to be tangent to the coadjoint orbit.  
The scheme \emph{lsoda} being implicit, we need an interpolation between $\Pi(t)$ and $\Pi(t+dt)$, which is achieved by a spherical linear interpolation, or `slerp'. 
Even with this interpolation, the condition $\delta \Pi\cdot \Pi=0$ is not precise enough and leads to an incorrect estimate of the norm $\|\delta \Pi\|$ in the computation of the Lyapunov exponent. 
We thus project out the part of $\delta \Pi$ aligned with $\Pi$ at each timestep such that $\delta \Pi$ remains in the tangent space to the coadjoint orbit. 
This projection gives coherent results, but it affects the convergence of the different estimations of the top Lyapunov exponent. Namely, for each different initial condition the corresponding top Lyapunov exponent will take more time to converge to its true value. 
We overcome this effect by averaging over several realisations of the top exponent computation, which is $50$ in our case.  


\end{document}